\documentclass[12pt]{amsart}

\usepackage[utf8]{inputenc}
\usepackage[T1,T2A]{fontenc}
\usepackage[english]{babel}
\usepackage{amsfonts}
\usepackage{amsbsy}
\usepackage{amssymb}
\usepackage{fixltx2e,graphicx}
\usepackage{mathrsfs}
\usepackage{hyperref}
\usepackage{longtable}
\usepackage{enumitem}
\usepackage{comment}

\catcode`~=11 
\newcommand{\urltilde}{\kern -.15em\lower .7ex\hbox{~}\kern .04em}
\catcode`~=13 

\makeatletter

\renewcommand{\@makecaption}[2]{
\vspace{\abovecaptionskip}%
\sbox{\@tempboxa}{#1. #2}%
\global\@minipagefalse \hbox to \hsize {{\scshape \hfil #1.
#2\hfil}} \vspace{\belowcaptionskip}}

\makeatother

\newcommand{\Hom}{\operatorname{Hom}}
\newcommand{\rk}{\operatorname{rk}}

\newcommand{\Ker}{\operatorname{Ker}}
\newcommand{\Lie}{\operatorname{Lie}}
\newcommand{\Der}{\operatorname{Der}}

\newcommand{\GL}{\operatorname{GL}}
\newcommand{\SL}{\operatorname{SL}}

\newcommand{\Aut}{\operatorname{Aut}}

\newcommand{\ord}{\operatorname{ord}}

\newcommand{\Spec}{\operatorname{Spec}}
\renewcommand{\Im}{\operatorname{Im}}

\newcommand{\ZZ}{\mathbb Z}
\newcommand{\QQ}{\mathbb Q}

\newcommand{\PP}{\mathbb P}
\newcommand{\KK}{\mathbb K}
\newcommand{\GG}{\mathbb G}
\renewcommand{\AA}{\mathbb A}

\newtheorem{theorem}{Theorem}

\newtheorem{proposition}[theorem]{Proposition}
\newtheorem{lemma}[theorem]{Lemma}
\newtheorem{corollary}[theorem]{Corollary}

\newtheorem*{question*}{Question}

\theoremstyle{definition}

\newtheorem{definition}[theorem]{Definition}

\theoremstyle{remark}

\newtheorem{remark}[theorem]{Remark}

\makeatletter

\@addtoreset{theorem}{section}

\makeatother

\numberwithin{equation}{section}

\numberwithin{equation}{section}

\newcounter{num}[table]
\newcommand{\no}{\refstepcounter{num}\arabic{num}}

\makeatletter

\@addtoreset{theorem}{section}
\@namedef{subjclassname@2020}{%
  \textup{2020} Mathematics Subject Classification}

\makeatother

\begin{document}

\date{}
\title[Root subgroups on horospherical varieties]{Root subgroups on horospherical varieties}
%
\author{Roman Avdeev}
\address{%
{\bf Roman Avdeev} \newline\indent HSE University, Moscow, Russia}
\email{suselr@yandex.ru}

\author{Vladimir Zhgoon}
\address{%
{\bf Vladimir Zhgoon} \newline\indent Moscow Institute of Physics and Technology. Dolgoprudnyi, Moscow region, Russia \newline\indent Scientific Research Institute for System Analysis of the Russian Academy of Sciences, Moscow, Russia \newline\indent HSE University, Moscow, Russia}
\email{zhgoon@mail.ru}



\begin{abstract}
Given a connected reductive algebraic group $G$ and a spherical $G$-variety~$X$, a $B$-root subgroup on $X$ is a one-parameter additive group of automorphisms of $X$ normalized by a Borel subgroup $B \subset G$.
We obtain a complete description of all $B$-root subgroups on a certain open subset of~$X$.
When $X$ is horospherical, we extend the construction of standard $B$-root subgroups introduced earlier by Arzhantsev and Avdeev for affine~$X$ and obtain a complete description of all standard $B$-root subgroups, which naturally generalizes the well-known description of root subgroups on toric varieties.
As an application, for horospherical $X$ that is either complete or contains a unique closed $G$-orbit, we determine all $G$-stable prime divisors in~$X$ that can be connected with the open $G$-orbit via the action of a suitable $B$-root subgroup.
For horospherical~$X$, we also find sufficient conditions for the existence of $B$-root subgroups on~$X$ that preserve the open $B$-orbit in~$X$.
Finally, when $G$ is of semisimple rank~$1$ and $X$ is horospherical and complete, we determine all $B$-root subgroups on~$X$, which enables us to describe the Lie algebra of the connected automorphism group of~$X$.
\end{abstract}

\maketitle


\section*{Introduction}
\label{sec0}

Let $X$ be an irreducible algebraic variety defined over an algebraically closed field $\KK$ of characteristic zero.
Every nontrivial action of the additive group $\GG_a = (\KK,+)$ on $X$ determines a subgroup $H$ in the automorphism group $\Aut(X)$, called a \textit{$\GG_a$-subgroup} on~$X$.
If $X$ is equipped with a regular action of an algebraic group~$F$ and $H$ is normalized by~$F$, then $H$ is called an \textit{$F$-root subgroup}.
In this case, $F$ acts on the Lie algebra of $H$ via a character, called the \textit{weight} of~$H$.

The most known case of the above situation appears in the theory of toric varieties.
Recall that a normal irreducible algebraic variety $X$ is called \textit{toric} if it is equipped with an action of an algebraic torus $T$ such that $X$ has an open $T$-orbit.
Toric varieties admit a complete combinatorial description in terms of objects of convex geometry called fans; see \cite{CLS, Fu, Oda}.
Moreover, there is a complete description of $T$-root subgroups on any given toric $T$-variety~$X$; see \cite{Dem, Oda}.
It turns out that every $T$-root subgroup on $X$ is uniquely determined by its weight and all weights appearing in this way form the set of so-called \textit{Demazure roots} of the associated fan.

A natural generalization of toric varieties in the setting of actions of arbitrary connected reductive groups is given by spherical varieties.
By definition, a normal irreducible algebraic variety $X$ is called \textit{spherical} if it is equipped with an action of a connected reductive algebraic group $G$ such that $X$ has an open orbit for the induced action of a Borel subgroup~$B \subset G$.
It was proposed in~\cite{AA} that a proper generalization of $T$-root subgroups for spherical varieties is given by $B$-root subgroups.
The same paper~\cite{AA} also initiated a systematic study of $B$-root subgroups on affine spherical varieties, which was continued in~\cite{AZ}.
An important result proved in~\cite{AZ} states that, given an arbitrary affine spherical $G$-variety~$X$ with open $G$-orbit~$O$, every $G$-stable prime divisor in~$X$ can be connected with~$O$ via the action of a suitable $B$-root subgroup.

In this paper, we study $B$-root subgroups on arbitrary (not necessarily affine) spherical $G$-varieties.
Our key ingredient is the local structure theorem (see~\cite{BLV, Kn94}), which describes the action of a certain parabolic subgroup $P \subset G$ on a distinguished open subset $X_0 \subset X$.
Our first result provides a complete description of all $B$-root subgroups on~$X_0$, which generalizes a similar result from~\cite{AZ} for the case of affine~$X$; see \S\,\ref{subsec_local_descr}.

Given a $B$-root subgroup $H$ on the open subset $X_0 \subset X$, an important problem is to determine whether it extends to a $B$-root subgroup on the whole~$X$.
Solving this problem can be divided into two stages: first, one needs to determine whether the vector field on~$X_0$ corresponding to~$H$ extends to~$X$, and second, one needs to check whether the resulting vector field on~$X$ can be integrated to a $B$-root subgroup on~$X$.
We remark that the latter integrability condition holds automatically if $X$ is either affine or complete.
In this paper, we focus on these issues in the case where $X$ is horospherical, that is, the stabilizer of a point in the open $G$-orbit $O \subset X$ contains a maximal unipotent subgroup of~$G$.

An important construction introduced in~\cite{AA} for affine horospherical varieties is that of standard $B$-root subgroups.
We extend this notion to arbitrary horospherical~$X$ and obtain a complete description of all standard $B$-root subgroups on~$X$, which remarkably generalizes the description of $T$-root subgroups on toric $T$-varieties.
As an application, if $X$ is horospherical and either complete or contains a unique closed $G$-orbit, we determine all $G$-stable prime divisors in $X$ that can be connected with $O$ via the action of a $B$-root subgroup.

A $B$-root subgroup on a spherical $G$-variety $X$ is said to be \textit{vertical} if it preserves the open $B$-orbit $O \subset X$ and \textit{horizontal} otherwise.
In this terminology, all $T$-root subgroups on toric $T$-varieties are horizontal.
All standard $B$-root subgroups on horospherical $G$-varieties are also horizontal.

One more contribution of this paper for horospherical $X$ consists in sufficient conditions under which a vector field corresponding to a vertical $B$-root subgroup on~$X_0$ extends to the whole~$X$.
The gap between these conditions and natural necessary conditions turns out to be rather small and observable, and it would be interesting to completely eliminate it.

We also study in detail the particular case where $G$ is of semisimple rank~$1$ (that is, up to a finite covering, $G$ is isomorphic to a direct product of $\SL_2$ and a torus) and $X$ is horospherical.
In this case, we obtain a complete description of all vertical $B$-root subgroups on~$X$.
If $X$ is complete, then, combining this with the description of all standard $B$-root subgroups, we obtain a complete description of all (both vertical and horizontal) $B$-root subgroups on~$X$.
Moreover, for complete $X$ we compute all commutation relations between simple $G$-modules of vector fields on~$X$ generated by vector fields corresponding to $B$-root subgroups.
This provides a description of the Lie algebra of the connected automorphism group of~$X$.

This paper is organized as follows.
In \S\,\ref{sec_prelim} we collect all basic notions and results needed in this paper.
In \S\,\ref{sec_DR_and_RS} we discuss toric varieties, Demazure roots, and the description of all $T$-root subgroups on arbitrary toric $T$-varieties.
Although the latter description is well known to specialists, we are not aware of any source providing a presentation of it in full generality, therefore we made a self-contained exposition with complete proofs.
Several ideas and techniques involved in this description are then used later in this paper for general spherical varieties.
In \S\,\ref{sec_gen_spher} we collect all the needed notions and facts on spherical and horospherical varieties.
In \S\,\ref{sec_basic_prop} we study basic properties of $B$-root subgroups on an arbitrary spherical $G$-variety $X$ and obtain out description of all $B$-root subgroups on the open subset $X_0 \subset X$ mentioned above.
In \S\,\ref{sec_hor_stand} we discuss standard $B$-root subgroups on horospherical $G$-varieties and obtain a complete description of them.
In \S\,\ref{sec_hor_vert} we obtain our sufficient conditions for vector fields of vertical $B$-root subgroups to extend from~$X_0$ to~$X$.
In \S\,\ref{sec_ssrank1} we work out the case where $G$ is of semisimple rank~$1$ for horospherical~$X$.
In Appendix~\ref{sec_matrix_exponentials} we present several identities for matrix exponentials needed in \S\,\ref{sec_hor_vert}.

\subsection*{Acknowledgements}

The work of Roman Avdeev was supported by the Russian Science Foundation (grant no.~22-41-02019).
The work of Vladimir Zhgoon was supported by Moscow Institute of Physics and Technology under the Priority 2030 Strategic Academic Leadership Program, by the state assignment for basic scientific research (project no. FNEF-2024-0001), and by the HSE University Basic Research Program.
The authors thank Ivan Arzhantsev and Dmitry Timashev for useful discussions.


\section{Preliminaries}
\label{sec_prelim}

\subsection{Some notation and conventions}

Throughout this paper, we work over an algebraically closed field $\KK$ of characteristic zero.
The notation $\KK^\times$ stands for the multiplicative group $(\KK \setminus \lbrace 0 \rbrace, \times)$.
The additive group $(\KK, +)$ is denoted by~$\GG_a$.
The character group of an algebraic group $G$ is denoted by $\mathfrak X(G)$ and used in additive notation.
All topological terms relate to the Zariski topology.
Given an irreducible algebraic variety~$X$, the group of its regular automorphisms is denoted by $\Aut(X)$ and the notation $\KK[X]$ (resp. $\KK(X))$ stands for the algebra of regular functions (resp. field of rational functions) on~$X$.

If an algebraic group $G$ acts on an algebraic variety~$X$, then the induced action of $G$ on~$\KK[X]$ and on~$\KK(X)$ is given by the formula $(gf)(x) = f(g^{-1}x)$ for all $g \in G$, $f \in \KK[X]$, and $x \in X$.

\subsection{\texorpdfstring{$\GG_a$}{G\_a}-subgroups, vector fields, and LND's}

Let $X$ be an irreducible algebraic variety.
Every nontrivial $\GG_a$-action on $X$ induces an algebraic subgroup in $\Aut(X)$, called a \textit{$\GG_a$-subgroup}.
Every $\GG_a$-subgroup on~$X$ induces a vector field on~$X$ (defined uniquely up to proportionality).
Every vector field on~$X$ obtained in this way will be called \textit{$\GG_a$-integrable}.
Since vector fields are sections of the tangent sheaf, \cite[Lemma~3.9]{Br13} implies the following important extension result: if $X$ is normal and $\xi$ is a vector field on an open subset~$X_0 \subset X$ whose complement $X \setminus X_0$ has codimension $\ge2$ in~$X$, then $\xi$ extends to a vector field on the whole~$X$.

Given a $\GG_a$-subgroup $H$ on~$X$, every nonzero element of the Lie algebra $\Lie(H)$ defines a locally nilpotent derivation (LND for short) $\partial$ on~$\KK[X]$.
If $X$ is quasi-affine then $H$ is recovered from $\partial$ by taking the exponent.

If $A = \KK[X]$ for an affine algebraic variety~$X$, for any LND $\partial$ on $A$ the map
\begin{equation} \label{eqn_phi_d}
\varphi_{\partial} \colon \GG_a\times A\rightarrow A, \quad (s,a) \mapsto \exp(s\partial)(a),
\end{equation}
defines a rational $\GG_a$-algebra structure on~$A$, hence induces a $\GG_a$-action on~$X$.
In fact, by~\cite[\S\,1.5]{Fr} any $\GG_a$-action on~$X$ arises this way, which yields

\begin{proposition} \label{prop_Ga-actions}
Given an affine variety~$X$, the map $\partial \mapsto \varphi_\partial$ induces a bijection between the nonzero LNDs on $\KK[X]$ modulo proportionality and the $\GG_a$-subgroups on~$X$.
\end{proposition}

\subsection{Root subgroups}

Now suppose that $X$ is equipped with a regular action of an algebraic group~$F$.
An \textit{$F$-root subgroup} on~$X$ is a $\GG_a$-subgroup on $X$ normalized by the action of~$F$.
Given an $F$-root subgroup $H$ on~$X$, $F$ acts on $\Lie(H)$ via multiplication by a character $\chi_H \in \mathfrak X(F)$, called the \textit{weight} of~$H$.

If $H$ is a $\GG_a$-subgroup on~$X$ and $\xi$ is the corresponding vector field on~$X$, then $H$ is an $F$-subgroup if and only if $\xi$ is $F$-semiinvariant with the same weight.

If $X$ is quasi-affine, $H$ is a $\GG_a$-subgroup on~$X$, and $\partial$ is an LND on $\KK[X]$ corresponding to~$H$ then $H$ is an $F$-subgroup if and only if $\partial$ is $F$-normalized with the same weight.

Let $X$ be arbitrary (not necessarily quasi-affine) and let $H$ be an $F$-root subgroup on~$X$.
Given a prime divisor $D \subset X$, we say that $F$ \textit{moves} $D$ (or $D$ is \textit{moved} by~$H$) if $HD \ne D$, that is, $D$ is $H$-unstable.
The following result was obtained in \cite[Prop.~2.6]{AA}.

\begin{proposition} \label{prop_moved_divisor}
Suppose that $F$ is connected and has an open orbit $O_F$ in~$X$.
If $O_F$ is not preserved by~$H$ then there is exactly one $F$-stable prime divisor in~$X$ moved by~$H$.
\end{proposition}

\subsection{Torus actions and gradings}
\label{ssec_torus_actions_and_gradings}

Let $T$ be an algebraic torus and let $Z$ be a normal irreducible $T$-variety.
The \textit{weight lattice}~$M(Z)$ (resp. \textit{weight monoid}~$\Gamma(Z)$) is the set of weights of all $T$-semiinvariant functions in~$\KK(Z)$ (resp. in~$\KK[Z])$.

In what follows we assume that $Z$ is affine.
Then the induced action of $T$ on~$\KK[Z]$ yields a grading
\begin{equation} \label{eqn_Tgrading}
\KK[Z] = \bigoplus \limits_{\lambda \in \mathfrak X(T)} \KK[Z]_\lambda,
\end{equation}
where $\KK[Z]_\lambda$ is the subspace of $T$-semi-invariant functions in~$\KK[Z]$ of weight~$\lambda$.
Observe that $\Gamma(Z) = \lbrace \lambda \in \mathfrak X(T) \mid \KK[Z]_\lambda \ne \lbrace 0 \rbrace \rbrace$.
Conversely, every grading of $\KK[Z]$ of the form~(\ref{eqn_Tgrading}) defines an action of $T$ on~$\KK[Z]$ and hence on~$Z$, so that there is a natural bijection between $T$-actions on~$Z$ and gradings by $\mathfrak X(T)$ on~$\KK[Z]$.

A derivation $\partial$ on~$\KK[Z]$ is said to be \textit{homogeneous} if for every $\lambda \in \Gamma$ there is $\lambda' \in \Gamma$ such that $\partial(\KK[Z]_\lambda) \subset \KK[Z]_{\lambda'}$.
It follows from the definition that for a homogeneous derivation $\partial$ on~$\KK[Z]$ there exists a unique weight $\mu \in M(Z)$ such that $\partial(\KK[Z]_\lambda) \subset \KK[Z]_{\lambda+\mu}$ for all $\lambda \in \Gamma(Z)$.
This $\mu$ is said to be the \textit{weight} of the homogeneous derivation~$\partial$.

It is easy to check that an LND $\partial$ on~$\KK[Z]$ is homogeneous of weight~$\mu$ if and only if $\partial$ is $T$-normalized of weight~$\mu$.

\subsection{Extension results for vector fields and group actions}

Let $K$ be a connected algebraic group and consider the Lie algebra $\mathfrak k = \Lie(K)$. 
The next result is extracted from \cite[\S\,1]{LV}; see also \cite[\S\,12.2]{Tim}.

\begin{proposition} \label{prop_LV_ext}
Let $U$ be an irreducible $K$-variety and consider the $\mathfrak k$-action on $U$ by vector fields.
Let $X$ be an irreducible variety containing $U$ as an open subset and suppose that the $\mathfrak k$-action on $U$ extends to~$X$.
Then there is another irreducible variety $\widehat X$ containing $X$ as an open subset such that the $K$-action on $U$ extends to a $K$-action on~$\widehat X$.
In particular, if $X$ is complete then $\widehat X = X$ and the $K$-action on $U$ extends to~$X$.
\end{proposition}

\begin{proposition} \label{prop_LV_ext2}
Under the hypotheses of Proposition~\textup{\ref{prop_LV_ext}} assume that $\mathfrak k$ acts trivially on $X \setminus U$.
Then the action of $K$ extends to $X$ in such a way that $K$ acts trivially on $X \setminus U$.
\end{proposition}

\subsection{Connected automorphism groups of complete rational varieties}
\label{ssec_aut_groups}

Let $X$ be a complete rational normal irreducible variety and let $A = \Aut(X)^0$ be the connected component of the identity of the automorphism group of~$X$.
It is known that $A$ is a linear algebraic group; let $\mathfrak a$ be its Lie algebra.

Now assume in addition that $X$ is a $G$-variety for a connected reductive algebraic group~$G$.
Without loss of generality we may assume that $G$ acts effectively on~$X$, so that there is an inclusion $G \subset A$.
In this paper we shall need the following structure result.

\begin{proposition} \label{prop_aut_group_dec}
There is a $G$-module decomposition
\begin{equation} \label{eqn_aut_group}
\mathfrak a = \mathfrak g \oplus \mathfrak c \oplus \bigoplus \limits_{i=1}^k \mathfrak a_i
\end{equation}
where $\mathfrak c$ is the Lie algebra of a subtorus $C \subset A$ centralizing~$G$ and each $\mathfrak a_i$ is a simple $G$-module whose highest weight vector with respect to~$B$ is a nilpotent element of~$\mathfrak a$.
\end{proposition}

\section{Demazure roots and root subgroups on toric varieties}
\label{sec_DR_and_RS}

\subsection{Cones, fans, and Demazure roots}
\label{subsec_Dem_roots}

Let $M$ be a lattice of finite rank and consider the dual lattice $N = \Hom_\ZZ(M, \ZZ)$ along with the natural pairing $\langle \cdot\,, \cdot \rangle \colon N \times M \to \ZZ$.
Consider also the rational vector spaces $M_\QQ = M \otimes_\ZZ \QQ$ and $N_\QQ = N \otimes_\ZZ \QQ$ and extend the pairing to a bilinear map $\langle \cdot\,, \cdot \rangle \colon N_\QQ \times M_\QQ \to \QQ$.

In what follows, by a cone in $N_\QQ$ (or in $M_\QQ$) we mean a finitely generated (or, equivalently, polyhedral) convex cone.

Let $\mathcal C$ be a cone in~$N_\QQ$.
It is said to be \textit{strictly convex} if $\mathcal C \cap (-\mathcal C) = \lbrace 0 \rbrace$, that is, $\mathcal C$~contains no nonzero subspaces of~$N_\QQ$.
The \textit{dimension} of $\mathcal C$ is that of its linear span.
The \textit{dual cone} of $\mathcal C$ is
\[
\mathcal C^{\vee} := \lbrace x \in M_{\QQ} \mid \langle v, x\rangle\ge 0 \ \text{for all} \ v \in \mathcal C \rbrace;
\]
this is a cone in~$M_\QQ$.
A \textit{face} of $\mathcal C$ is a subset $\mathcal C' \subseteq \mathcal C$ of the form
\[
\mathcal C' = \lbrace v \in \mathcal C \mid \langle v, x \rangle = 0 \rbrace
\]
for some $x \in \mathcal C^\vee$.
Every face of $\mathcal C$ is a cone itself.
Note that every cone is a face of its own.
A face of codimension one is called a \textit{facet}.
A face of dimension one of a strictly convex cone is called a \textit{ray}.

Let $\mathcal C \subset N_\QQ$ be a strictly convex cone.
Let $\mathcal C^1$ be the set of primitive elements $\rho$ of the lattice $N$ such that $\QQ_{\ge0}\rho$ is a ray of~$\mathcal C$.
Observe that every face of $\mathcal C$ is generated by a subset of~$\mathcal C^1$.

\begin{definition} \label{def_DR_cone}
An element $\mu \in M$ is said to be a \textit{Demazure root} of the cone $\mathcal C$ if there exists $\rho_\mu \in \mathcal C^1$ such that $\langle \rho_\mu, \mu \rangle = -1$ and $\langle \rho, \mu \rangle \ge 0$ for all $\rho \in \mathcal C^1 \setminus \lbrace \rho_\mu \rbrace$.
\end{definition}

Let $\mathfrak R(\mathcal C)$ denote the set of all Demazure roots of~$\mathcal C$.

\begin{lemma} \label{lemma_faces1}
Let $\mathcal C \subset N_\QQ$ be a strictly convex cone, $\mu \in \mathfrak R(\mathcal C)$, and $\rho_\mu \in \mathcal C^1$ the corresponding element.
\begin{enumerate}[label=\textup{(\alph*)},ref=\textup{\alph*}]
\item \label{lemma_faces1_a}
Suppose $\mathcal K$ is a face of $\mathcal C$ such that $\langle \mathcal K, \mu \rangle = 0$ and $\widetilde{\mathcal K}$ is the cone generated by $\mathcal K$ and~$\rho_\mu$.
Then $\widetilde{\mathcal K}$ is a face of~$\mathcal C$.

\item \label{lemma_faces1_b}
Suppose $\widetilde{\mathcal K}$ is a face of $\mathcal C$ such that $\langle \widetilde{\mathcal K}, \mu \rangle \le 0$ and $\rho_\mu \in \widetilde{\mathcal K}^1$.
Then the cone $\mathcal K = \lbrace v \in \widetilde{\mathcal K} \mid \langle v, \mu \rangle = 0 \rbrace$ is a face of~$\mathcal C$.
\end{enumerate}
\end{lemma}

\begin{proof}
(\ref{lemma_faces1_a})
Let $\nu \in M$ be such that $\langle \rho, \nu \rangle = 0$ for all $\rho \in \mathcal K^1$ and $\langle \rho, \nu \rangle > 0$ for all $\rho \in \mathcal C^1 \setminus \mathcal K^1$; put also $k =\langle \rho_\mu, \nu \rangle > 0$.
Then the element $\nu' = \nu +k\mu$ satisfies $\langle \widetilde{\mathcal K}, \nu' \rangle = 0$ and $\langle \rho, \nu' \rangle > 0$ for all $\rho \in \mathcal C^1 \setminus (\mathcal K^1 \cup \lbrace \rho_\mu \rbrace)$, so $\widetilde{\mathcal K}$ is indeed a face of~$\mathcal C$.

(\ref{lemma_faces1_b})
Let $\nu \in M$ be such that $\langle \rho, \nu \rangle = 0$ for all $\rho \in \widetilde{\mathcal K}^1$ and $\langle \rho, \nu \rangle > 0$ for all $\rho \in \mathcal C^1 \setminus \widetilde{\mathcal K}^1$.
Then, for a sufficiently small $c > 0$, the element $\nu' = \nu - c\mu$ satisfies $\langle \rho, \nu' \rangle = 0$ for all $\rho \in \widetilde{\mathcal K}^1 \setminus \lbrace \rho_\mu \rbrace$ and $\langle \rho, \nu' \rangle > 0$ for all $\rho \in \mathcal C^1 \setminus (\widetilde{\mathcal K}^1 \setminus \lbrace \rho_\mu \rbrace)$, so $\mathcal K$ is a face of~$\mathcal C$.
\end{proof}

A \textit{fan} in $N_\QQ$ is a finite collection $\mathfrak F$ of strictly convex cones in $N_\QQ$ satisfying the following conditions:

\begin{enumerate}[label=\textup{(F\arabic*)},ref=\textup{F\arabic*}]
\item \label{fan1}
if $\mathcal C \in \mathfrak F$, then each face of $\mathcal C$ also belongs to~$\mathfrak F$;

\item \label{fan2}
if $\mathcal C_1, \mathcal C_2 \in \mathfrak F$, then $\mathcal C_1 \cap \mathcal C_2$ is a face of both $\mathcal C_1$ and~$\mathcal C_2$.
\end{enumerate}

Let $\mathfrak F$ be a fan in $N_\QQ$ and
let $\mathfrak F^1$ be the set of primitive elements $\rho$ of the lattice~$N$ such that $\QQ_{\ge0}\rho \in \mathfrak F$.
Note that $\mathfrak F^1 = \bigcup \limits_{\mathcal C \in \mathfrak F} \mathcal C^1$.

\begin{definition} \label{def_DR_fan}
An element $\mu \in M$ is said to be a \textit{Demazure root} of the fan~$\mathfrak F$ if the following conditions are fulfilled:

\begin{enumerate}[label=\textup{(DR\arabic*)},ref=\textup{DR\arabic*}]
\item \label{DR1}
there exists $\rho_\mu \in \mathfrak F^1$ with $\langle \rho_\mu, \mu \rangle = -1$;

\item \label{DR2}
$\langle \rho, \mu \rangle \ge 0$ for all $\rho \in \mathfrak F^1 \setminus \lbrace \rho_\mu \rbrace$;

\item \label{DR3}
if a cone $\mathcal K \in \mathfrak F$ satisfies $\langle \mathcal K, \mu \rangle = 0$, then the cone generated by $\mathcal K$ and $\rho_\mu$ belongs to~$\mathfrak F$.
\end{enumerate}
\end{definition}

Let $\mathfrak R(\mathfrak F)$ denote the set of all Demazure roots of the fan~$\mathfrak F$.
In what follows, for every $\mu \in \mathfrak R(\mathfrak F)$ we fix the notation $\rho_\mu$ for the element in $\mathfrak F^1$ satisfying $\langle \rho_\mu, \mu \rangle = -1$.

An important example of a fan is given by the collection $\mathfrak F$ of faces of a single strictly convex cone~$\mathcal C \subset N_\QQ$.
Clearly, in this situation one has $\mathcal C^1 = \mathfrak F^1$.
By Lemma~\ref{lemma_faces1}(\ref{lemma_faces1_a}), every $\mu \in \mathfrak R(\mathcal C)$ automatically satisfies (\ref{DR3}), and so $\mathfrak R(\mathcal C) = \mathfrak R(\mathfrak F)$.

\begin{lemma} \label{lemma_add_cone}
Let $\mathfrak F$ be a fan in $N_\QQ$ and let $\mu \in M$ satisfy \textup(\ref{DR1}\textup) and \textup(\ref{DR2}\textup).
Suppose a cone $\mathcal K \in \mathfrak F$ satisfies $\langle \mathcal K, \mu \rangle = 0$ and let $\widetilde{\mathcal K}$ be the cone generated by $\mathcal K$ and~$\rho_\mu$.
Then the collection $\mathfrak F \cup \lbrace\text{all faces of} \ \widetilde{\mathcal K} \rbrace$ is a fan in $N_\QQ$.
\end{lemma}

\begin{proof}
Let $\mathcal C \in \mathfrak F$ be an arbitrary cone and let $\mathcal E$ be an arbitrary face of~$\widetilde{\mathcal K}$.
We need to show that $\mathcal C \cap \mathcal E$ is a face of both $\mathcal C$ and~$\mathcal E$.
If $\rho_\mu \notin \mathcal E^1$, then $\mathcal E$ is a face of~$\mathcal K$, and the assertion is clear.
If $\rho_\mu \in \mathcal E^1$, then, by Lemma~\ref{lemma_faces1}(\ref{lemma_faces1_b}), $\mathcal E$ is generated by a face of $\mathcal K$ and~$\rho_\mu$.
In this case, the subsequent argument is the same as for $\mathcal E = \widetilde{\mathcal K}$, so it suffices to consider the case $\mathcal E = \widetilde{\mathcal K}$.
It remains to show that $\mathcal C \cap \widetilde{\mathcal K}$ is a face of both $\mathcal C$ and~$\widetilde{\mathcal K}$.
If $\langle \mathcal C, \mu \rangle \ge 0$, then $\mathcal C \cap \widetilde{\mathcal K} = \mathcal C \cap \mathcal K$, therefore $\mathcal C \cap \widetilde{\mathcal K}$ is a face of both $\mathcal C$ and $\mathcal K$ (and hence of $\widetilde{\mathcal K}$).
Otherwise $\rho_\mu \in \mathcal C^1$.
Consider the cone $\mathcal P = \mathcal C \cap \mathcal K$; it is a face of both $\mathcal C$ and $\mathcal K$.
Let $\widetilde{\mathcal P}$ be the cone generated by $\mathcal P$ and $\rho_\mu$.
We now show that $\mathcal C \cap \widetilde{\mathcal K} = \widetilde{\mathcal P}$.
The inclusion $\supset$ is clear, so we prove the reverse one.
Take any $v \in \mathcal C \cap \widetilde{\mathcal K}$.
Since $v \in \mathcal C$, one has $v = a\rho_\mu + w$ for some $a \in \QQ_{\ge0}$ and $w \in \QQ_{\ge0}(\mathcal C^1 \setminus \lbrace \rho_\mu \rbrace)$.
Similarly, $v \in \widetilde{\mathcal K}$ implies $v = a'\rho_\mu + w'$ for some $a' \in \QQ_{\ge0}$ and $w' \in \mathcal K$.
The relation $\langle a\rho_\mu + w, \mu \rangle = \langle a'\rho_\mu + w',\mu \rangle$ yields $a-a' = \langle w,\mu \rangle \ge 0$, hence $(a-a')\rho_\mu + w = w'$.
Since $(a-a')\rho_\mu + w \in \mathcal C$ and $w' \in \mathcal K$, it follows that $w' \in \mathcal C \cap \mathcal K = \mathcal P$, thus $v = a'\rho_\mu + w' \in \widetilde{\mathcal P}$.
We have shown that $\mathcal C \cap \widetilde{\mathcal K} = \widetilde{\mathcal P}$.
By Lemma~\ref{lemma_faces1}(\ref{lemma_faces1_a}), $\widetilde{\mathcal P}$ is a face of both $\mathcal C$ and $\widetilde{\mathcal K}$.
\end{proof}

\begin{corollary}
Let $\mathfrak F$ be a fan in $N_\QQ$ and let $\mu \in M$ satisfy \textup(\ref{DR1}\textup) and \textup(\ref{DR2}\textup).
Then there is a fan $\widetilde{\mathfrak F}$ in $N_\QQ$ containing $\mathfrak F$ such that $\mu \in \mathfrak R(\widetilde{\mathfrak F})$.
\end{corollary}

\begin{corollary} \label{crl_convex_DR3}
Let $\mathfrak F$ be a fan in~$N_\QQ$ and suppose that the set $\bigcup \limits_{\mathcal C \in \mathfrak F} \mathcal C$ is convex.
Then every $\mu \in M$ satisfying \textup(\ref{DR1}\textup) and~\textup(\ref{DR2}\textup) automatically satisfies~\textup(\ref{DR3}\textup).
\end{corollary}

In this paper, we shall also need the following property.

\begin{lemma} \label{lemma_>0}
Suppose that $\mathfrak F$ is a complete fan in $N_\QQ$ and $\mu_1,\mu_2 \in \mathfrak R(\mathfrak F)$.
If $\langle \rho_{\mu_1}, \mu_2 \rangle > 0$ and $\langle \rho_{\mu_2}, \mu_1 \rangle > 0 $, then $\mu_1 + \mu_2 = 0$.
\end{lemma}

\begin{proof}
The hypothesis implies $\langle \rho, \mu_1 + \mu_2 \rangle \ge 0$ for all $\rho \in \mathfrak F(X)^1$.
Since the fan $\mathfrak F(X)$ is complete, the latter is possible only if $\mu_1 + \mu_2 = 0$.
\end{proof}

\subsection{Toric varieties}
\label{ssec_tor_var}

Basic references for this subsections are~\cite{CLS, Fu, Oda}.

Let $T$ be a torus.
For every $\lambda \in \mathfrak X(T)$, let $f_\lambda$ denote the regular function on $T$ representing the character~$-\lambda$.
Then $f_\lambda$ is $T$-semi\-invariant of weight~$\lambda$, $f_{\lambda_1} \cdot f_{\lambda_2} = f_{\lambda_1+\lambda_2}$ for all $f_{\lambda_1},f_{\lambda_2} \in \mathfrak X(T)$, and there is the decomposition $\KK[T] = \bigoplus \limits_{\lambda \in \mathfrak X(T)} \KK f_\lambda$.

A $T$-variety $Z$ is said to be \textit{toric} if it is irreducible, normal, and has an open $T$-orbit.
Recall from \S\,\ref{ssec_torus_actions_and_gradings} the notions of the weight lattice~$M(Z)$ and weight monoid $\Gamma(Z)$.
Note that $M(Z)$ is identified with $\mathfrak X(T/T_0)$ where $T_0 \subset T$ is the kernel of the action of~$T$ on~$Z$.
In particular, $T$ acts on~$Z$ effectively if and only if $M(Z) = \mathfrak X(T)$.

Fix a sublattice $M \subset \mathfrak X(T)$ and let $M_\QQ, N, N_\QQ$ be as in~\S\,\ref{subsec_Dem_roots}.
In what follows we provide a description of all toric $T$-varieties with weight lattice~$M$.

Affine toric $T$-varieties with weight lattice $M$ are in bijection with strictly convex cones in~$N_\QQ$.
Namely, given such a cone $\mathcal C \subset N_\QQ$, consider the algebra $A_{\mathcal C} = \bigoplus \limits_{\lambda \in M \cap \mathcal C^\vee} \KK f_\lambda \subset \KK[T]$ and put $Z_{\mathcal C} = \Spec A_{\mathcal C}$.
Then $Z_{\mathcal C}$ is an affine toric $T$-variety with $M(Z_{\mathcal C}) = M$ and $\Gamma(Z_{\mathcal C}) = M \cap \mathcal C^\vee$.
Given a face $\mathcal E$ of the cone $\mathcal C$, the inclusion of algebras $A_{\mathcal C} \subset A_{\mathcal E}$ determines a morphism $Z_{\mathcal E} \to Z_{\mathcal C}$, which is a $T$-equivariant open embedding.
Thus $Z_{\mathcal E}$ is naturally identified with a $T$-stable affine open subset of~$Z_{\mathcal C}$.
The faces of $\mathcal C$ are in bijection with the $T$-orbits in~$Z_{\mathcal C}$: the $T$-orbit $O_{\mathcal E}$ corresponding to a face $\mathcal E$ of $\mathcal C$ is the unique closed $T$-orbit in~$Z_{\mathcal E}$.
One has $\dim O_{\mathcal E} = \rk M - \dim \mathcal E$.
Moreover, given two faces $\mathcal E_1, \mathcal E_2$ of~$\mathcal C$, one has $O_{\mathcal E_1} \subset \overline O_{\mathcal E_2}$ if and only if $\mathcal E_2$ is a face of~$\mathcal E_1$.

Arbitrary (not necessarily affine) toric $T$-varieties with weight lattice $M$ are parametrized by fans in $N_\QQ$.
More precisely, given a fan $\mathfrak F$ in~$N_\QQ$, the corresponding toric $T$-variety $Z$ is the union of $T$-stable affine open subsets $Z_{\mathcal C}$ where $\mathcal C$ runs over all cones in~$\mathfrak F$.
Given two cones $\mathcal C_1, \mathcal C_2 \in \mathfrak F$, the corresponding subsets $Z_{\mathcal C_1}, Z_{\mathcal C_2}$ are glued together via their common open subset $Z_{\mathcal C_1 \cap \mathcal C_2}$.
The $T$-orbits in $Z$ are in bijection with~$\mathfrak F$; we denote by $O_{\mathcal C}$ the $T$-orbit corresponding to a cone $\mathcal C \in \mathfrak F$.
Clearly, $\dim O_{\mathcal C} = \rk M - \dim \mathcal C$.

The set $\mathfrak F^1$ is in natural bijection with the $T$-stable prime divisors in~$Z$.
Namely, an element $\rho \in \mathfrak F^1$ corresponds to a $T$-stable prime divisor $D_\rho \in Z$ such that $D_\rho = \overline O_{\mathcal C}$ with $\mathcal C = \QQ_{\ge0}\rho$.
For every $\lambda \in M$, the order of the function $f_\lambda$ along $D_\rho$ equals $\langle \rho, \lambda \rangle$.
In particular, one has $\Gamma(Z) = \lbrace \lambda \in M \mid \langle \rho, \lambda \rangle \ge 0 \ \text{for all} \ \rho \in \mathfrak F^1 \rbrace$.

\begin{remark}
In the definition of a toric $T$-variety it is often additionally required that the action of~$T$ be effective.
This corresponds to $M = \mathfrak X(T)$ in our notation.
\end{remark}

\subsection{Root subgroups on affine toric varieties}

Fix a sublattice $M \subset \mathfrak X(T)$ and let $M_\QQ, N, N_\QQ$ be as in~\S\,\ref{subsec_Dem_roots}.
Let $Z$ be an affine toric $T$-variety with $M(Z) = M$.
As discussed in \S\,\ref{ssec_tor_var}, $Z = Z_{\mathcal C} = \Spec A_{\mathcal C}$ for a strictly convex cone $\mathcal C \subset N_\QQ$.
The weight monoid of $Z$ is $\Gamma = M \cap \mathcal C^\vee$.
Given any $\mu \in \mathfrak R(\mathcal C)$, one defines a $T$-normalized LND $\partial_\mu$ of weight~$\mu$ on~$A_{\mathcal C}$ by the rule
\begin{equation} \label{eqn_T-norm_LND}
\partial_\mu(f_\lambda)=\langle \rho_\mu, \lambda \rangle f_{\lambda + \mu}
\end{equation}
for all $\lambda \in M \cap \mathcal C^\vee$.
This LND corresponds to a $T$-root subgroup on~$Z$, which we denote by $H_\mu$.

For future reference, we mention that, in view of formula~(\ref{eqn_phi_d}), for every $c \in \KK^\times$ the $\GG_a$-action on $\KK[Z]$ corresponding to the LND $c\partial_\mu$, is given by
\begin{equation} \label{eqn_dmu_int}
(s, f_\lambda) \mapsto f_\lambda(1 + scf_\mu)^{\langle \rho_\mu, \lambda \rangle}
\end{equation}
for all $s \in \KK$ and $\lambda \in M \cap \mathcal C^\vee$.

It is known from \cite[Theorem~2.7]{L1} that every nonzero $T$-normalized LND on~$A_{\mathcal C}$ has the form $c\partial_\mu$ for some $\mu \in \mathfrak R(\mathcal C)$ and $c \in \KK^\times$.
For convenience of the reader, below we provide a direct proof of this result.

\begin{theorem} \label{thm_T-root_subgroups_aff}
The following assertions hold.
\begin{enumerate}[label=\textup{(\alph*)},ref=\textup{\alph*}]
\item \label{thm_T-root_subgroups_aff_a}
The map $\mu \mapsto \partial_\mu$ is a bijection between $\mathfrak R(\mathcal C)$ and the nonzero $T$-normalized LNDs on $\KK[Z]$ modulo proportionality.
\item \label{thm_T-root_subgroups_aff_b}
The map $\mu \mapsto H_\mu$ is a bijection between $\mathfrak R(\mathcal C)$ and the $T$-root subgroups on~$Z$.
\end{enumerate}
\end{theorem}

\begin{proof}
Part~(\ref{thm_T-root_subgroups_aff_b}) is a direct consequence of~(\ref{thm_T-root_subgroups_aff_a}) by Proposition~\ref{prop_Ga-actions}.
To prove~(\ref{thm_T-root_subgroups_aff_a}), it suffices to show that every nonzero $T$-normalized LND on~$\KK[Z] = A_{\mathcal C}$ is proportional to $\partial_\mu$ for some $\mu \in \mathfrak R(\mathcal C)$.
Let $\partial$ be a nonzero $T$-normalized LND on $A_{\mathcal C}$ of weight~$\mu$.
Then it naturally extends to a $T$-normalized derivation (still denoted by~$\partial$) of weight~$\mu$ on the algebra $A = \bigoplus \limits_{\lambda \in M} \KK f_\lambda$.
Note that $\partial$ is homogeneous of weight~$\mu$ (see~\S\,\ref{ssec_torus_actions_and_gradings}) but may be no more locally nilpotent on~$A$.
Choose a basis $\lambda_1,\ldots,\lambda_n \in M$; then there are $\xi_1,\ldots,\xi_n \in \KK$ such that $\partial(f_{\lambda_i}) = \xi_i f_{\lambda_i+\mu}$ for all $i=1,\ldots,n$.
Let $\overline \rho \in \Hom_\ZZ(M,\KK)$ be such that $\langle \overline \rho, \lambda_i \rangle = \xi_i$ for all $i = 1,\ldots, n$.
Then a direct computation shows that $\partial(f_\lambda) = \langle \overline \rho, \lambda \rangle f_{\lambda + \mu}$ for all $\lambda \in M$.
Since $\partial \ne 0$, one has $\overline \rho \ne 0$, whence there is $\lambda_0 \in \Gamma$ such that $\langle \overline \rho, \lambda_0 \rangle \ne 0$.
As $\partial$ is locally nilpotent on $A_{\mathcal C}$, there is the minimal $r \in \ZZ_{>0}$ such that $\partial^r(f_{\lambda_0}) \ne 0$ and $\partial^{r+1}(f_{\lambda_0}) = 0$.
Then $\langle \overline \rho, \lambda_0 + r\mu \rangle = 0$ and hence $\langle \overline \rho, \lambda_0 \rangle + r\langle \overline \rho, \mu \rangle = 0$.
Since $\langle \overline \rho, \lambda_0 \rangle \ne 0$ and $r > 0$, it follows that $\langle \overline \rho, \mu \rangle \ne 0$.
Put $c = - \langle \overline \rho, \mu \rangle$ and $\rho_\mu = \overline \rho/c$.
Then for every $\lambda \in M$ we obtain $\partial(f_\lambda) = c \langle \rho_\mu, \lambda \rangle f_{\lambda + \mu}$ and $\langle \rho_\mu, \mu\rangle = \langle \overline \rho, \mu \rangle / c = -1$.
Again, since $\partial$ is locally nilpotent on~$A_{\mathcal C}$, for every $\lambda \in \Gamma$ there is the minimal $r \in \ZZ_{\ge0}$ such that $\partial^r(f_{\lambda}) \ne 0$ and $\partial^{r+1}(f_{\lambda}) = 0$.
Then $\langle \rho_\mu, \lambda + r\mu \rangle = 0$ and hence $\langle \rho_\mu, \lambda \rangle = -r\langle \rho_\mu, \mu \rangle = r \ge 0$.
We conclude that $\rho_\mu \in \mathcal C$ and $\rho_\mu \in N$.
Moreover, $\rho_\mu$ is primitive in~$N$ as $\langle \rho_\mu, \mu \rangle = -1$.
Now take any $\rho \in \mathcal C^1$ and assume that $\rho \ne \rho_\mu$.
Then $\rho, \rho_\mu$ are not proportional, hence there is $\lambda_0 \in \Gamma$ with $\langle \rho, \lambda_0 \rangle = 0$ and $\langle \rho_\mu, \lambda_0 \rangle = r > 0$.
Then $\partial^r(f_{\lambda_0}) \ne 0$, hence $f_{\lambda_0 + r\mu} \in A_{\mathcal C}$, hence $\lambda_0 + r\mu \in \Gamma$, hence $\langle \rho, \lambda_0 + r\mu \rangle \ge 0$, hence $\langle \rho, \mu \rangle \ge 0$.
If $\rho_\mu \notin \mathcal C^1$, then $\rho_\mu = \sum \limits_{\rho \in \mathcal C^1} a_\rho \rho$ with $a_\rho \ge 0$ for all $\rho \in \mathcal C^1$, which implies $\langle \rho_\mu, \mu \rangle \ge 0$, a contradiction.
Thus $\rho_\mu \in \mathcal C^1$ and $\langle \rho, \mu \rangle \ge 0$ for all $\rho \in \mathcal C^1 \setminus \lbrace \rho_\mu \rbrace$.
We have proved that $\mu \in \mathfrak R(\mathcal C)$ and $\partial = c \partial_\mu$ as required.
\end{proof}

In the remaining part of this subsection, we study how a $T$-root subgroup on $Z$ acts on $T$-orbits.
To this end, we consider a more general situation, which will be also needed later in~\S\,\ref{ssec_stand_aff}.

Suppose $Z$ is an arbitrary irreducible affine $T$-variety (not necessarily toric) with weight lattice $M$ and weight monoid~$\Gamma$.
Suppose in addition that $\Gamma = M \cap \mathcal C^\vee$ for a strictly convex cone $\mathcal C$ in~$N_\QQ$.
For every face $\mathcal E$ of~$\mathcal C$, put $\Gamma_{\mathcal E} = \lbrace \lambda \in \Gamma \mid \langle v, \lambda \rangle = 0 \ \text{for all} \ v \in \mathcal E \rbrace$, so that $\QQ_{\ge0} \Gamma_{\mathcal E}$ is the face of $\mathcal C^\vee$ dual to~$\mathcal E$.
Consider the grading $\KK[Z] = \bigoplus \limits_{\lambda \in \Gamma} \KK[Z]_\lambda$ induced by the $T$-action.
For every face $\mathcal E$ of~$\mathcal C$, consider the ideal $I_{\mathcal E} = \bigoplus \limits_{\lambda \in \Gamma \setminus \Gamma_{\mathcal E}} \KK[Z]_\lambda$ in $\KK[Z]$ and let $Y_{\mathcal E}$ be the respective closed subvariety in~$Z$.
Fix any $\mu \in \mathfrak R(\mathcal C)$ and let $\rho_\mu \in \mathcal C^1$ be the corresponding element.
Let $\partial$ be a nonzero homogeneous LND on $\KK[Z]$ of weight~$\mu$ and assume that $\Ker \partial = \bigoplus \limits_{\lambda \in \Gamma_{\mu}} \KK[Z]_{\lambda}$ where $\Gamma_\mu = \Gamma_{\QQ_{\ge0}\rho_\mu}$ for short.
Let $H$ denote the $T$-root subgroup on~$Z$ corresponding to~$\partial$.

\begin{proposition} \label{prop_stable_ideals}
Given a face $\mathcal E$ of~$\mathcal C$, the following assertions hold.
\begin{enumerate}[label=\textup{(\alph*)},ref=\textup{\alph*}]
\item \label{prop_stable_ideals_a}
If there is $\rho \in \mathcal E^1$ such that $\langle \rho, \mu \rangle > 0$, then $Y_{\mathcal E}$ is pointwise fixed by~$H$.
\item \label{prop_stable_ideals_b}
If $\langle \mathcal E, \mu \rangle = 0$, then $Y_{\mathcal E}$ is $H$-stable with nontrivial $T$-action.
\item \label{prop_stable_ideals_c}
If $\langle \mathcal E, \mu \rangle \le 0$ and $\rho_\mu \in \mathcal E^1$, then $Y_{\mathcal E}$ is $H$-unstable.
\end{enumerate}
\end{proposition}

\begin{proof}
(\ref{prop_stable_ideals_a})
In this case, $\partial(\KK[Z]) \subset I_{\mathcal E}$, hence the induced action of $\partial$ on $\KK[Z]/I_{\mathcal E} \simeq \KK[Y_{\mathcal E}]$ is trivial, and so $Y_{\mathcal E}$ is pointwise fixed by~$H$.

(\ref{prop_stable_ideals_b})
As $\langle \mathcal E, \mu \rangle = 0$, one has $\partial(I_{\mathcal E}) \subset I_{\mathcal E}$, hence $Y_{\mathcal E}$ is $H$-stable.
Further, observe that $\partial$ preserves the subalgebra $\bigoplus \limits_{\lambda \in \Gamma_{\mathcal E}} \KK[Z]_\lambda$ and is nontrivial on it.
Since this subalgebra maps isomorphically to $\KK[Z]/I_{\mathcal E} \simeq \KK[Y_{\mathcal E}]$, it follows that $H$ acts nontrivially on~$Y_{\mathcal E}$.

(\ref{prop_stable_ideals_c})
Choose an element $\nu \in \Gamma_{\mathcal E}$ such that $\langle \rho, \nu \rangle > 0$ for all $\rho \in \mathcal C^1 \setminus \mathcal E^1$.
Then there is $N \in \ZZ_{>0}$ such that the element $\lambda = N\nu - \mu$ belongs to~$\Gamma$.
Clearly, $\lambda \notin \Gamma_{\mathcal E}$, $\lambda + \mu \in \Gamma_{\mathcal E}$, and $\KK[Z]_\lambda \cap \Ker \partial = \lbrace 0 \rbrace$, which implies $\partial(I_{\mathcal E}) \not\subset I_{\mathcal E}$, and so $Z(\mathcal E)$ is not $H$-stable.
\end{proof}

For every face $\mathcal E$ of~$\mathcal C$, let $U_{\mathcal E}$ denote the open subset of $Y_{\mathcal E}$ obtained by removing all subvarieties $Y_{\mathcal E'}$ where $\mathcal E'$ runs over all faces of $\mathcal C$ strictly containing~$\mathcal E$.

\begin{proposition} \label{prop_stable_subvar}
Given a face $\mathcal E$ of~$\mathcal C$, the following assertions hold.
\begin{enumerate}[label=\textup{(\alph*)},ref=\textup{\alph*}]
\item \label{prop_stable_subvar_a}
If there is $\rho \in \mathcal E^1$ such that $\langle \rho, \mu \rangle > 0$, then $U_{\mathcal K}$ is pointwise fixed by~$H$.

\item \label{prop_stable_subvar_b}
If $\langle \mathcal E, \mu \rangle \le 0$, then there exist faces $\mathcal K, \widetilde{\mathcal K}$ of $\mathcal C$ such that $\langle \mathcal K, \mu \rangle = 0$, $\widetilde{\mathcal K}$ is generated by $\mathcal K$ and~$\rho_\mu$, $\mathcal E \in \lbrace \mathcal K, \widetilde{\mathcal K}\rbrace$, $HU_{\mathcal E} \subset U_{\mathcal K} \cup U_{\widetilde{\mathcal K}}$, and $HU_{\mathcal E}$ meets both $U_{\mathcal K}$ and $U_{\widetilde{\mathcal K}}$.
\end{enumerate}
\end{proposition}

\begin{proof}
(\ref{prop_stable_subvar_a})
This is a direct consequence of Proposition~\ref{prop_stable_ideals}(\ref{prop_stable_ideals_a}).

(\ref{prop_stable_subvar_b})
If $\langle \mathcal E, \mu \rangle = 0$, then we put $\mathcal K = \mathcal E$; by Lemma~\ref{lemma_faces1}(\ref{lemma_faces1_a}), the cone $\widetilde{\mathcal K}$ generated by $\mathcal K$ and~$\rho_\mu$ is a face of~$\mathcal C$.
If $\rho_\mu \in \mathcal E^1$, then we put $\widetilde{\mathcal K} = \mathcal E$; by Lemma~\ref{lemma_faces1}(\ref{lemma_faces1_b}), the cone $\mathcal K = \lbrace v \in \widetilde{\mathcal K} \mid \langle v, \mu \rangle = 0 \rbrace$ is a face of~$\mathcal C$.
By Proposition~\ref{prop_stable_ideals}(\ref{prop_stable_ideals_b},\ref{prop_stable_ideals_c}), the subset $Y_{\mathcal K}$ is $H$-stable and $U_{\widetilde{\mathcal K}}$ is not, therefore it suffices to show that the subset $U_{\mathcal K} \cup U_{\widetilde{\mathcal K}} \subset Y_{\mathcal K}$ is $H$-stable.
Let $\mathcal E'\ne \widetilde{\mathcal K}$ be a face of $\mathcal C$ strictly containing $\mathcal K$; then we need to show that
\begin{equation} \label{eqn_HY}
HY_{\mathcal E'} \cap  (U_{\mathcal K} \cup U_{\widetilde{\mathcal K}}) = \varnothing.
\end{equation}
If there is $\rho \in \mathcal E'^1$ with $\langle \rho, \mu \rangle > 0$ or $\langle \mathcal E', \mu \rangle = 0$, then $Y_{\mathcal E'}$ is $H$-stable by Proposition~\ref{prop_stable_ideals}(\ref{prop_stable_ideals_a},\ref{prop_stable_ideals_b}), and so~(\ref{eqn_HY}) holds.
If $\langle \mathcal E', \mu \rangle \le 0$ and $\rho_\mu \in \mathcal E'^1$, then $\mathcal E'' = \lbrace v \in \mathcal E' \mid \langle v, \mu \rangle = 0 \rbrace$ is also a face of $\mathcal C$ by Lemma~\ref{lemma_faces1}(\ref{lemma_faces1_a}); moreover, $\mathcal E''$ strictly contains~$\mathcal K$ and $\mathcal E'' \ne \widetilde{\mathcal K}$.
By Proposition~\ref{prop_stable_ideals}(\ref{prop_stable_ideals_b}), $Y_{\mathcal E''}$ is $H$-stable.
As $Y_{\mathcal E'} \subset Y_{\mathcal E''}$, (\ref{eqn_HY}) holds in this case as well.
\end{proof}

We now come back to the situation where $Z$ is an affine toric $T$-variety with weight lattice~$M$ and weight monoid~$\Gamma$.
Then for every face $\mathcal E$ of~$\mathcal C$ one has $U_{\mathcal E} = O_{\mathcal E}$ and $Y_{\mathcal E} = \overline O_{\mathcal E}$.
Fix $\mu \in \mathfrak R(\mathcal C)$ and let $H = H_\mu$ be the corresponding $T$-root subgroup on~$Z$.
Propositions~\ref{prop_stable_oc} and~\ref{prop_stable_orbits} below follow from Propositions~\ref{prop_stable_ideals} and~\ref{prop_stable_subvar}, respectively.

\begin{proposition} \label{prop_stable_oc}
Given a face $\mathcal E$ of~$\mathcal C$, the following assertions hold.
\begin{enumerate}[label=\textup{(\alph*)},ref=\textup{\alph*}]
\item \label{prop_stable_oc_a}
If there is $\rho \in \mathcal E^1$ such that $\langle \rho, \mu \rangle > 0$, then $\overline O_{\mathcal E}$ is pointwise fixed by~$H$.
\item \label{prop_stable_oc_b}
If $\langle \mathcal E, \mu \rangle = 0$, then $\overline O_{\mathcal E}$ is $H$-stable with nontrivial $T$-action.
\item \label{prop_stable_oc_c}
If $\langle \mathcal E, \mu \rangle \le 0$ and $\rho_\mu \in \mathcal E^1$, then $\overline O_{\mathcal E}$ is $H$-unstable.
\end{enumerate}
\end{proposition}

Recall from \S\,\ref{ssec_tor_var} that every $\rho \in \mathcal C^1$ corresponds to a $T$-stable prime divisor $D_\rho$ in~$Z$.

\begin{corollary} \label{crl_T-orbits_H}
Given $\rho \in \mathcal C^1$, the divisor $D_\rho$ is $H$-stable if and only if $\langle \rho,\mu \rangle \ge 0$.
Moreover, $D_\rho$ is pointwise fixed by~$H$ if and only if $\langle \rho, \mu \rangle > 0$.
\end{corollary}

\begin{proposition} \label{prop_stable_orbits}
Given a face $\mathcal E$ of~$\mathcal C$, the following assertions hold.
\begin{enumerate}[label=\textup{(\alph*)},ref=\textup{\alph*}]
\item \label{prop_stable_orbits_a}
If there is $\rho \in \mathcal E^1$ such that $\langle \rho, \mu \rangle > 0$, then $O_{\mathcal E}$ is pointwise fixed by~$H$.

\item \label{prop_stable_orbits_b}
If $\langle \mathcal E, \mu \rangle \le 0$, then there exist faces $\mathcal K, \widetilde{\mathcal K}$ of $\mathcal C$ such that $\langle \mathcal K, \mu \rangle = 0$, $\widetilde{\mathcal K}$ is generated by $\mathcal K$ and~$\rho_\mu$, $\mathcal E \in \lbrace \mathcal K, \widetilde{\mathcal K}\rbrace$, and $HO_{\mathcal E} = O_{\mathcal K} \cup O_{\widetilde{\mathcal K}}$.
\end{enumerate}
\end{proposition}

\subsection{Root subgroups on arbitrary toric varieties}

In this subsection, we provide a complete self-contained description of all $T$-root subgroups on an arbitrary toric $T$-variety, which goes back to Demazure \cite{Dem}; see also \cite{Oda}.
The main result of this subsection is Theorem~\ref{thm_T-root_subgroups}.

Let $Z$ be a toric $T$-variety.
Let $M$ be the weight lattice of~$Z$, put $N = \Hom_\ZZ(M,\ZZ)$, and retain the notation of~\S\,\ref{subsec_Dem_roots}.
Let $\mathfrak F$ be the fan in $N_\QQ$ corresponding to~$Z$ as in \S\,\ref{ssec_tor_var}.

\begin{proposition} \label{prop_HorbT}
Let $H$ be a $T$-root subgroup on~$Z$ and let $Y \subset Z$ be an $H$-orbit.
Suppose $Y$ is not a point.
Then $Y$ meets exactly two $T$-orbits $O_1,O_2 \subset Z$, which satisfy $\dim O_1 = \dim O_2 + 1$.
Moreover, $Y \cap O_2$ is a single point and $HO_1 = HO_2 = O_1 \cup O_2$.
\end{proposition}

\begin{proof}
Since $Y$ is not a point, one has $Y \simeq \AA^1$.
Put $T_0 = \lbrace t \in T \mid tY \subset Y \rbrace$.
We claim that
\begin{equation} \label{eqn_TyY}
Ty \cap Y = T_0y \quad \text{for all} \quad y \in Y.
\end{equation}
The inclusion $\supset$ is clear.
Conversely, take any $y' \in Ty \cap Y$; then $y' = ty$ for some $t \in T$.
It remains to show that $t \in T_0$.
For every $z \in Y$ one has $z = hy$ for some $h \in H$, and so $tz = thy = tht^{-1}ty = (tht^{-1})y' \in Hy' = Y$.
Thus $t \in T_0$ and (\ref{eqn_TyY}) is proved.

Consider the homomorphism $\varphi \colon T_0 \to \Aut Y \simeq \Aut \AA^1$.
Since all automorphisms of $\AA^1$ are well-known to have the form $x \mapsto ax+b$ for some $a,b \in \KK$, $a \ne 0$, it follows that $\Im\varphi$ is a diagonalizable subgroup in $\Aut Y$.
By~(\ref{eqn_TyY}), there are only finitely many $T_0$-orbits in~$Y$, which implies $\Im \varphi \simeq \KK^\times$, whence $Y$ splits into two $T_0$-orbits with one of them being a $T_0$-fixed point~$y_0$.
Thanks to~(\ref{eqn_TyY}), $Y$ meets two $T$-orbits $O_1 = T(Y \setminus \lbrace y_0 \rbrace)$ and $O_2 = Ty_0$.
For every $y \in Y$, $t \in T_y$, and $h \in H$ one has $t(hy) = (tht^{-1})ty = (tht^{-1})y \subset Hy = Y$, hence $T_y \subset T_0$ and $T_y = (T_0)_y$.
It follows that $\dim O_1 - \dim O_2 = \dim (Y \setminus \lbrace y_0 \rbrace) - \dim \lbrace y_0 \rbrace = 1$.
Clearly, $tY$ is an $H$-orbit for every $t \in T$, which completes the proof.
\end{proof}

\begin{proposition} \label{prop_TRS_prop}
Let $H$ be a $T$-root subgroup on~$Z$ of weight~$\mu$.
The following assertions hold.
\begin{enumerate}[label=\textup{(\alph*)},ref=\textup{\alph*}]
\item \label{prop_TRS_prop_a}
$\mu \in M$.
\item \label{prop_TRS_prop_b}
There is $\rho_\mu \in \mathfrak F^1$ such that $\langle \rho_\mu, \mu \rangle = -1$.
Moreover, given $\rho \in \mathfrak F^1$, the divisor $D_{\rho}$ is $H$-stable if and only if $\rho \ne \rho_\mu$.
\item \label{prop_TRS_prop_c}
Fix an isomorphism $\GG_a \xrightarrow{\sim} H$, $s \mapsto H(s)$.
Then there is a constant $c \in \KK^\times$ such that
\begin{equation} \label{fla_action}
H(s)\cdot f_\lambda = f_\lambda(1 + csf_\mu)^{\langle \rho_\mu, \lambda \rangle}
\end{equation}
for all $s \in \KK$ and $\lambda \in M$.
\item \label{prop_TRS_prop_d}
Every $\rho \in \mathfrak F^1 \setminus \lbrace \rho_\mu \rbrace$ satisfies $\langle \rho, \mu \rangle \ge 0$.
\item \label{prop_TRS_prop_e}
Suppose a cone $\mathcal K \in \mathfrak F$ is such that $\langle \mathcal K, \mu \rangle = 0$.
Then the cone generated by $\mathcal K$ and $\rho_\mu$ also belongs to~$\mathfrak F$.
\end{enumerate}
\end{proposition}

\begin{proof}
(\ref{prop_TRS_prop_a})
Clearly, $\mu$ vanishes on all elements of the kernel of the action of $T$ on~$Z$, hence $\mu \in M$.

(\ref{prop_TRS_prop_b})
Let $O$ be the open $T$-orbit in~$Z$.
Since $H$ acts nontrivially on~$Z$, by Proposition~\ref{prop_HorbT} there is a $T$-orbit $O_H \subset Z$ of codimension~$1$ such that $HO = O \cup O_H$.
Put $Z_H = HO$; this is an $H$-stable affine open subset of~$Z$.
Let $\rho_\mu \in \mathfrak F^1$ be the element corresponding to~$O_H$ (and its closure in~$Z$).
Then $Z_H$, regarded as an affine toric $T$-variety, corresponds to the cone $\QQ_{\ge0} \rho_\mu$.
By Theorem~\ref{thm_T-root_subgroups_aff}(\ref{thm_T-root_subgroups_aff_b}), $\mu$ is a Demazure root of this cone, and so $\langle \rho_\mu, \mu \rangle = -1$.
Since the subset $Z_H \subset Z$ is $H$-stable, all $T$-stable prime divisors in~$Z$ except~$D_{\rho_\mu}$ are $H$-stable, whence the second claim.

(\ref{prop_TRS_prop_c})
One has $\KK[Z_H] = \bigoplus\limits_{\lambda \in M : \langle \rho_{\mu}, \lambda \rangle \ge 0} \KK f_{\lambda}$.
Thanks to Theorem~\ref{thm_T-root_subgroups_aff} and formula~(\ref{eqn_dmu_int}), there is a constant $c \in \KK^\times$ such that the action of $H$ on $\KK[Z_H]$ is given by~(\ref{fla_action}) for all $s \in \KK$ and $\lambda \in M$ with $\langle \rho_\mu, \lambda \rangle \ge 0$.
Observe that the same formula remains valid for all $\lambda \in M$.

(\ref{prop_TRS_prop_d})
If $\rho = - \rho_\mu$, then the assertion is obvious.
In what follows we assume that $\rho$ and $\rho_\mu$ are not proportional.
Then there exists $\nu \in M$ such that $\langle \rho, \nu \rangle = 0$ and $\langle \rho_\mu, \nu \rangle = k > 0$.
Clearly, $\mathrm{ord}_{D_\rho} (f_\nu) = \langle \rho, \nu \rangle = 0$.
Thanks to~(\ref{prop_TRS_prop_b}), the divisor $D_\rho$ is $H$-stable, and so for all $s \in \KK$ one has $\mathrm{ord}_{D_\rho} (H(s)\cdot f_\nu) = 0$.
Now assume that $\langle \rho, \mu \rangle < 0$.
Then, by formula~(\ref{fla_action}), for all $s \ne 0$ one has $\mathrm{ord}_{D_\rho} (H(s)\cdot f_\nu) = k \cdot \mathrm{ord}_{D_\rho} (1+csf_\mu) = k\langle \rho,\mu\rangle < 0$, a contradiction.

(\ref{prop_TRS_prop_e})
Recall from~(\ref{prop_TRS_prop_b}) that all divisors $D_\rho$ with $\rho \in \mathfrak F^1 \setminus \lbrace \rho_\mu \rbrace$ are $H$-stable.
Removing from~$Z$ a suitable collection of them, we may assume that $\mathfrak F^1 = \mathcal K^1 \cup \lbrace \rho_\mu \rbrace$.
Consider the cone $\widetilde{\mathcal K} = \QQ_{\ge0} (\mathcal K^1 \cup \lbrace \rho_\mu \rbrace)$ and let $\widetilde{\mathfrak F}$ be the fan consisting of all faces of~$\widetilde{\mathcal K}$.
Lemma~\ref{lemma_add_cone} yields $\mathfrak F \subset \widetilde{\mathfrak F}$.
Let $\widetilde Z$ be the affine toric $T$-variety corresponding to~$\widetilde{\mathcal K}$ (and~$\widetilde{\mathfrak F})$.
Then there is a natural $T$-equivariant embedding $Z \hookrightarrow \widetilde Z$.
Since $\mu$ is a Demazure root of the cone $\widetilde{\mathcal K}$, the action of $H$ on $Z$ naturally extends to~$\widetilde Z$.
By Proposition~\ref{prop_stable_orbits}(\ref{prop_stable_orbits_b}), the $T$-orbits $O_{\mathcal K}$ and $O_{\widetilde{\mathcal K}}$ in $\widetilde Z$ are connected by~$H$, hence $O_{\widetilde{\mathcal K}} \subset Z$ and $\widetilde{\mathcal K} \in \mathfrak F$.
\end{proof}

\begin{corollary} \label{cor_weight_is_DR}
Let $H$ be a $T$-root subgroup on~$Z$ of weight~$\mu$.
Then $\mu \in \mathfrak R(\mathfrak F)$.
\end{corollary}

\begin{proposition} \label{prop_ex_of_RS}
Suppose $\mu \in \mathfrak R(\mathfrak F)$.
Then there exists a unique $T$-root subgroup of weight~$\mu$ on~$Z$.
\end{proposition}

\begin{proof}
For the cone $\mathcal C_\mu = \QQ_{\ge0}\rho_\mu \in \mathfrak F$, consider the $\GG_a$-action on $Z_{\mathcal C_\mu}$ corresponding to the LND $\partial_\mu$ on $\KK[Z_{\mathcal C_\mu}]$ given by formula~(\ref{eqn_T-norm_LND}).
We shall show that this $\GG_a$-action extends to the whole~$Z$, which would yield the desired $T$-root subgroup on~$Z$.

By~(\ref{eqn_dmu_int}), the corresponding $\GG_a$-action on~$\KK[Z_{\mathcal C_\mu}]$ is given by the formula
\begin{equation} \label{fla_action1}
(s, f_\lambda) \mapsto f_\lambda(1 + sf_\mu)^{\langle \rho_\mu, \lambda \rangle}
\end{equation}
for all $s \in \KK$ and $\lambda \in M$ with $\langle \rho_\mu, \lambda \rangle \ge 0$.
Observe that the corresponding algebra homomorphism $\KK[Z_{\mathcal C_\mu}] \to \KK[\AA^1 \times Z_{\mathcal C_\mu}]$ is given by
\begin{equation} \label{eqn_alg_hom}
f_\lambda \mapsto  f_\lambda(1 - \epsilon f_\mu)^{\langle \rho_\mu, \lambda \rangle}
\end{equation}
for all $\lambda \in M$ with $\langle \rho_\mu, \lambda \rangle \ge 0$, where $\epsilon$ is the coordinate function on~$\AA^1$, $\epsilon(s) = s$ for all $s \in \KK$.

Take an arbitrary cone $\mathcal C \in \mathfrak F$; it remains to prove that the above $\GG_a$-action on $Z_{\mathcal C_\mu}$ extends to a morphism $\AA^1 \times Z_{\mathcal C} \to Z$.

Case 1: $\rho_\mu \in \mathcal C^1$.
Then $\mu$ is a Demazure root of the cone~$\mathcal C$, therefore the LND $\partial_\mu$ preserves the subalgebra $\KK[Z_{\mathcal C}] \subset \KK[Z_{\mathcal C_\mu}]$, and so the $\GG_a$-action on $Z_{\mathcal C_\mu}$ extends to~$Z_{\mathcal C}$.

Case 2: $\rho_\mu \notin \mathcal C^1$.
Then $\langle \rho, \mu \rangle \geqslant 0$ for all $\rho \in \mathcal C^1$.
Let $\mathcal K$ be the cone generated by the set $\lbrace \rho \in \mathcal C^1 \mid \langle \rho, \mu \rangle = 0 \rbrace$.
Then $\mathcal K$ is a face of~$\mathcal C$.
Let $\widetilde{\mathcal K}$ be the cone generated by~$\mathcal K$ and~$\rho_\mu$.
By~(\ref{DR3}), one has $\widetilde{\mathcal K} \in \mathfrak F$.
Consider the regular functions $g = 1 - \epsilon f_\mu$ and $h = f_\mu$ on $\AA^1 \times Z_{\mathcal C}$ and let $(\AA^1 \times Z_{\mathcal C})_g$ and $(\AA^1 \times Z_{\mathcal C})_h$ be the principal open subsets defined by the nonvanishing of~$g$ and~$h$, respectively.
Then formula~(\ref{eqn_alg_hom}) defines algebra homomorphisms
\[
\KK[Z_{\mathcal C}] \to \KK[(\AA^1 \times Z_{\mathcal C})_g] = \KK[\AA^1 \times Z_{\mathcal C}][g^{-1}]
\]
and
\[
\KK[Z_{\widetilde{\mathcal K}}] \to \KK[(\AA^1 \times Z_{\mathcal C})_h] = \KK[\AA^1 \times Z_{\mathcal C}][h^{-1}]
\]
(in the second case, $f_\lambda \cdot f_\mu^N$ is regular on $Z_{\mathcal C}$ for a sufficiently large power~$N$), which in turn define morphisms $(\AA^1 \times Z_{\mathcal C})_g \to Z_{\mathcal C}$ and $(\AA^1 \times Z_{\mathcal C})_h \to Z_{\widetilde{\mathcal K}}$.
Since $g + \epsilon h = 1$, one has $(\AA^1 \times Z_{\mathcal C})_g \cup (\AA^1 \times Z_{\mathcal C})_h  = \AA^1 \times Z_{\mathcal C}$, and thus the two morphisms in fact glue together to a morphism $\AA^1 \times Z_{\mathcal C} \to Z_{\mathcal C} \cup Z_{\widetilde{\mathcal K}}$, which extends the $\GG_a$-action on $Z_{\mathcal C_\mu}$.

Let $H$ denote the $T$-root subgroup of weight~$\mu$ on~$Z$ constructed above.
As follows from the proof of Proposition~\ref{prop_TRS_prop}(\ref{prop_TRS_prop_b}), any $T$-root subgroup $H'$ of weight~$\mu$ on~$Z$ preserves the open subset~$Z_{\mathcal C_\mu}$.
By Theorem~\ref{thm_T-root_subgroups_aff}(\ref{thm_T-root_subgroups_aff_b}), $H'$ and~$H$ coincide on~$Z_{\mathcal C_\mu}$, hence they coincide on the whole~$Z$.
\end{proof}

For every $T$-root subgroup $H$ on~$Z$, let $\mu(H)$ denote its weight.
The next result follows from Corollary~\ref{cor_weight_is_DR} and Proposition~\ref{prop_ex_of_RS}.

\begin{theorem} \label{thm_T-root_subgroups}
The map $H \mapsto \mu(H)$ is a bijection between the $T$-root subgroups on~$Z$ and the set $\mathfrak R(\mathfrak F)$.
\end{theorem}

\begin{proposition} \label{prop_stable_orbits2}
Given a cone $\mathcal E \in \mathfrak F$, the following assertions hold.
\begin{enumerate}[label=\textup{(\alph*)},ref=\textup{\alph*}]
\item \label{prop_stable_orbits2_a}
If there is $\rho \in \mathcal E^1$ such that $\langle \rho, \mu \rangle > 0$, then $O_{\mathcal E}$ is pointwise fixed by~$H$.

\item \label{prop_stable_orbits2_b}
If $\langle \mathcal E, \mu \rangle \le 0$, then there exist cones $\mathcal K, \widetilde{\mathcal K} \in \mathfrak F$ such that $\langle \mathcal K, \mu \rangle = 0$, $\widetilde{\mathcal K}$ is generated by $\mathcal K$ and~$\rho_\mu$, $\mathcal E \in \lbrace \mathcal K, \widetilde{\mathcal K}\rbrace$, and $HO_{\mathcal E} = O_{\mathcal K} \cup O_{\widetilde{\mathcal K}}$.
\end{enumerate}
\end{proposition}

\begin{proof}
(\ref{prop_stable_orbits2_a})
It suffices to prove the claim in the case $\mathcal E = \QQ_{\ge0}\rho$.
Thanks to Proposition~\ref{prop_TRS_prop}(\ref{prop_TRS_prop_b}), removing from~$Z$ a suitable collection of $T$-stable prime divisors we may assume that $\mathfrak F^1 = \lbrace \rho_\mu, \rho \rbrace$.
Consider the cone $\mathcal K = \QQ_{\ge0}\lbrace \rho_\mu, \rho \rbrace$ and let $\widetilde Z$ be the affine toric $T$-variety corresponding to~$\mathcal K$.
Then $Z \subset \widetilde Z$ and $\mu \in \mathfrak R(\mathcal K)$, so $H$ extends to a $T$-root subgroup on~$\widetilde Z$.
Now the claim follows from Proposition~\ref{prop_stable_orbits}(\ref{prop_stable_orbits_a}).

(\ref{prop_stable_orbits2_b})
Again, thanks to Proposition~\ref{prop_TRS_prop}(\ref{prop_TRS_prop_b}), removing from~$Z$ a suitable collection of $T$-stable prime divisors we may assume that $\mathfrak F^1 = \mathcal E^1 \cup \lbrace \rho_\mu \rbrace$.
Let $\widetilde{\mathcal K}$ be the cone generated by $\mathcal E^1$ and~$\rho_\mu$.
Then either $\mathcal E = \widetilde{\mathcal K}$ or $\langle \mathcal E, \mu \rangle = 0$.
In the latter case, $\widetilde{\mathcal K} \in \mathfrak F$ by Proposition~\ref{prop_TRS_prop}(\ref{prop_TRS_prop_e}).
We have obtained that $\widetilde{\mathcal K} \in \mathfrak F$, hence $Z = Z_{\widetilde{\mathcal K}}$ is affine and the claim follows from Proposition~\ref{prop_stable_orbits}(\ref{prop_stable_orbits_b}).
\end{proof}

\section{Generalities on spherical varieties}
\label{sec_gen_spher}

\subsection{Notation for reductive groups}

In what follows, $G$ will denote a connected reductive algebraic group.
Fix a Borel subgroup $B \subset G$ and a maximal torus $T \subset B$.
There is a unique Borel subgroup $B^-$ of $G$ such that $B \cap B^- = T$; it is said to be opposite to~$B$.
Let $U$ (resp.~$U^-$) denote the unipotent radical of $B$ (resp.~$B^-$); both $U$ and $U^-$ are maximal unipotent subgroups of~$G$.

We identify the groups $\mathfrak X(T)$ and $\mathfrak X(B)$ via restricting characters from $B$ to~$T$.
Similarly, $\mathfrak X(G)$ will be regarded as a subgroup of~$\mathfrak X(T)$.

Let $\Delta \subseteq \mathfrak X(T)$ be the root system of~$G$ with respect to $T$ and let $\Pi \subseteq \Delta$ be the set of simple roots with respect to~$B$.
For every $\alpha \in \Delta$, we let $\alpha^\vee \in \Hom_\ZZ(\mathfrak X(T), \ZZ)$ be the corresponding dual root and let  $U_\alpha \subseteq G$ be the corresponding one-parameter unipotent subgroup.

Let $\Lambda^+ \subseteq \mathfrak X(T)$ be the monoid of dominant weights with respect to~$B$.
Recall that $\Lambda^+$ is in bijection with the (isomorphism classes of) simple finite-dimensional $G$-modules.
Under this bijection, every $\lambda \in \Lambda^+$ corresponds to the simple $G$-module with highest weight~$\lambda$.

\subsection{Spherical varieties and related notions}

Recall that a $G$-variety is said to be \textit{spherical} if it is normal, irreducible, and possesses an open $B$-orbit.

\begin{theorem}[{\cite[Theorem~2]{VK}}] \label{thm_VK}
Let $X$ be a normal irreducible $G$-variety.
The following assertions hold.
\begin{enumerate}[label=\textup{(\alph*)},ref=\textup{\alph*}]
\item
If $X$ is spherical then the $G$-module $\KK[X]$ is multiplicity free.

\item
If the $G$-module $\KK[X]$ is multiplicity free and $X$ is quasi-affine then $X$ is spherical.
\end{enumerate}
\end{theorem}

In what follows we let $X$ be a spherical $G$-variety.
The \textit{weight lattice} (resp. \textit{weight monoid}) of $X$ is the set $M = M(X)$ (resp.~$\Gamma = \Gamma(X)$) consisting of weights of $B$-semiinvariant functions in $\KK(X)$ (resp.~$\KK[X]$).
Clearly, $M$ is a sublattice of $\mathfrak X(T)$ and $\Gamma$ is a submonoid of~$M \cap \Lambda^+$.
When $X$ is quasiaffine, we have $M = \ZZ\Gamma$ (see, for instance,~\cite[Prop.~5.14]{Tim}).
Thanks to Theorem~\ref{thm_VK}, for every $\lambda \in \Gamma$ there is a unique simple $G$-submodule $\KK[X]_\lambda \subseteq \KK[X]$ with highest weight~$\lambda$, and one has the decomposition $\KK[X] = \bigoplus \limits_{\lambda \in \Gamma} \KK[X]_\lambda$.

Since $X$ contains an open $B$-orbit, for every $\lambda \in M$ there exists a unique up to proportionality $B$-semiinvariant rational function $f_\lambda$ on~$X$ of weight~$\lambda$.
Requiring all such functions to take the value~$1$ at a fixed point of the open $B$-orbit, we shall assume that $f_\lambda f_\mu = f_{\lambda + \mu}$ for all $\lambda, \mu \in M$.

Let $M_\QQ$, $N$, $N_\QQ$ be as in \S\,\ref{subsec_Dem_roots}.
Every discrete $\QQ$-valued valuation $v$ of the field $\KK(X)$ vanishing on~$\KK^\times$ determines an element $\varphi(v) \in N_\QQ$ such that $ \langle \varphi(v), \lambda \rangle = v(f_\lambda)$ for all $\lambda \in M$.
It is known (see~\cite[\S\,7.4]{LV} or~\cite[Cor.~1.8]{Kn91}) that the restriction of the map $v \mapsto \varphi(v)$ to the set of $G$-invariant discrete $\QQ$-valued valuations of $\KK(X)$ vanishing on~$\KK^\times$ is injective; we denote its image by~$\mathcal V = \mathcal V(X)$.
Moreover, $\mathcal V \subseteq N_\QQ$ is a finitely generated convex cone of full dimension containing the image of the antidominant Weyl chamber; see~\cite[Prop.~3.2 and Cor.~4.1,~i)]{BriP} or~\cite[Cor.~5.3]{Kn91}.
The cone $\mathcal V$ is called the \textit{valuation cone} of~$X$.

Let $\mathcal D^B = \mathcal D^B(X)$ (resp.~$\mathcal D^G = \mathcal D^G(X)$) denote the set of all $B$-stable (resp.~$G$-stable) prime divisors in~$X$.
Put also $\mathcal D = \mathcal D(X) = \mathcal D^B \setminus \mathcal D^G$; elements of $\mathcal D$ are called \textit{colors} of~$X$.
Every $D \in \mathcal D^B$ defines an element $\varkappa(D) \in N$ by the formula $\langle \varkappa(D), \lambda \rangle = \ord_D(f_\lambda)$ for all $\lambda \in M$.
It follows from the definitions that $\varkappa(\mathcal D^G) \subset \mathcal V$.
Thanks to the normality of~$X$ we have
\begin{equation} \label{eqn_Gamma}
\Gamma = \lbrace \lambda \in M \mid \langle \varkappa(D), \lambda \rangle \ge 0 \ \text{for all} \ D \in \mathcal D^B \rbrace.
\end{equation}

The colors of $X$ can be divided into three types $(U)$, $(T)$, and $(N)$; see \cite[\S\,2]{AZ} for details.

\subsection{Colored fans}

Let $O$ be a spherical homogeneous space for~$G$, that is, a homogeneous spherical $G$-variety.
By an \textit{embedding} of $O$ we mean a spherical $G$-variety $X$ containing $O$ as an open $G$-orbit.
Note that for any embedding $X$ of $O$ there are natural identifications $M(X) = M(O)$, $\mathcal V(X) = \mathcal V(O)$, and $\mathcal D(X) = \mathcal D(O)$.
An embedding $X$ of $O$ is said to be \textit{simple} if $X$ contains exactly one closed $G$-orbit.

A \textit{colored cone} is a pair $(\mathcal C, \mathcal F)$ with $\mathcal C \subset N_\QQ$ and $\mathcal F \subset \mathcal D$ having the following properties:

(CC1) $\mathcal C$ is a cone generated by $\mathcal F$ and finitely many elements of $\mathcal V$.

(CC2) $\mathcal C^\circ \cap \mathcal V \ne \varnothing$.

A colored cone is said to be \textit{strictly convex} if the following property holds:

(SCC) $\mathcal C$ is strictly convex and $0 \notin \varkappa(\mathcal F)$.

Given a simple embedding $X$ of $O$, let $Y \subset X$ be the closed $G$-orbit.
Let $\mathcal C(X) \subset N_\QQ$ be the cone generated by the set $\lbrace \varkappa(D) \mid D \in \mathcal D^B \ \text{with} \ Y \subset D \rbrace$.
Put $\mathcal F(X) = \lbrace D \in \mathcal D \mid Y \subset D \rbrace$.
Then \cite[Thm.~3.1]{Kn91} states (see also \cite[\S\,8.10, Prop.]{LV}) that the map $X \mapsto (\mathcal C(X), \mathcal F(X))$ is a bijection between simple embeddings of $O$ and strictly convex colored cones.

A \textit{face} of a colored cone $(\mathcal C, \mathcal F)$ is a pair $(\mathcal C_0, \mathcal F_0)$ where $\mathcal C_0$ is a face of $\mathcal C$, $\mathcal C_0^\circ \cap \mathcal V \ne \varnothing$, and $\mathcal F_0 = \mathcal F \cap \varkappa^{-1}(\mathcal C_0)$.

A \textit{colored fan} is a nonempty finite collection ${}^c \mathfrak F$ of colored cones with the following properties:

(CF1) every face of a colored cone in ${}^c \mathfrak F$ belongs to~${}^c \mathfrak F$;

(CF2) for every $v \in \mathcal V$ there is at most one colored cone $(\mathcal C, \mathcal F) \in {}^c \mathfrak F$ such that $v \in \mathcal C^\circ$.

A colored fan ${}^c \mathfrak F$ is said to be \textit{strictly convex} if so are all colored cones in~${}^c \mathfrak F$.

Given a spherical $G$-variety~$X$, for every $G$-orbit $Y \subset X$ let $X_Y \subset X$ be the union of all $G$-orbits in $X$ containing $Y$ in their closure.
Then $X_Y$ is a simple $G$-stable subvariety of~$X$.
Let ${}^c \mathfrak F(X)$ be the collection of colored cones $(\mathcal C(X_Y), \mathcal F(X_Y))$ over all $G$-orbits $Y$ in~$X$.

By \cite[Thm.~3.3]{Kn91}, the map $X \mapsto {}^c \mathfrak F(X)$ is a bijection between ($G$-isomorphism classes of) embeddings of $O$ and colored fans in~$N_\QQ$.

Let $X$ be a spherical $G$-variety and let ${}^c \mathfrak F$ be its colored fan in $N_\QQ$.
Let ${}^c \mathfrak F^1$ denote the set of primitive elements $\rho$ of the lattice $N$ such that ${}^c \mathfrak F$ contains a colored cone of the form $(\QQ_{\ge0}\rho, \mathcal F)$ for some subset $\mathcal F \subset \mathcal D$.
Similarly to \cite[Lemma~2.4]{Kn91} one proves the following result.

\begin{proposition} \label{prop_primitive_in_lattice}
The restriction of $\varkappa$ to $\mathcal D^G$ is an injective map to $\mathfrak F^1$ and its image is $\lbrace \rho \in \mathfrak F^1 \mid \QQ_{\ge0}\rho \cap \varkappa(\mathcal D) = \varnothing\rbrace$.
\end{proposition}

\begin{proposition}[{see \cite[Thm.~4.2]{Kn91}}]
The variety $X$ is complete if and only if for every $v \in \mathcal V$ there is a colored cone $(\mathcal C, \mathcal F) \in \mathfrak F$ such that $v \in \mathcal C$.
\end{proposition}

\begin{proposition}[{see \cite[Thm.~6.7]{Kn91}}] \label{prop_aff_crit}
A spherical variety $X$ is affine if and only if $X$ is simple and its colored cone $(\mathcal C, \mathcal F)$ satisfies the following property: there exists $\chi \in M$ such that $\langle \mathcal V, \chi \rangle \le 0$, $\langle \mathcal C, \chi \rangle = 0$, and $\langle \varkappa(D), \chi \rangle > 0$ for all $D \in \mathcal D \setminus \mathcal F$.
\end{proposition}

\subsection{Local structure theorem}
\label{subsec_lst}

Let $X$ be a spherical $G$-variety.

For every subset $\mathcal F \subset \mathcal D$, put $D_{\mathcal F} = \bigcup\limits_{D \in \mathcal F} D$, $X_{\mathcal F} = X \setminus D_{\mathcal F}$ and let $P_\mathcal F$ denote the stabilizer in $G$ of the set~$X_{\mathcal F}$.
Then $P_\mathcal F$ is a parabolic subgroup of~$G$ containing~$B$.
In our study of $B$-root subgroups on spherical varieties a key role is played by the local structure theorem (see~\cite[Thm.~2.3, Prop.~2.4]{Kn94}, \cite[Thm.~1.4]{BLV}), which in our situation may be stated as follows.

\begin{theorem} \label{thm_lst}
Suppose $\mathcal F \subset \mathcal D$ is an arbitrary subset and $P = L \rightthreetimes P_u$ is a Levi decomposition of the group $P = P_\mathcal F$.
Then there exists a closed $L$-stable subvariety $Z \subset X_\mathcal F$ such that the map $P_u \times Z \to X_\mathcal F$ given by the formula $(p,z) \mapsto pz$ is a $P$-equivariant isomorphism, where the action of~$P$ on $P_u \times Z$ is defined by $lu(p,z) = (lupl^{-1}, lz)$ for all $l \in L$, $u,p \in P_u$, $z \in Z$.
Moreover, if $P$ coincides with the stabilizer of the open $B$-orbit in~$X$, then the derived subgroup of~$L$ acts trivially on~$Z$.
\end{theorem}

Below we shall also need the following result.

\begin{proposition}[{\cite[Prop.~2]{AZ}}] \label{prop_stabilizer}
Suppose $\mathcal F = \mathcal D$ or $\mathcal F = \mathcal D \setminus \lbrace D_0 \rbrace$ where $D_0$ is a color of type~$(T)$.
Then the group $P_\mathcal F$ coincides with the stabilizer of the open $B$-orbit in~$X$.
\end{proposition}

Let $\mathcal F = \mathcal D$ or $\mathcal F = \mathcal D \setminus \lbrace D_0 \rbrace$ with $D_0$ being a color of type~$(T)$.

Apply Theorem~\ref{thm_lst} and use the notation $Z, P, L, P_u$ as in that theorem.
Then there is a $P$-equivariant isomorphism $X_{\mathcal F} \simeq P_u \times Z$.
We shall assume $L \supset T$.
Thanks to Proposition~\ref{prop_stabilizer}, we know that the derived subgroup of $L$ acts trivially on~$Z$.
Since $B$ has an open orbit in~$X_{\mathcal F}$, the variety $Z$ contains an open $T$-orbit; we denote it by~$Z_0$.
So $Z$ is a toric $T$-variety.

For every $\lambda \in M$, the restriction of $f_\lambda$ to~$Z$ is a $T$-semiinvariant rational function, which will be still denoted by~$f_\lambda$.
Conversely, every $T$-semiinvariant rational function on~$Z$ trivially extends to a $B$-semiinvariant rational function on~$X_{\mathcal F}$.
Thus $M$ naturally identifies with the weight lattice of $Z$ as a toric $T$-variety.
Let $L_0 \subset L$ be the kernel of the action of $L$ on~$Z$ and put $T_0 = T \cap L_0$.
Then $M$ consists of exactly those characters of $T$ that restrict trivially to~$T_0$.

Every $G$-orbit $O \subset X$ with $ O \cap X_{\mathcal F} \ne \varnothing$ meets $Z$ in a single $T$-orbit.
Let $\mathfrak F(Z)$ denote the fan of $Z$ as a toric $T$-variety.
The next result is straightforward.

\begin{proposition} \label{prop_fan_of_Z}
The following assertions hold.
\begin{enumerate}[label=\textup{(\alph*)},ref=\textup{\alph*}]
\item \label{prop_fan_of_Z_a}
If $\mathcal F = \mathcal D$ then $\mathfrak F(Z) = \lbrace \mathcal C \mid (\mathcal C, \varnothing) \in {}^c \mathfrak F(X)\rbrace$.

\item \label{prop_fan_of_Z_b}
If $\mathcal F = \mathcal D \setminus \lbrace D_0 \rbrace$ then
\(
\mathfrak F(Z) = \lbrace \mathcal C \mid (\mathcal C, \varnothing) \in {}^c \mathfrak F(X) \ \text{or} \ (\mathcal C, \lbrace D_0 \rbrace) \in {}^c \mathfrak F(X) \rbrace.
\)
\end{enumerate}
\end{proposition}

\subsection{Horospherical varieties}

An irreducible $G$-variety $X$ is said to be \textit{horospherical} if the stabilizer of a point in general position in $X$ contains a maximal unipotent subgroup of~$G$.
A normal horospherical $G$-variety is spherical if and only if it contains an open $G$-orbit.
In what follows, by abuse of terminology, all horospherical $G$-varieties will be assumed to be spherical.

Let $X$ be an affine spherical $G$-variety with weight monoid~$\Gamma$ and consider the $G$-module decomposition $\KK[X] \simeq \bigoplus \limits_{\lambda \in \Gamma} \KK[X]_\lambda$.
The following well-known result is deduced from \cite[Thm.~6]{VP72}.

\begin{proposition}
The following conditions are equivalent.
\begin{enumerate}[label=\textup{(\arabic*)},ref=\textup{\arabic*}]
\item
$X$ is horospherical.

\item
$\KK[X]_\lambda \cdot \KK[X]_\mu \subseteq \KK[X]_{\lambda + \mu}$ for all $\lambda, \mu \in \Gamma$.
\end{enumerate}
\end{proposition}

Suppose $X$ is a horospherical $G$-variety.
It is known that in this case all colors are of type~$(U)$.
It turns out that, under the conditions of the local structure theorem (see Theorem~\ref{thm_lst}), the section $Z \subset X_{\mathcal D}$ can be chosen in a canonical way.
More precisely, the next result follows from the construction given in~\cite[\S\,2.4]{Kn94}.

\begin{proposition} \label{prop_Z_hor}
Under the conditions of Theorem~\textup{\ref{thm_lst}}, the section $Z \subset X_{\mathcal D}$ can be chosen as the closure of the $T$-orbit of any $U^-$-fixed point in the open $G$-orbit in~$X$.
\end{proposition}

In what follows, we shall always take the section $Z$ as in the above proposition.

The next result characterizes horospherical varieties in terms of the valuation cone; see, for instance, \cite[Cor.~6.2]{Kn91}.

\begin{proposition} \label{prop_hor_val_cone}
Let $X$ be a spherical $G$-variety with valuation cone~$\mathcal V \subset N_\QQ$.
Then $X$ is horospherical if and only if $\mathcal V = N_\QQ$.
\end{proposition}

In particular, the colored faces of any colored cone $(\mathcal C, \mathcal F)$ are exactly those of the form $(\mathcal C', \mathcal F \cap \varkappa^{-1}(\mathcal C'))$ where $\mathcal C'$ is a face of~$\mathcal C$.

The next result is obtained by combining Propositions~\ref{prop_aff_crit} and~\ref{prop_hor_val_cone}.

\begin{proposition} \label{prop_hor_aff_crit}
A horospherical variety $X$ is affine if and only if $X$ is simple and its colored cone $(\mathcal C, \mathcal F)$ satisfies $\mathcal F = \mathcal D$.
\end{proposition}

\begin{proposition} \label{prop_colors_hor}
For every $D \in \mathcal D$ there is $\alpha \in \Pi$ such that $\langle \varkappa(D), \lambda \rangle = \langle \alpha^\vee, \lambda \rangle$ for all $\lambda \in M$.
\end{proposition}

Let $O$ be a horospherical homogeneous space with weight lattice~$M$ and let ${}^c\mathfrak F$ be a colored fan in~$N_\QQ$.
Let $\mathfrak F$ be the usual fan obtained from ${}^c \mathfrak F$ by taking all cones (without colors).
Let $\mu \in M \cap \Lambda^+$ satisfy conditions \textup(\ref{DR1}\textup) and \textup(\ref{DR2}\textup) for the fan $\mathfrak F$.

\begin{lemma} \label{lemma_faces2}
Suppose two colored cones $(\widetilde{\mathcal K}, \mathcal F_1), (\mathcal K, \mathcal F_2) \in {}^c \mathfrak F$ are such that $\langle \mathcal K, \mu \rangle = 0$ and $\widetilde{\mathcal K}$ is generated by $\mathcal K$ and~$\rho_\mu$.
Then $\mathcal F_1 = \mathcal F_2$.
\end{lemma}

\begin{proof}
It follows from Lemma~\ref{lemma_faces1}(\ref{lemma_faces1_b}) that $\mathcal K$ is a face of $\widetilde{\mathcal K}$, hence $(\mathcal K, \mathcal F_2)$ is a face of~$(\widetilde{\mathcal K}, \mathcal F_1)$ and thus $\mathcal F_2 = \mathcal F_1 \cap \varkappa^{-1}(\mathcal K)$.
Since $\mu \in \Lambda^+$, one has $\varkappa(\mathcal F_1) \subset \mathcal K$ by Proposition~\ref{prop_colors_hor}, which yields $\mathcal F_1 = \mathcal F_2$.
\end{proof}

\begin{lemma} \label{lemma_add_cone2}
Suppose a colored cone $(\mathcal K, \mathcal F) \in {}^c\mathfrak F$ satisfies $\langle \mathcal K, \mu \rangle = 0$ and let $\widetilde{\mathcal K}$ be the cone generated by $\mathcal K$ and~$\rho_\mu$.
Then the collection ${}^c \widetilde{\mathfrak F} = {}^c \mathfrak F \cup \lbrace\text{all faces of} \ (\widetilde{\mathcal K}, \mathcal F) \rbrace$ is a colored fan in $N_\QQ$.
\end{lemma}

\begin{proof}
Let $\widetilde{\mathfrak F}$ be the collection of cones obtained from $\mathfrak F$ by adding all faces of~$\widetilde{\mathcal K}$.
By Lemma~\ref{lemma_add_cone}, $\widetilde{\mathfrak F}$ is a fan in~$N_\QQ$.
Let $\mathcal C$ be a face of $\widetilde{\mathcal K}$ and assume that there are two colored cones $(\mathcal C, \mathcal F_1), (\mathcal C, \mathcal F_2) \in {}^c\widetilde{\mathfrak F}$ such that $(\mathcal C, \mathcal F_1) \in {}^c \mathfrak F$ and $(\mathcal C, \mathcal F_2)$ is a face of $(\widetilde K, \mathcal F)$.
We need to show that $\mathcal F_1 = \mathcal F_2$.
If $\mathcal C \subset \mathcal K$, then both $\mathcal F_1, \mathcal F_2$ are equal to~$\mathcal F \cap \varkappa^{-1}(\mathcal C)$.
If $\rho_\mu \in \mathcal C^1$, then put $\mathcal C_0 = \lbrace v \in \mathcal C \mid \langle v, \mu \rangle = 0 \rbrace$; this is a face of $\mathcal C$ by Lemma~\ref{lemma_faces1}(\ref{lemma_faces1_b}).
By Lemma~\ref{lemma_faces2}, both $(\mathcal C_0, \mathcal F_1)$ and $(\mathcal C_0,\mathcal F_2)$ belong to~${}^c\mathfrak F$.
Observe that $\mathcal C_0 = \mathcal C \cap \mathcal K$ is a face of~$\mathcal K$, hence both $\mathcal F_1, \mathcal F_2$ are equal to $\mathcal F \cap \varkappa^{-1}(\mathcal C_0)$.
\end{proof}

\subsection{The connected automorphism group of a complete spherical variety}
\label{ssec_aut_group_sph}

Let $X$ be a complete spherical $G$-variety.
Without loss of generality we may assume that $G$ acts effectively on~$X$.
Since $B$ has an open orbit in~$X$, it follows that $X$ is a rational variety, hence we find ourselves in the setting of~\S\ref{ssec_aut_groups}.
Recall that the group $A = \Aut(X)^0$ is a linear algebraic group and its Lie algebra~$\mathfrak a$ admits a decomposition~(\ref{eqn_aut_group}) from Proposition~\ref{prop_aut_group_dec}.

It follows from~(\ref{eqn_aut_group}) that, as a $G$-module, $\mathfrak a$ is completely determined by the set of $B$-root subgroups on~$X$ along with the torus $C$ centralizing~$G$.
The action of~$C$ preserves the open $G$-orbit $O \subset X$ and induces a $G$-equivariant automorphism of it.
Let $H$ be the stabilizer in~$G$ of a point in~$O$, so that $O \simeq G/H$.
Then the group of $G$-equivariant automorphisms of $G/H$ is naturally identified with the group $N_G(H)/H$ acting on $G/H$ on the right.
It is known from~\cite[Corollary~5.2]{BriP} that the group $N_G(H)/H$ is diagonalizable.
Put $K = (N_G(H)/H)^0$; then the action of $K$ on $O$ extends to the whole~$X$, so that there is the chain of inclusions $C \subset K \subset G \times C$ of subgroups of~$A$.
More precisely, $K$ is identified with the connected center of~$G \times C$.
As a result, replacing $G$ with $G \times C$ if necessary we may assume that $C$ is trivial and $K$ is the connected center of~$G$.

Now suppose in addition that $X$ is horospherical.
Then we may assume $H \supset U^-$, in which case $N_G(H) = Q$ is a parabolic subgroup of~$G$ containing~$B^-$.
Then the group $K$ is identified with $Q/H \simeq T / (T \cap H)$ and one obtains a natural identification $\mathfrak X(K) \simeq M$.
The latter will be used in~\S\,\ref{ssec_comm_rel_stand}.

\section{\texorpdfstring{$B$}{B}-root subgroups on arbitrary spherical varieties}
\label{sec_basic_prop}

\subsection{First properties of \texorpdfstring{$B$}{B}-root subgroups}
\label{subsec_first_properties}

Throughout this subsection, $X$ is an arbitrary irreducible $G$-variety (not necessarily spherical).
Let $H$ be a $B$-root subgroup on $X$ of weight~$\chi_H$.
The next result is a generalization of~\cite[Proposition~5.1]{AA}.

\begin{proposition} \label{prop_weight_is_dom}
The following assertions hold.
\begin{enumerate}[label=\textup{(\alph*)},ref=\textup{\alph*}]
\item \label{prop_weight_is_dom_a}
$\chi_H \in \Lambda^+$.
\item \label{prop_weight_is_dom_b}
$H$ is $G$-normalized \textup(and hence a $G$-root subgroup on~$X$\textup) if and only if $\chi_H \in \mathfrak X(G)$.
\end{enumerate}
\end{proposition}

\begin{proof}
(\ref{prop_weight_is_dom_a})
Let $\xi$ be the vector field on $X$ induced by the action of~$H$.
Then $\xi$ is a $B$-semiinvariant global section of the tangent sheaf of~$X$.
Since the latter sheaf is coherent and $G$-linearized, its space of global sections is a rational $G$-module; see \cite[Thm.~C.3]{Tim}.
Thus $\xi$ is a highest-weight vector and $\chi_H \in \Lambda^+$.

(\ref{prop_weight_is_dom_b})
The field $\xi$ is $G$-normalized if and only if it generates a one-dimensional $G$-submodule in the space of global sections of the tangent sheaf of~$X$.
The latter condition is equivalent to $\chi_H \in \mathfrak X(G)$.
\end{proof}

\subsection{Vertical and horizontal \texorpdfstring{$B$}{B}-root subgroups}

Starting from this subsection, $X$ is again a spherical $G$-variety.

\begin{definition}
A $B$-root subgroup $H$ on $X$ is said to be \textit{vertical} if $H$ preserves the open $B$-orbit in $X$ and \textit{horizontal} otherwise.
\end{definition}

Let $H$ be a $B$-root subgroup on~$X$.
It follows from the definition that if $H$ is vertical then $HD = D$ for all~$D \in \mathcal D^B$.
On the other hand, if $H$ is horizontal then by Proposition~\ref{prop_moved_divisor} there is exactly one prime divisor $D \in \mathcal D^B$ moved by~$H$.

\begin{proposition}[{\cite[Proposition~1]{AZ}}] \label{prop_moved_types}
Suppose $H$ is horizontal and $D \in \mathcal D^B$ is moved by~$H$.
Then either $D \in \mathcal D^G$ or $D$ is a color of type~\textup($T$\textup).
\end{proposition}

Following \cite{AZ}, in the situation of the above proposition we say that $H$ is \textit{toroidal} (or of toroidal type) if $D \in \mathcal D^G$ and \textit{blurring} (or of blurring type) if $D$ is a color of type~$(T)$.

\subsection{Weights of horizontal \texorpdfstring{$B$}{B}-root subgroups}

Let $H$ be a horizontal $B$-root subgroup of weight~$\mu$ on~$X$ and let $D \in \mathcal D^B$ be moved by~$H$.
Let $\mathcal F = \mathcal D$ if $D \in \mathcal D^G$ or $\mathcal F = \mathcal D \setminus \lbrace D \rbrace$ if $D$ is a color of type~$(T)$.
Recall from Proposition~\ref{prop_primitive_in_lattice} that the element $\varkappa(D)$ is primitive in the lattice~$N$.
Apply Theorem~\ref{thm_lst} and retain all the notation used in that theorem and in~\S\,\ref{subsec_lst}.

Consider the natural projection $X_{\mathcal F} \simeq P_u \times Z \to Z$.
Since $H$ is $P_u$-invariant, it induces a $T$-normalized $\GG_a$-action on~$Z$.
As $D \cap Z \ne \varnothing$, this action is nontrivial and hence we get $T$-root subgroup $H_Z$ on $Z$ of the same weight~$\mu$.
Now Theorem~\ref{thm_T-root_subgroups} and Proposition~\ref{prop_TRS_prop}(\ref{prop_TRS_prop_b}) yields the following result.

\begin{proposition} \label{prop_hor_Dem_root}
Under the above assumptions, $\mu \in \mathfrak R_{\varkappa(D)}(\mathfrak F(Z))$.
In particular, $\mu \in M$.
\end{proposition}

\begin{proposition} \label{prop_pairing_with_mu}
The following assertions hold.
\begin{enumerate}[label=\textup{(\alph*)},ref=\textup{\alph*}]
\item \label{prop_pairing_with_mu_a}
$\langle \varkappa(D), \mu\rangle = -1$.
\item \label{prop_pairing_with_mu_b}
$\langle \varkappa(D'), \mu \rangle \ge 0$ for all $D' \in \mathcal D^B \setminus \lbrace D \rbrace$.
\end{enumerate}
\end{proposition}

\begin{proof}
(\ref{prop_pairing_with_mu_a})
This follows directly from Proposition~\ref{prop_hor_Dem_root}.

(\ref{prop_pairing_with_mu_b})
Put $\rho = \varkappa(D)$ and $\rho' = \varkappa(D')$ for short.
If $\rho' = -\rho$, then the assertion is obvious.
In what follows we assume that $\rho'$ and $\rho$ are not proportional.
Then there exists a weight $\nu \in M$ such that $\langle \rho, \nu \rangle = k > 0$ and $\langle \rho', \nu \rangle = 0$.
Clearly, $\mathrm{ord}_{D'} (f_\nu) = \langle \rho', \nu \rangle = 0$.
Fix an isomorphism $\GG_a \xrightarrow{\sim} H$, $s \mapsto H(s)$.
By Proposition~\ref{prop_TRS_prop}(\ref{prop_TRS_prop_c}), $H_Z$ (and hence~$H$) acts on the functions $f_\lambda$ with $\lambda \in M$ by formula~(\ref{fla_action}).
Since the divisor $D'$ is $H$-stable, for all $s \in \KK$ we have $\mathrm{ord}_{D'} (H(s)\cdot f_\nu) = 0$.
Now assume that $\langle \rho, \mu \rangle < 0$.
Then for all $s \ne 0$ we have $\mathrm{ord}_{D} (H(s)\cdot f_\nu) = k \cdot \mathrm{ord}_{D'} (1+csf_\mu) = k\langle \rho,\mu\rangle < 0$, a contradiction.
\end{proof}

\subsection{Local description of \texorpdfstring{$B$}{B}-root subgroups}
\label{subsec_local_descr}

Let $X$ be a spherical $G$-variety and let $H$ be a $B$-root subgroup on~$X$.
Put $\mathcal F_H = \lbrace D \in \mathcal D \mid HD = D\rbrace$.
Recall that $\mathcal F_H = \mathcal D$ in the case of vertical or toroidal $H$ and $\mathcal F_H = \mathcal D \setminus \lbrace D_0 \rbrace$ in the case of blurring $H$ moving a color $D_0$ of type~$(T)$.
Then $H$ preserves the open subset $X_{\mathcal F_H} \subset X$.

Now let $\mathcal F = \mathcal D$ or $\mathcal F = \mathcal D \setminus \lbrace D_0 \rbrace$ with $D_0$ being a color of type~$(T)$.
We shall describe all $B$-root subgroups on~$X_{\mathcal F}$.

Apply Theorem~\ref{thm_lst} and retain all the notation used in that theorem and in \S\,\ref{subsec_lst}.
Recall that there is a $P$-equivariant isomorphism $X_{\mathcal F} \simeq P_u \times Z$ where $Z$ is a toric $T$-variety whose fan is described in Proposition~\ref{prop_fan_of_Z}.
Let $\Gamma_Z \subset M$ be the weight monoid of~$Z$ and fix an arbitrary point $z_0 \in Z_0$.
We shall also assume that $f_\lambda(z_0) = 1$ for all $\lambda \in M$.

Consider the adjoint representation of the group $L$ on the space $\mathfrak p_u = \mathrm{Lie}(P_u)$ and decompose $\mathfrak p_u$ into a direct sum of irreducible $L$-invariant subspaces.
It is well known (see~\cite[Thm.~1.9]{Kos} or \cite[Ch.~3, Lemma~3.9]{GOV}) that all summands in this decomposition are pairwise non-isomorphic as $L$-modules; let $\Omega \subset \Delta$ be the set of highest weights of these summands with respect to the Borel subgroup $B \cap L \subset L$.
For each $\alpha \in \Omega$ fix a nonzero vector $e_\alpha \in \mathfrak p_u$ of weight~$\alpha$ and let $\varepsilon_\alpha$ be the vector field on $P_u$ determined by the action of the group~$\lbrace \exp(te_\alpha) \mid t \in \KK \rbrace$ on the right.
We naturally extend this vector field to $P_u \times Z$.

For each character $\mu \in \mathfrak X(T)$ put
\begin{equation}
\Omega_\mu = \lbrace \alpha \in \Omega \mid \left.\mu\right|_{T_0} = \left.\alpha\right|_{T_0} \rbrace; \quad \Omega_\mu^0 : = \lbrace \alpha \in \Omega_\mu \mid \mu - \alpha \in \Gamma_Z \rbrace.
\end{equation}
Note that the condition
$\left.\mu\right|_{T_0} = \left.\alpha\right|_{T_0}$ is equivalent to $\mu - \alpha \in M$.

Below by a $\GG_a$-integrable vector field we mean a vector field induced by a $\GG_a$-action.

\begin{theorem} \label{thm_LNDs_on_X0}
Given $\mu \in \mathfrak X(T)$, every $B$-normalized $\GG_a$-integrable vector field of weight $\mu$ on $P_u \times Z$ has the form
\begin{equation} \label{eqn_LNDs_on_X0}
\sum_{\alpha \in \Omega_\mu^0} c_\alpha f_{\mu - \alpha} \varepsilon_\alpha + \xi_Z
\end{equation}
where $c_\alpha \in \KK$ and $\xi_Z$ is a $T$-normalized $\GG_a$-integrable vector field of weight $\mu$ on~$Z$ extended naturally to~$P_u \times Z$.
Conversely, every vector field on $P_u \times Z$ of the above form is $B$-normalized of weight~$\mu$ and $\GG_a$-integrable.
\end{theorem}

\begin{proof}
Suppose $\xi$ is a $B$-normalized $\GG_a$-integrable vector field of weight~$\mu$ on $P_u \times Z$ and let $H$ be the corresponding $B$-root subgroup.
Since $\xi$ is $P_u$-invariant, the natural projection $P_u \times Z \to Z$ induces a well-defined pushforward $\xi_Z$ of $\xi$ to~$Z$.
In what follows we naturally extend $\xi_Z$ to $P_u \times Z$.
Consider the vector field $\xi - \xi_Z$ on $P_u \times Z$.
Since $B$ acts transitively on $P_u \times Z_0$, $\xi-\xi_Z$ is uniquely determined by its value $v$ at the point $(e,z_0)$ where $e \in P_u$ is the identity element.
Note that the tangent space to $P_u \times Z_0$ at $(e,z_0)$ is naturally identified with $\mathfrak p_u \oplus T_{z_0}Z_0$.
Since the pushforward of $\xi - \xi_Z$ to~$Z$ is trivial, it follows that $v$ is a $B \cap L_0$-semiinvariant vector in~$\mathfrak p_u$ of weight $\left.\mu\right|_{T_0}$, therefore $v = \sum\limits_{\alpha \in \Omega_\mu} c_\alpha e_\alpha$ for some $c_\alpha \in \KK$.
On the other hand, observe that the vector field $\sum\limits_{\alpha \in \Omega_\mu} c_\alpha f_{\mu - \alpha}\varepsilon_\alpha$ on $P_u \times Z_0$ is also $B$-semiinvariant of weight~$\mu$ and corresponds to the same tangent vector at $(e,z_0)$, hence it coincides with~$\xi - \xi_Z$.
Next, since this vector field extends to $P_u \times Z$, the condition $c_\alpha = 0$ should hold for all $\alpha \in \Omega_\mu$ with $\mu - \alpha \notin \Gamma_Z$, which proves the first claim.

Now suppose $\xi$ is a vector field on $P_u \times Z$ of the form~(\ref{eqn_LNDs_on_X0}).
Then $\xi$ is automatically $B$-normalized of weight~$\mu$, and it remains to prove that $\xi$ is $\GG_a$-integrable.
If $\xi_Z = 0$ then, by \cite[Thm.~3]{AZ}, $\xi$ is $\GG_a$-integrable on any subset of the form $P_u \times Z'$ where $Z' \subset Z$ is an affine open $T$-stable subset, hence $\xi$ is $\GG_a$-integrable on the whole $P_u \times Z$.
In what follows we assume that $\xi_Z \ne 0$.
Then $\mu \in \mathfrak R(\mathfrak F(Z))$.
Let $\rho_\mu \in \mathfrak F^1(Z)$ be the element such that $\langle \rho_\mu, \mu \rangle = -1$.
For every $\mathcal C \in \mathfrak F(Z)$, let $Z_{\mathcal C} \subset Z$ be the corresponding $T$-stable affine open subset.
For the cone $\mathcal C_0 = \QQ_{\ge0}\rho_\mu$, we know from \cite[Thm.~3]{AZ} that $\xi$ integrates to a $\GG_a$-action on $P_u \times Z_{\mathcal C_0}$.
Then we get a $\GG_a$-action on the field $\KK(P_u \times Z)$.
Clearly, $\xi$ (regarded as a derivation) preserves the algebra $\KK[Z_{\mathcal C_0}]$, hence by formula~(\ref{eqn_dmu_int}) there is a constant $c \in \KK^\times$ such that the $\GG_a$-action on $f_\lambda$ is given by the formula
\begin{equation} \label{fla_action2}
(s, f_\lambda) \mapsto f_\lambda(1 + scf_\mu)^{\langle \rho_\mu, \lambda \rangle}
\end{equation}
for all $s \in \KK$ and $\lambda \in M$.
Next, $\xi$ preserves the algebra $\KK[P_u] \otimes \KK[Z]$, hence so does the $\GG_a$-action and we get the natural algebra homomorphism
\begin{equation} \label{eqn_algebra_hom}
\KK[P_u] \otimes \KK[Z] \to \KK[\AA^1] \otimes \KK[P_u] \otimes \KK[Z].
\end{equation}

We now take an arbitrary cone $\mathcal C \in \mathfrak F$ and show that the above $\GG_a$-action extends to $P_u \times Z_{\mathcal C}$.
More precisely, we shall show that the $\GG_a$-action extends to a morphism $\AA^1 \times P_u \times Z_{\mathcal C} \to P_u \times Z$.

Case 1: $\rho_\mu \in \mathcal C^1$.
Then $\xi$ integrates to a $\GG_a$-action on $P_u \times Z_{\mathcal C}$ again by \cite[Thm.~3]{AZ}.

Case 2: $\rho_\mu \notin \mathcal C^1$.
Note that $\langle \rho, \mu \rangle \ge 0$ for all $\rho \in \mathcal C^1$.
Let $\mathcal K$ be the cone generated by the set $\lbrace \rho \in \mathcal C^1 \mid \langle \rho, \mu \rangle = 0 \rbrace$.
Then $\mathcal K$ is a face of~$\mathcal C$.
Let $\mathcal B$ be the cone spanned by $\mathcal K$ and~$\rho_\mu$.
Since $\mu$ is a Demazure root of~$\mathfrak F(Z)$, one has $\mathcal B \in \mathfrak F(Z)$.
Consider the functions $g = 1+scf_\mu$ and $h = f_\mu$ on $\AA^1 \times P_u \times Z_{\mathcal C}$ ($s$ is regarded as a coordinate function on~$\AA^1$) and let $(\AA^1 \times P_u \times Z_{\mathcal C})_g$ and $(\AA^1 \times P_u \times Z_{\mathcal C})_h$ be the corresponding principal open subsets.
Now the restriction of the homomorphism in~(\ref{eqn_algebra_hom}) to $\KK[P_u] \otimes \KK$ along with the formula
\[f_\lambda \mapsto f_\lambda(1+sf_\mu)^{\langle \rho_\mu, \lambda \rangle} = f_\lambda g^{\langle \rho_\mu, \lambda \rangle}
\]
arising from~(\ref{fla_action2}) defines algebra homomorphisms
\[
\KK[P_u \times Z_{\mathcal C}] \to \KK[(\AA^1 \times P_u \times Z_{\mathcal C})_g] = \KK[\AA^1 \times P_u \times Z_{\mathcal C}][g^{-1}]
\]
and
\[
\KK[P_u \times Z_{\mathcal B}] \to \KK[(\AA^1 \times P_u \times Z_{\mathcal C})_h] = \KK[\AA^1 \times P_u \times Z_{\mathcal C}][h^{-1}]
\]
(in the second case, $f_\lambda \cdot f_\mu^N$ is regular on $Z_{\mathcal C}$ for a sufficiently large power~$N$), which in turn define morphisms $(\AA^1 \times P_u \times Z_{\mathcal C})_g \to P_u \times Z_{\mathcal C}$ and $(\AA^1 \times P_u \times Z_{\mathcal C})_h \to P_u \times Z_{\mathcal B}$.
Since $g - sh = 1$, we have $(\AA^1 \times P_u \times Z_{\mathcal C})_g \cup (\AA^1 \times P_u \times Z_{\mathcal C})_h  = \AA^1 \times P_u \times Z_{\mathcal C}$ and thus the two morphisms in fact glue together to a morphism $\AA^1 \times P_u \times Z_{\mathcal C} \to P_u \times (Z_{\mathcal C} \cup Z_{\mathcal B})$, which extends the $\GG_a$-action on $P_u \times Z_{\mathcal C_\mu}$.
\end{proof}

\begin{corollary}
For every $\mu \in \mathfrak X(T)$, all $B$-normalized $\GG_a$-integrable vector fields of weight $\mu$ on $P_u \times Z$ form a vector space of dimension~$|\Omega_\mu^0| + \delta(\mu)$, where $\delta(\mu) = 1$ if $\mu \in \mathfrak R(\mathfrak F(Z))$ and $\delta(\mu) = 0$ otherwise.
\end{corollary}

\begin{remark} \label{rem_horizontal}
It follows from the above proof that a $B$-root subgroup on $X$ corresponding to the vector field~(\ref{eqn_LNDs_on_X0}) on~$X_{\mathcal F}$ is horizontal if and only if $\xi_Z \ne 0$.
\end{remark}

\section{Standard \texorpdfstring{$B$}{B}-root subgroups in the horospherical case}
\label{sec_hor_stand}

\subsection{The affine case}
\label{ssec_stand_aff}

Let $X$ be an affine horospherical $G$-variety with weight monoid~$\Gamma$.
For every $\lambda \in \Gamma$ let $\KK[X]_\lambda'$ be the $T$-stable complement of $\KK f_\lambda$ in $\KK[X]_\lambda$, so that $\KK[X]_\lambda = \KK f_\lambda \oplus \KK[X]_\lambda'$.
Then $I = \bigoplus \limits_{\lambda \in \Gamma} \KK[X]_\lambda'$ is a $U^-$-stable ideal in $\KK[X]$.
Let $\widetilde Z \subset X$ be the closed subvariety corresponding to~$I$.
Since $U^-$ acts trivially on $\KK[X]/I$, the variety $\widetilde Z$ consists of $U^-$-fixed points.
It follows that $Z = \widetilde Z \cap X_{\mathcal D}$.

Let $\mathcal E \subset N_\QQ$ be the cone dual to~$\QQ_{\ge0}\Gamma$.
By Proposition~\ref{prop_hor_aff_crit}, $X$ is simple with colored cone of the form $(\mathcal E, \mathcal D)$.

Take any $\mu \in \mathfrak R(\mathcal E) \cap \Lambda^+$ and let $\rho_\mu \in \mathcal E^1$ be the element with $\langle \rho_\mu, \mu \rangle = -1$.
It was proved in~\cite[\S\,6.1]{AA} that the map $\partial_\mu \colon \KK[X] \to \KK[X]$ given by the formula $\partial_\mu(g) = \langle \rho, \lambda \rangle f_\mu g$ for all $\lambda \in \Gamma$ and $g \in \KK[X]_\lambda$ is a $B$-normalized LND of weight~$\mu$.
The corresponding $B$-root subgroup on~$X$ is said to be \textit{standard}.

\begin{proposition} \label{prop_stand_char}
For a $B$-root subgroup~$H$ on~$X$, the following conditions are equivalent.
\begin{enumerate}[label=\textup{(\arabic*)},ref=\textup{\arabic*}]
\item \label{prop_stand_char_1}
$H$ is standard.
\item \label{prop_stand_char_2}
The varieties~$\widetilde Z$ and~$Z$ are preserved by~$H$.
\end{enumerate}
\end{proposition}

\begin{proof}
(\ref{prop_stand_char_1}) $\Rightarrow$ (\ref{prop_stand_char_2})
This follows from the fact that the corresponding LND $\partial$ preserves the ideal $I \subset \KK[X]$.

(\ref{prop_stand_char_2}) $\Rightarrow$ (\ref{prop_stand_char_1})
Since $H$ preserves~$Z$, it is horizontal, and thus its weight $\mu$ belongs to $\mathfrak R(\mathcal E) \cap \Lambda^+$.
Then there exists a standard $B$-root subgroup $H'$ on~$X$ of weight~$\mu$, and it preserves $\widetilde Z$ and~$Z$.
The vector fields of~$H$ and~$H'$ on~$Z$ are necessarily proportional, hence $H = H'$.
\end{proof}

Let $\mu$ and $\rho_\mu$ be as above and let $\mathcal L_\mu$ denote the simple $G$-module in $\Der \KK[X]$ generated by~$\partial_\mu$.
All derivations in $\mathcal L_\mu$ (which are not necessarily $B$-semi-invariant in general) admit a simple description as follows.
Let $O$ be the open $G$-orbit in~$X$; then $O$ is a quasi-affine horospherical homogeneous space.
Let $\Gamma_O$ denote the weight monoid of~$O$ and consider the $G$-module decomposition $\KK[O] = \bigoplus \limits_{\lambda \in \Gamma_O} \KK[O]_\lambda$.
Observe that $\mu \in \Gamma_O$ and $f_\mu \in \KK[O]_\mu$.
There exists a unique isomorphism $\KK[O]_\mu \xrightarrow{\sim} \mathcal L_\mu$ sending $f_\mu$ to~$\partial_\mu$.
Under this isomorphism, a function $f \in \KK[O]_\mu$ corresponds to the derivation $\partial_f \in \Der \KK[X]$ such that for all $\lambda \in \Gamma$ and $g \in \KK[X]_\lambda$ one has $\partial_f(g) = \langle \rho_\mu, \lambda \rangle f g$.
We note that the derivation $\partial_f$ is locally nilpotent.

By \cite[Theorem~8]{VP72}, $G$-orbits in~$X$ are in bijection with faces of~$\mathcal E$.
For every face $\mathcal C$ of~$\mathcal E$, let $O_{\mathcal C}$ be the corresponding $G$-orbit in~$X$, $\overline O_{\mathcal C}$ its closure in~$X$, and $\Gamma_{\mathcal C}$ the intersection of $\Gamma$ with the face of $\QQ_{\ge0} \Gamma$ dual to~$\mathcal C$.
Then the ideal in $\KK[X]$ defining $\overline O_{\mathcal C}$ is $I_{\mathcal C} = \bigoplus \limits_{\lambda \in \Gamma \setminus \Gamma_{\mathcal C}} \KK[X]_\lambda$.
If $\mathcal C = \QQ_{\ge0} \rho$ for some $\rho \in \mathcal E^1$ then we shall write $O_\rho, \overline O_\rho$, $\Gamma_\rho$, $I_\rho$ instead of $O_{\mathcal C}$, $\overline O_{\mathcal C}$, $\Gamma_{\mathcal C}$, $I_{\mathcal C}$, respectively.

Take any $f \in \KK[O]_\mu$ and let $H$ be the $\GG_a$-subgroup on~$X$ corresponding to the LND~$\partial_f$.
Propositions~\ref{prop_stable_oc_G} and~\ref{prop_stable_orbits_G} below are particular cases of Propositions~\ref{prop_stable_ideals} and~\ref{prop_stable_subvar}, respectively.

\begin{proposition} \label{prop_stable_oc_G}
Given a face $\mathcal C'$ of~$\mathcal C$, the following assertions hold.
\begin{enumerate}[label=\textup{(\alph*)},ref=\textup{\alph*}]
\item \label{prop_stable_oc_G_a}
If there is $\rho \in \mathcal C'^1$ such that $\langle \rho, \mu \rangle > 0$, then $\overline O_{\mathcal C'}$ is pointwise fixed by~$H$.
\item \label{prop_stable_oc_G_b}
If $\langle \mathcal C', \mu \rangle = 0$, then $\overline O_{\mathcal C'}$ is $H$-stable with nontrivial $T$-action.
\item \label{prop_stable_oc_G_c}
If $\langle \mathcal C', \mu \rangle \le 0$ and $\rho_\mu \in \mathcal C'^1$, then $\overline O_{\mathcal C'}$ is $H$-unstable.
\end{enumerate}
\end{proposition}

\begin{proposition} \label{prop_stable_orbits_G}
Given a face $\mathcal C'$ of~$\mathcal C$, the following assertions hold.
\begin{enumerate}[label=\textup{(\alph*)},ref=\textup{\alph*}]
\item \label{prop_stable_orbits_G_a}
If there is $\rho \in \mathcal C'^1$ such that $\langle \rho, \mu \rangle > 0$, then $O_{\mathcal C'}$ is pointwise fixed by~$H$.

\item \label{prop_stable_orbits_G_b}
If $\langle \mathcal C', \mu \rangle \le 0$, then there exist faces $\mathcal K, \widetilde{\mathcal K}$ of $\mathcal C$ such that $\langle \mathcal K, \mu \rangle = 0$, $\widetilde{\mathcal K}$ is generated by $\mathcal K$ and~$\rho_\mu$, $\mathcal C' \in \lbrace \mathcal K, \widetilde{\mathcal K}\rbrace$, and $HO_{\mathcal C'} = O_{\mathcal K} \cup O_{\widetilde{\mathcal K}}$.
\end{enumerate}
\end{proposition}

\subsection{Quotient realization of a simple horospherical variety}
\label{subsec_simple_horosph}

In this subsection, we describe a construction that will be useful later.

Let $X$ be a simple horospherical $G$-variety with colored cone $(\mathcal C, \mathcal F)$ and let $\mu \in \mathfrak R(\mathcal C) \cap \Lambda^+$.

Since $X$ is normal, we know that $X$ is quasiprojective and hence can be realized as a locally closed $G$-stable subvariety in the projectivization $\PP(V)$ of a finite-dimensional $G$-module~$V$.
Let $\varphi \colon V \setminus \lbrace 0 \rbrace \to \PP(V)$ be the canonical projection.
Consider the group $\widetilde G = G \times C$ where $C \simeq \KK^\times$ acts on~$V$ via scalar transformations.
Put also $\widetilde B = B \times C$ and $\widetilde T = T \times C$, so that $\widetilde B$ is a Borel subgroup of $\widetilde G$ and $\widetilde T$ is a maximal torus in~$\widetilde B$.
Put $\widetilde X = \varphi^{-1}(X)$; this is a quasiaffine spherical (and even horospherical) $\widetilde G$-variety.
By~\cite[Thm.]{Kn93}, the algebra $\KK[\widetilde X]$ is finitely generated, so $\widehat X = \Spec \KK[\widetilde X]$ is an affine spherical $\widetilde G$-variety with $\KK[\widehat X] = \KK[\widetilde X]$.
Then we have a natural $\widetilde G$-equivariant embedding $\widetilde X \hookrightarrow \widehat X$ with boundary of codimension $\ge2$.

Let $\widetilde M, \widetilde N, \widetilde \varkappa, \ldots$ be the corresponding objects relative to~$\widetilde X$ (and also to~$\widehat X$).
Then the map $\varphi$ induces a natural inclusion $M \hookrightarrow \widetilde M$ and the corresponding surjective map $\widetilde N \to N$.
For every $D \in \mathcal D^B$, put $\widetilde D = \varphi^{-1}(D)$ and let $\widehat D$ be the closure of $\widetilde D$ in~$\widehat X$.
Then each $\widetilde D$ is a $\widetilde B$-stable prime divisor in~$\widetilde X$; moreover, $\widetilde D$ is a color of $\widetilde X$ if and only if $D$ is a color of~$X$.
Then for every $D \in \mathcal D^B$ and $\lambda \in M$ one has $\langle \widetilde \varkappa(\widetilde D), \lambda \rangle = \langle \varkappa(D), \lambda \rangle$.

Clearly, the map $\varphi$ induces a bijection between the $G$-orbits in $X$ and the $\widetilde G$-orbits in~$\widetilde X$; moreover, this bijection respects the inclusion of orbit closures.
It follows that $\widetilde X$ is a simple spherical $\widetilde G$-variety corresponding to the colored cone $(\widetilde{\mathcal C}, \widetilde{\mathcal F})$ where $\widetilde{\mathcal C} = \QQ_{\ge0} \lbrace \widetilde \varkappa(\widetilde D) \mid D \in \mathcal D^G \cup \mathcal F \rbrace$ and $\widetilde{\mathcal F} = \lbrace \widetilde D \mid D \in \mathcal F \rbrace$.

Since $\widehat X$ is affine and horospherical, it follows from Proposition~\ref{prop_hor_aff_crit} that $\widehat X$ is simple and its colored cone is of the form $(\widetilde{\mathcal E}, \widetilde{\mathcal D})$ where $\widetilde{\mathcal E}$ is the cone generated by the set $\lbrace \widetilde \varkappa(\widetilde D) \mid D \in \mathcal D^B \rbrace$.

Let $\widehat O$ be the open $\widetilde G$-orbit in~$\widehat X$ (and also in~$\widetilde X$).
For every $\lambda \in \Gamma(\widehat O)$ let $\KK[\widehat O]_{\lambda} \subset \KK[\widehat O]$ be the simple $\widetilde G$-submodule with highest weight~$\lambda$.

In view of Proposition~\ref{prop_colors_hor} one has $\mathfrak R(\mathcal C) \cap \Lambda^+ \subset \mathfrak R(\widetilde{\mathcal E})$.

Take any $\mu \in \mathfrak R(\mathcal C) \cap \Lambda^+$ and any nonzero function $f \in \KK[\widehat O]_\mu$.
Consider the LND $\partial_f$ on $\KK[\widehat X]$ as in \S\,\ref{ssec_stand_aff} and let $\widehat H$ be the $\GG_a$-subgroup on $\widetilde X$ corresponding to~$\partial_f$.

\begin{proposition} \label{prop_stable_subset}
The following assertions hold.
\begin{enumerate}[label=\textup{(\alph*)},ref=\textup{\alph*}]
\item \label{prop_stable_subset_a}
The subset $\widetilde X \subset \widehat X$ is $\widehat H$-stable.
\item \label{prop_stable_subset_b}
The action of $\widehat H$ on $\widetilde X$ descends to a $\GG_a$-subgroup $H$ on~$X$.
\end{enumerate}
\end{proposition}

\begin{proof}
(\ref{prop_stable_subset_a})
Put $\mathcal D^{\widetilde B}_0 = \lbrace \widetilde D \mid D \in \mathcal D \setminus \mathcal F \ \text{and} \ \QQ_{\ge0}\widetilde{\varkappa}(\widetilde D) \ \text{is a face of} \ \widetilde{\mathcal E} \rbrace$.
For every $\widetilde D \in \mathcal D^{\widetilde B}_0$, let $O_{\widetilde D}$ denote the $\widetilde G$-orbit in $\widehat X$ corresponding to the face $\QQ_{\ge0}\widetilde D$ of the cone~$\widetilde{\mathcal E}$ and let $\overline O_{\widetilde D}$ be the closure of $O_{\widetilde D}$ in $\widehat X$.
Then $\widetilde X = \widehat X \setminus \bigcup \limits_{\widetilde D \in \mathcal D^{\widetilde B}_0} \overline O_{\widetilde D}$.
By Proposition~\ref{prop_stable_oc_G}(\ref{prop_stable_oc_G_a},\,\ref{prop_stable_oc_G_b}), each subset $\overline O_{\widetilde D}$ is $H$-stable, hence so is~$\widetilde X$.

(\ref{prop_stable_subset_b})
Since $\lambda \in \Lambda^+$, it follows that $\widehat H$ commutes with~$C$, whence the claim.
\end{proof}

\subsection{The general case}

Let $X$ be an arbitrary horospherical $G$-variety (not necessarily affine).

\begin{definition}
A $B$-root subgroup on~$X$ is said to be \textit{standard} if it preserves the canonical section $Z \subset X_{\mathcal D}$.
\end{definition}

By Remark~\ref{rem_horizontal}, any standard $B$-root subgroup on~$X$ is automatically horizontal and uniquely determined by its weight.

\begin{proposition} \label{prop_hor_simple}
Suppose $X$ is simple with colored cone $(\mathcal C, \mathcal F)$ and $\mu \in \mathfrak X(T)$.
Then the following conditions are equivalent.
\begin{enumerate}[label=\textup{(\arabic*)},ref=\textup{\arabic*}]
\item \label{prop_hor_simple_1}
There exists a standard $B$-root subgroup on $X$ of weight~$\mu$.
\item \label{prop_hor_simple_2}
$\mu \in \mathfrak R(\mathcal C) \cap \Lambda^+$.
\end{enumerate}
\end{proposition}

\begin{proof}
(\ref{prop_hor_simple_1})$\Rightarrow$(\ref{prop_hor_simple_2})
Let $H$ be a standard $B$-root subgroup on~$X$ of weight~$\mu$.
Then~$\mu \in \Lambda^+$ by Proposition~\ref{prop_weight_is_dom}(\ref{prop_weight_is_dom_a}).
Since $H$ is horizontal, by Proposition~\ref{prop_moved_divisor} there is a unique prime divisor $D \in \mathcal D^B$ moved by~$H$.
Thanks to Proposition~\ref{prop_moved_types} and and the fact that $X$ has no colors of type~$(T)$, one actually has $D \in \mathcal D^G$.
Then $\mu \in \mathfrak R(\mathcal C)$ by Proposition~\ref{prop_pairing_with_mu}.

(\ref{prop_hor_simple_2})$\Rightarrow$(\ref{prop_hor_simple_1})
Retain all the notation of \S\,\ref{subsec_simple_horosph} and consider the standard $\widetilde B$-root subgroup on $\widehat X$ of weight~$\lambda$.
By Proposition~\ref{prop_stable_subset}, $\widehat H$ preserves $\widetilde X$ and descends to a $B$-root subgroup~$H$ on~$X$ of the same weight~$\mu$.
Clearly, the set of $U^-$-fixed points in $\widetilde X$ is $\varphi^{-1}(Z)$, hence $Z$ is $H$-stable and $H$ is standard.
\end{proof}

The next corollary is implied by Proposition~\ref{prop_stable_orbits_G}.

\begin{corollary} \label{cor_stable_orbits_G1}
Retain the hypotheses of Proposition~\textup{\ref{prop_hor_simple}}.
Given a face $\mathcal C'$ of~$\mathcal C$, the following assertions hold.
\begin{enumerate}[label=\textup{(\alph*)},ref=\textup{\alph*}]
\item \label{cor_stable_orbits_G1_a}
If there is $\rho \in \mathcal C'^1$ such that $\langle \rho, \mu \rangle > 0$, then $O_{\mathcal C'}$ is pointwise fixed by~$H$.

\item \label{cor_stable_orbits_G1_b}
If $\langle \mathcal C', \mu \rangle \le 0$, then there exist faces $\mathcal K, \widetilde{\mathcal K}$ of $\mathcal C$ such that $\langle \mathcal K, \mu \rangle = 0$, $\widetilde{\mathcal K}$ is generated by $\mathcal K$ and~$\rho_\mu$, $\mathcal C' \in \lbrace \mathcal K, \widetilde{\mathcal K}\rbrace$, and $HO_{\mathcal C'} = O_{\mathcal K} \cup O_{\widetilde{\mathcal K}}$.
\end{enumerate}
\end{corollary}

\begin{corollary}
Under the hypotheses of Proposition~\textup{\ref{prop_hor_simple}}, take any $D \in \mathcal D^G$.
The following conditions are equivalent.
\begin{enumerate}[label=\textup{(\arabic*)},ref=\textup{\arabic*}]
\item
There exists a $B$-root subgroup on~$X$ that moves~$D$.

\item
$\mathfrak R_{\varkappa(D)}(\mathcal C) \cap \Lambda^+ \ne \varnothing$.
\end{enumerate}
\end{corollary}

We now turn to the case of arbitrary horospherical $X$ (not necessarily simple).
Let ${}^c\mathfrak F(X)$ be the colored fan of $X$ and
let $\mathfrak F(X)$ be the fan obtained from ${}^c\mathfrak F(X)$ by taking all cones \textup(without colors\textup).
Observe that $\mathfrak F(Z) \subset \mathfrak F(X)$ by Proposition~\ref{prop_fan_of_Z}(\ref{prop_fan_of_Z_a}).

\begin{lemma} \label{lemma_DR_hor}
Let $\mu \in \mathfrak R(\mathfrak F(X)) \cap \Lambda^+$.
\begin{enumerate}[label=\textup{(\alph*)},ref=\textup{\alph*}]
\item \label{lemma_DR_hor_a}
If a colored cone $(\mathcal K, \mathcal F) \in {}^c\mathfrak F(X)$ satisfies $\langle \mathcal K, \mu \rangle = 0$ and $\widetilde{\mathcal K}$ is the cone generated by $\mathcal K$ and~$\rho_\mu$, then $(\widetilde{\mathcal K}, \mathcal F) \in {}^c\mathfrak F(X)$.
\item \label{lemma_DR_hor_b}
$\mu \in \mathfrak R(\mathfrak F(Z)) \cap \Lambda^+$.
\end{enumerate}
\end{lemma}

\begin{proof}
(\ref{lemma_DR_hor_a})
Since $\mu \in \mathfrak R(\mathfrak F(X))$, it follows from~(\ref{DR3}) that $\widetilde{\mathcal K} \in \mathfrak F(X)$, hence $(\widetilde{\mathcal K}, \mathcal F') \in {}^c\mathfrak F(X)$ for some $\mathcal F' \subset \mathcal D$.
Then Lemma~\ref{lemma_faces2} implies $\mathcal F' = \mathcal F$.

(\ref{lemma_DR_hor_b})
It suffices to show that $\mu \in \mathfrak R(\mathfrak F(Z))$.
Let $D_\mu \in \mathcal D^B$ be the divisor corresponding to~$\rho_\mu \in \mathfrak F^1(X)$.
By Proposition~\ref{prop_colors_hor}, one actually has $D_\mu \in \mathcal D^G$, whence~(\ref{DR1}).
As $\mathfrak F^1(Z) \subset \mathfrak F^1(X)$, property~(\ref{DR2}) also holds.
Let $\mathcal K \in \mathfrak F(Z)$ be a cone such that $\langle \mathcal K, \mu \rangle = 0$ and let $\widetilde{\mathcal K}$ be the cone generated by $\mathcal K$ and~$\rho_\mu$.
Since $(\mathcal K, \varnothing) \in \mathfrak F^c(X)$, one has $(\widetilde{\mathcal K}, \varnothing) \in {}^c\mathfrak F(X)$ by part~(\ref{lemma_DR_hor_a}), hence $\widetilde{\mathcal K} \in \mathfrak F(Z)$ and we get~(\ref{DR3}).
\end{proof}

\begin{proposition} \label{prop_standard_gen}
For a weight $\mu \in \mathfrak X(T)$, the following conditions are equivalent.
\begin{enumerate}[label=\textup{(\arabic*)},ref=\textup{\arabic*}]
\item \label{prop_standard_gen_1}
There exists a standard $B$-root subgroup on~$X$ of weight~$\mu$.
\item \label{prop_standard_gen_2}
$\mu \in \mathfrak R(\mathfrak F(X)) \cap \Lambda^+$.
\end{enumerate}
\end{proposition}

\begin{proof}
(\ref{prop_standard_gen_1}) $\Rightarrow$ (\ref{prop_standard_gen_2})
By Proposition~\ref{prop_weight_is_dom}(\ref{prop_weight_is_dom_a}), it suffices to prove that $\mu \in \mathfrak R(\mathfrak F(X))$.
Properties~(\ref{DR1}) and~(\ref{DR2}) hold by Proposition~\ref{prop_pairing_with_mu}.
Let a colored cone $(\mathcal K, \mathcal F) \in {}^c\mathfrak F(X)$ be such that $\langle \mathcal K, \mu \rangle = 0$ and let $\widetilde{\mathcal K}$ be the cone generated by~$\mathcal K$ and~$\rho_\mu$.
By Lemma~\ref{lemma_add_cone2}, the collection ${}^c \widetilde{\mathfrak F}(X) = {}^c \mathfrak F(X) \cup \lbrace\text{all faces of} \ (\widetilde{\mathcal K}, \mathcal F) \rbrace$ is a (strictly convex) colored fan in~$N_\QQ$.
Let $\widetilde X \supset X$ be the spherical $G$-variety corresponding to~${}^c \widetilde{\mathfrak F}(X)$ and let $X_0 \subset \widetilde X$ be the simple spherical subvariety corresponding to the colored cone~$(\widetilde{\mathcal K}, \mathcal F)$.
By Proposition~\ref{prop_hor_simple}, the action of~$H$ can be extended to~$X_0$ (and hence to the whole~$\widetilde X$).
By Corollary~\ref{cor_stable_orbits_G1}(\ref{cor_stable_orbits_G1_b}), the $G$-orbits in~$X_0$ corresponding to $(\mathcal K, \mathcal F)$ and $(\widetilde{\mathcal K}, \mathcal F)$ are connected by~$H$, therefore $(\widetilde{\mathcal K}, \mathcal F) \in {}^c \mathfrak F(X)$ and hence $\widetilde{\mathcal K} \in \mathfrak F(X)$, which proves~(\ref{DR3}).

(\ref{prop_standard_gen_2}) $\Rightarrow$ (\ref{prop_standard_gen_1})
By Lemma~\ref{lemma_DR_hor}(\ref{lemma_DR_hor_b}), one has $\mu \in \mathfrak R(\mathfrak F(Z))$.
Then there exists a $T$-root subgroup on $Z$ of weight~$\mu$, which trivially extends to a $B$-root subgroup $H$ on $X_{\mathcal D} \simeq P_u \times Z$.
Now let $\mathcal C \in \mathfrak F(X)$ be an arbitrary cone, let $(\mathcal C, \mathcal F) \in {}^c \mathfrak F(X)$ be the corresponding colored cone, and let $X_{\mathcal C}$ denote the corresponding simple spherical subvariety in~$X$.

\textit{Case}~1: $\rho_\mu \in \mathcal C^1$.
Then by Proposition~\ref{prop_hor_simple} we know that $H$ extends to a standard $B$-root subgroup on~$X_{\mathcal C}$.

\textit{Case}~2: $\langle \mathcal C, \mu \rangle = 0$.
Let $\widetilde{\mathcal C}$ be the cone generated by $\mathcal C$ and~$\rho_\mu$.
By Lemma~\ref{lemma_DR_hor}(\ref{lemma_DR_hor_a}), one has $(\widetilde{\mathcal C}, \mathcal F) \in {}^c \mathfrak F(X)$.
Again by Proposition~\ref{prop_hor_simple} we know that $H$ extends to a standard $B$-root subgroup on~$X_{\widetilde{\mathcal C}}$.

\textit{Case}~3: $\langle \mathcal C, \mu \rangle \ge 0$ and there is $\rho \in \mathcal C^1$ such that $\langle \rho, \mu \rangle > 0$.
We claim that the vector field $\xi$ corresponding to~$H$ vanishes on~$O_{\mathcal C}$.
Let $\mathcal F_0 \subset \mathcal D$ be the set such that $(\QQ_{\ge0}\rho, \mathcal F_0) \in {}^c \mathfrak F(X)$.
Consider the open subset $X' \subset X$ corresponding to the colored fan $\lbrace (0, \varnothing), (\QQ_{\ge0}\rho_\mu, \varnothing), (\QQ_{\ge0}\rho, \mathcal F_0) \rbrace$.
Let $\mathcal C_0$ be the cone generated by $\rho_\mu$ and~$\rho$ and let $X_0$ be the simple spherical $G$-variety with colored cone~$(\mathcal C_0, \mathcal F_0)$.
Observe that $X_0 \supset X'$.
Then $\xi$ extends to $X_0$ and integrates there to a $B$-root subgroup~$H_0$.
For this action, by Corollary~\ref{cor_stable_orbits_G1}(\ref{cor_stable_orbits_G1_b}), $O_{\QQ_{\ge0}\rho}$ consists of $H_0$-fixed points, hence $\xi$ vanishes on~$O_{\QQ_{\ge0}\rho}$, hence $\xi$ vanishes on~$O_{\mathcal C}$.
By Proposition~\ref{prop_LV_ext2}, $H$ extends to a trivial action on~$O_{\mathcal C}$.
\end{proof}

Let $H$ be a standard $B$-root subgroup on $X$ of weight~$\mu$.
The proof of Proposition~\ref{prop_standard_gen} implies the following result, which generalizes Proposition~\ref{prop_stable_orbits_G} and Corollary~\ref{cor_stable_orbits_G1}.

\begin{proposition} \label{prop_stable_orbits_hor}
Given a cone $\mathcal E \in \mathfrak F(X)$, the following assertions hold.
\begin{enumerate}[label=\textup{(\alph*)},ref=\textup{\alph*}]
\item \label{prop_stable_orbits_hor_a}
If there is $\rho \in \mathcal E^1$ such that $\langle \rho, \mu \rangle > 0$, then $O_{\mathcal E}$ is pointwise fixed by~$H$.

\item \label{prop_stable_orbits_hor_b}
If $\langle \mathcal E, \mu \rangle \le 0$, then there exist cones $\mathcal K, \widetilde{\mathcal K} \in \mathfrak F$ such that $\langle \mathcal K, \mu \rangle = 0$, $\widetilde{\mathcal K}$ is generated by $\mathcal K$ and~$\rho_\mu$, $\mathcal E \in \lbrace \mathcal K, \widetilde{\mathcal K}\rbrace$, and $HO_{\mathcal E} = O_{\mathcal K} \cup O_{\widetilde{\mathcal K}}$.
\end{enumerate}
\end{proposition}

\begin{proposition} \label{prop_hor_horiz}
Suppose the fan $\mathfrak F(X)$ is convex \textup(which holds in particular when $X$ is affine, simple, or complete\textup).
For a weight $\mu \in \mathfrak X(T)$, the following conditions are equivalent.
\begin{enumerate}[label=\textup{(\arabic*)},ref=\textup{\arabic*}]
\item \label{prop_hor_horiz_1}
There exists a horizontal $B$-root subgroup on~$X$ of weight~$\mu$.

\item \label{prop_hor_horiz_2}
$\mu \in \mathfrak R(\mathfrak F(X)) \cap \Lambda^+$.
\end{enumerate}
\end{proposition}

\begin{proof}
(\ref{prop_hor_horiz_1})$\Rightarrow$(\ref{prop_hor_horiz_2})
Thanks to Proposition~\ref{prop_weight_is_dom}(\ref{prop_weight_is_dom_a}), one has $\mu \in \Lambda^+$.
By Proposition~\ref{prop_pairing_with_mu}, $\mu$ satisfies~(\ref{DR1}) and~(\ref{DR2}).
Property~(\ref{DR3}) follows from Corollary~\ref{crl_convex_DR3}.

(\ref{prop_hor_horiz_2})$\Rightarrow$(\ref{prop_hor_horiz_1})
By Proposition~\ref{prop_standard_gen}, there exists a standard $B$-root subgroup of weight~$\mu$.
\end{proof}

\begin{corollary}
Under the hypotheses of Proposition~\textup{\ref{prop_hor_horiz}}, take any $D \in \mathcal D^G$.
The following conditions are equivalent.
\begin{enumerate}[label=\textup{(\arabic*)},ref=\textup{\arabic*}]
\item
There exists a $B$-root subgroup on~$X$ that moves~$D$.

\item
$\mathfrak R_{\varkappa(D)}(\mathfrak F(Z)) \cap \Lambda^+ \ne \varnothing$.
\end{enumerate}
\end{corollary}

\subsection{\texorpdfstring{$G$}{G}-root subgroups}

Let $X$ be a horospherical $G$-variety.

\begin{proposition} \label{prop_G-root_subgroups}
Suppose $H$ is a $G$-root subgroup on~$X$ of weight~$\mu$.
Then $H$ is standard \textup(and hence horizontal\textup).
In particular, $H$ is uniquely determined by its weight among the $B$-root subgroups on~$X$.
\end{proposition}

\begin{proof}
Since $H$ commutes with~$U^-$, it preserves the canonical section~$Z$, hence $H$ is standard.
\end{proof}

Combining Propositions~\ref{prop_G-root_subgroups} and~\ref{prop_standard_gen} we obtain the next result.

\begin{proposition}
For a weight $\mu \in \mathfrak X(G)$, the following conditions are equivalent.
\begin{enumerate}[label=\textup{(\arabic*)},ref=\textup{\arabic*}]
\item
There exists a $G$-root subgroup on~$X$ of weight~$\mu$.

\item
$\mu \in \mathfrak R(\mathfrak F(X))$.
\end{enumerate}
\end{proposition}

\subsection{Commutation relations}
\label{ssec_comm_rel_stand}

Let $X$ be an arbitrary horospherical $G$-variety with open $G$-orbit~$O$.
For every standard $B$-root subgroup of weight~$\mu$ on~$X$, let $\xi_\mu$ be the corresponding $B$-semiinvariant vector field on~$X$ and let $\mathcal L_\mu$ be the simple $G$-submodule generated by~$\xi_{\mu}$.
In this subsection we compute commutation relations between the $G$-modules $\mathcal L_\mu$ under certain restrictions.
When $X$ is complete, this will yield commutation relations between all possible $\mathcal L_\mu$.

Recall from \S\,\ref{ssec_aut_group_sph} that the weight lattice $M$ is naturally identified with the group $\mathfrak X(K)$ where $K$ is the connected component of the identity of the group of equivariant automorphisms of~$O$.
Thus every element $\rho \in N$ corresponds to a one-parameter subgroup of~$K$, and we let $\nu_\rho$ denote the corresponding vector field on~$X$.

Let $\mu_1,\mu_2 \in \mathfrak R(\mathfrak F(X)) \cap \Lambda^+$.
Put $\rho_1 = \rho_{\mu_1}$ and $\rho_2 = \rho_{\mu_2}$ for short.

When $X$ is affine, we regard $\mathcal L_{\mu}$ as a $G$-submodule of $\Der \KK[X]$.
Recall that $\mathcal L_{\mu}$ is identified with $\KK[O]_\mu$.

\begin{proposition} \label{prop_comm_rel_aff}
Suppose $X$ is affine and the derivations $\partial_1 \in \mathcal L_{\mu_1}, \partial_2 \in \mathcal L_{\mu_2}$ are defined by functions $f_1 \in \KK[O]_{\mu_1}, f_2 \in \KK[O]_{\mu_2}$, respectively.
Then for every $\lambda \in \Gamma$ and $g \in \KK[X]_\lambda$ one has $[\partial_1,\partial_2](g) = (\langle \rho_1, \mu_2 \rangle \langle \rho_2, \lambda \rangle -\langle \rho_2, \mu_1 \rangle \langle \rho_1, \lambda \rangle) f_1 f_2 g$.
\end{proposition}

\begin{proof}
The claim is implied by the computation
\begin{multline*}
[\partial_1,\partial_2](g) = \partial_1\partial_2(g) - \partial_2\partial_1(g) = \partial_1(\langle \rho_2, \lambda \rangle f_2 g) - \partial_2(\langle \rho_1, \lambda \rangle f_1 g) =\\
\langle \rho_1, \lambda + \mu_2 \rangle \langle \rho_2, \lambda \rangle f_1 f_2 g -\langle \rho_2, \lambda + \mu_1 \rangle \langle \rho_1, \lambda \rangle f_1 f_2 g = (\langle \rho_1, \mu_2 \rangle \langle \rho_2, \lambda \rangle -\langle \rho_2, \mu_1 \rangle \langle \rho_1, \lambda \rangle) f_1 f_2 g.
\end{multline*}
\end{proof}

\begin{proposition} \label{prop_comm_rel_stand}
The following assertions hold.
\begin{enumerate}[label=\textup{(\alph*)},ref=\textup{\alph*}]
\item \label{prop_comm_rel_stand_a}
If $\rho_1 = \rho_2$, then $[\mathcal L_{\mu_1}, \mathcal L_{\mu_2}] = 0$.
\item \label{prop_comm_rel_stand_b}
If $\rho_1 \ne \rho_2$ and $\langle \rho_2, \mu_1 \rangle = \langle \rho_1,\mu_1 \rangle = 0$, then $[\mathcal L_{\mu_1}, \mathcal L_{\mu_2}] = 0$.
\item \label{prop_comm_rel_stand_c}
If $\rho_1 \ne \rho_2$, $\langle \rho_2, \mu_1 \rangle = 0$, and $\langle \rho_1, \mu_2 \rangle > 0$, then $\mu_1+\mu_2 \in \mathfrak R_{\rho_2}(\mathfrak F(X))$, $[\mathcal L_{\mu_1}, \mathcal L_{\mu_2}] = \mathcal L_{\mu_1 + \mu_2}$, and $[\xi_{\mu_1},\xi_{\mu_2}] = \langle \rho_1, \mu_2 \rangle \xi_{\mu_1+\mu_2}$.
\item \label{prop_comm_rel_stand_d}
If $\mu_1 + \mu_2 = 0$, then $\mu_1,\mu_2 \in \mathfrak X(G)$, $\dim \mathcal L_{\mu_1} = \dim \mathcal L_{\mu_2} = 1$, and $[\xi_{\mu_1},\xi_{\mu_2}] = \nu_{\rho_1-\rho_2}$.
\end{enumerate}
\end{proposition}

\begin{proof}
In cases~(\ref{prop_comm_rel_stand_a}--\ref{prop_comm_rel_stand_c}),
passing to an open subset of~$X$, we may assume that $\mathfrak F(X)$ consists of all faces of the cone $\mathcal C$ generated by $\rho_1$ and~$\rho_2$.
Then $X$ is simple, and we apply the construction of \S\,\ref{subsec_simple_horosph}.
Let $\widetilde {\mathcal L}_{\mu_1}, \widetilde {\mathcal L}_{\mu_2}$ be the $\widetilde G$-modules generated by $\partial_{f_1}, \partial_{f_2}$, respectively.
Then the claim follows from Proposition~\ref{prop_comm_rel_aff}.

In case~(\ref{prop_comm_rel_stand_d}) we automatically get $\mu_1,\mu_2 \in \mathfrak X(G)$ and $\dim \mathcal L_{\mu_1} = \dim \mathcal L_{\mu_2} = 1$, and thus it suffices to compute $[\xi_{\mu_1},\xi_{\mu_2}]$.
In turn, the latter can be done on~$Z$, and a direct computation yields $[\xi_{\mu_1},\xi_{\mu_2}] = \nu_{\rho_1-\rho_2}$.
\end{proof}

\begin{proposition} \label{prop_comm_rel_compl}
Suppose $X$ is complete.
Then Proposition~\textup{\ref{prop_comm_rel_stand}} lists the commutation relations between all possible $G$-modules $\mathcal L_{\mu_1}, \mathcal L_{\mu_2}$ with $\mu_1,\mu_2 \in \mathfrak R(\mathfrak F(X)) \cap \Lambda^+$.
\end{proposition}

\begin{proof}
If cases (\ref{prop_comm_rel_stand_a}--\ref{prop_comm_rel_stand_c}) of Proposition~\ref{prop_comm_rel_stand} do not hold, then if $\langle \rho_1, \mu_2 \rangle > 0$ and $\langle \rho_2, \mu_1 \rangle > 0 $.
Lemma~\ref{lemma_>0} then yields $\mu_1 + \mu_2 = 0$.
\end{proof}

\section{Vertical \texorpdfstring{$B$}{B}-root subgroups in the horospherical case}
\label{sec_hor_vert}

In this subsection we assume that $G = C \times G^{ss}$ where $C$ is a torus and $G^{ss}$ is a simply connected semisimple group.
For every $\gamma \in \Pi$, let $\varpi_\gamma \in \mathfrak X(T^{ss})$ be the corresponding fundamental weight.
Consider the open subset $G_0 = U T U^- \simeq U \times T \times U^-$ of~$G$.
For every $\lambda \in \mathfrak X(T)$ let $F_\lambda \in \KK[T]$ be the function representing the character $-\chi$.
Then $F_\lambda$ naturally extends to a $(B\times B^-)$-semiinvariant function in $\KK[G_0]$ of biweight~$(\lambda, -\lambda)$.

\subsection{Certain LND's on~\texorpdfstring{$\KK[G]$}{K[G]}}

For every $\beta \in \Pi$, let $D_\beta$ be the corresponding $B \times B^-$-stable prime divisor in~$G$.
Then for every $\lambda \in \mathfrak X(T)$ the order of $F_\lambda$ along~$D_\beta$ equals~$\langle \beta^\vee, \lambda \rangle$.
Consider the minimal parabolic subgroup $P_\beta \supset B$ with standard Levi subgroup $L_\beta$ and the corresponding opposite parabolic subgroup $P^-_{\beta}$.
Let $R_\beta$ (resp.~$R^-_\beta$) be the unipotent radical of $P_\beta$ (resp.~$P^-_\beta$).
Put also $G_\beta = R_\beta P^-_\beta$; this is an open subset in $G$ containing~$G_0$.
One has $G_\beta \setminus G_0 = G_\beta \cap D_\beta$.

For every $\beta,\gamma \in \Pi$, let $V_{\beta,\gamma}$ be the simple $L_\beta$-module with highest weight $\varpi_\gamma$ and let $v_{\beta,\gamma}$ be a lowest-weight vector in~$V_{\beta,\gamma}$.
Then $V_{\beta,\gamma}$ is one-dimensional if $\beta \ne \gamma$ and two-dimensional if $\beta = \gamma$.
Put also $v'_{\beta,\beta} = e_\beta v_{\beta,\beta}$, so that $v_{\beta,\beta}, v'_{\beta,\beta}$ form a basis in~$V_{\beta,\beta}$.
Put $V_\beta = \bigoplus \limits_{\gamma \in \Pi} V_{\beta,\gamma}$; then the representation of $L_\beta^{ss} = L_\beta \cap G^{ss}$ on $V_\beta$ is faithful.
Let $\widetilde{L}_\beta^{ss}$ be the image of $L_\beta^{ss}$ in $\GL(V_{\beta,\beta})$.
Then $\widetilde{L}_\beta^{ss}$ equals either $\GL(V_{\beta,\beta})$ or $\SL(V_{\beta,\beta})$.
We have $L_\beta \simeq C \times \widetilde{L}_\beta^{ss} \times \prod\limits_{\gamma \in \Pi \setminus \beta} \GL(V_{\beta,\gamma})$.

Consider the open subset $G_0 \cap L_\beta = U_\beta \times T \times U_{-\beta} \subset L_\beta$ and let $\xi_\beta$ be the coordinate function on $U_\beta$, so that $\xi_\beta(u_\beta(x)) = x$ for all $x \in \KK$.
We shall also regard $\xi_\beta$ as a regular function on~$G_0 \cap L_\beta$.

\begin{lemma} \label{lemma_xi_beta}
Given $\lambda \in \mathfrak X(T)$ and $k \in \ZZ_{\ge0}$, the function $F_\lambda \xi_\beta^k \in \KK[G_0 \cap L_\beta]$ extends to a regular function on $L_\beta$ if and only if $\langle \beta^\vee, \lambda \rangle \ge k$.
\end{lemma}

\begin{proof}
In the basis $v'_{\beta,\beta}, v_{\beta,\beta}$ of $V_{\beta,\beta}$, the matrices of $u_\beta(x) \in U_\beta, t \in T, u_{-\beta}(y) \in U_{-\beta}$ in $\widetilde{L}_\beta^{ss}$ are $\begin{pmatrix} 1 & x \\ 0 & 1 \end{pmatrix}, \begin{pmatrix} \chi_1(t) & 0 \\ 0 & \chi_2(t) \end{pmatrix}, \begin{pmatrix} 1 & 0 \\ y & 1 \end{pmatrix}$, respectively, where $\chi_1 = \varpi_\beta$ and $\chi_2 = \varpi_\beta - \beta$.
The matrix of $u_\beta(x) t u_{-\beta}(y)$ equals $\begin{pmatrix} \chi_1(t) +\chi_2(t)xy & \chi_2(t)x \\ \chi_2(t)y & \chi_2(t) \end{pmatrix}$.
Thus $\xi_\beta = g_{12}/g_{22}$ on $G_0 \cap L_\beta$ where $g_{ij}$ stands for the $ij$-th element of the matrix of an element of $\widetilde{L}_\beta^{ss}$.
Clearly, $g_{22} = F_{\beta - \varpi_\beta}$, hence $\xi_\beta = g_{12}F_{\varpi_\beta-\beta}$.
Since the order of $g_{12}$ along $D_\beta$ is~$0$ and $\langle \beta^\vee, \varpi_\beta - \beta \rangle = -1$, we get the claim.
\end{proof}

For every $\alpha \in \Delta^+$ consider the $\GG_a$-action on~$G_0$ given by 
\begin{equation} \label{eqn_Ualpha-action}
(s,(x,y)) \mapsto (xu_\alpha(-s),y) \quad \text{for all} \quad s \in \KK, x \in U, y \in B^-
\end{equation}
and let $\partial_\alpha$ be the derivation of $\KK[G_0]$ corresponding to this action.
Then $\partial_\alpha$ is $B \times B^-$-normalized of biweight $(\alpha, 0)$.

\begin{proposition} \label{prop_ext_to_G}
Let $\alpha \in \Delta^+$ and $\mu \in \mathfrak X(T)$.
\begin{enumerate}[label=\textup{(\alph*)},ref=\textup{\alph*}]
\item \label{prop_ext_to_G_a}
The derivation $F_{\mu-\alpha} \partial_\alpha$ preserves $\KK[G]$ if and only if it preserves $\KK[G_\beta]$ for all $\beta \in \Pi$.

\item \label{prop_ext_to_G_b}
Given $\beta \in \Pi$, the derivation $F_{\mu-\alpha} \partial_\alpha$ preserves $\KK[G_\beta]$ if and only if $\langle \beta^\vee, \mu \rangle \ge c(\alpha,\beta)$, where the values $c(\alpha,\beta)$ for all possible cases are collected in Table~\textup{\ref{table_ab}}.
\end{enumerate}
\end{proposition}

\begin{table}[h]
\caption{} \label{table_ab}
\begin{tabular}{|c|c|c|c|c|}
\hline
No. & $\Delta(\alpha,\beta)$ & $\beta$ & $\angle(\alpha, \beta)$ & $c(\alpha,\beta)$ \\
\hline
\no \label{no1} & $\mathsf A_1$ & $\beta=\alpha$ & $0$ & $2$ \\
\hline
\no \label{no2} & $\mathsf A_1 \times \mathsf A_1$ & -- & $\pi/2$ & $0$ \\
\hline
\no \label{no3} & $\mathsf A_2$ & -- & $\pi/3$ & $1$ \\
\hline
\no \label{no4} & $\mathsf A_2$ & -- & $2\pi/3$ & $0$ \\
\hline
\no \label{no5} & $\mathsf B_2$ & short & $\pi/4$ & $2$ \\
\hline
\no \label{no6} & $\mathsf B_2$ & short & $\pi/2$ & $1$ \\
\hline
\no \label{no7} & $\mathsf B_2$ & short & $3\pi/4$ & $0$ \\
\hline
\no \label{no8} & $\mathsf B_2$ & long & $\pi/4$ & $1$ \\
\hline
\no \label{no9} & $\mathsf B_2$ & long & $3\pi/4$ & $0$ \\
\hline
\no \label{no10} & $\mathsf G_2$ & short & $\pi/6$ & $3$ \\
\hline
\no \label{no11} & $\mathsf G_2$ & short & $\pi/3$ & $2$ \\
\hline
\no \label{no12} & $\mathsf G_2$ & short & $2\pi/3$ & $1$ \\
\hline
\no \label{no13} & $\mathsf G_2$ & short & $5\pi/6$ & $0$ \\
\hline
\no \label{no14} & $\mathsf G_2$ & long & $\pi/6$ & $1$  \\
\hline
\no \label{no15} & $\mathsf G_2$ & long & $5\pi/6$ & $0$ \\
\hline
\end{tabular}
\end{table}

\begin{remark}
Some cases in Table~\textup{\ref{table_ab}} are excluded as duplicate.
\end{remark}

\begin{proof}
(\ref{prop_ext_to_G_a})
This is straightforward from $\KK[G] = \bigcap \limits_{\beta \in \Pi} \KK[G_\beta]$.

(\ref{prop_ext_to_G_b})
Put $\Delta(\alpha,\beta) = \Delta \cap \ZZ\lbrace \alpha, \beta \rbrace$; this is a root system of rank $\le 2$.

Case $\alpha = \beta$.
Then $F_{\mu-\alpha} \partial_\alpha$ preserves $\KK[G_\beta]$ if and only if it preserves $\KK[L_\beta]$.
On the open subset $G_0 \cap L_\beta \simeq U_\beta \times T \times U_{-\beta}$ we have $(s,(x,t,y)) \mapsto (x-s, t, y)$ for all $s,x,y \in \KK$ and $t \in T$.
Hence $\partial_\alpha(\xi_\beta) = -1$, $\partial_\alpha(\xi_{-\beta}) = 0$, $\partial_\alpha(\KK[T]) = 0$.
It follows that on $\KK[L_\beta]$ we have $\partial_\alpha(g_{11}) = -g_{21}$, $\partial_\alpha(g_{12}) = -g_{22}  = -F_{\beta-\varpi_\beta}$, $\partial(g_{21}) = \partial(g_{22}) = 0$.
Since the order of $g_{21}$ along $D_\beta$ is~$0$, it follows that $F_{\mu-\alpha} \partial_\alpha$ preserves $\KK[L_\beta]$ if and only if $\langle \beta^\vee, \mu - \alpha \rangle \ge 0$, that is, $\langle \beta^\vee, \mu \rangle \ge 2$.

In what follows we assume $\alpha \ne \beta$, so that the rank of $\Delta(\alpha,\beta)$ equals~$2$.
The general strategy is as follows.
Put $\Theta(\alpha,\beta) = \Delta^+ \cap \ZZ_{\ge0} \lbrace \alpha, \beta \rbrace$ and $\Theta_0(\alpha,\beta) = \Theta(\alpha,\beta) \setminus \lbrace \beta \rbrace$.
Let $R''_\beta$ be the direct product (in any fixed order) of all $U_\gamma$ with $\gamma \in \Delta^+ \setminus \Theta(\alpha,\beta)$.
Let $R'_\beta$ be the direct product of all $U_\gamma$ with $\gamma \in \Theta_0(\alpha,\beta)$ in the order that will be specified in each case below.
Then we fix the isomorphism $R''_\beta \times R'_\beta \times U_\beta \times T \times U_{-\beta} \times R^-_\beta \xrightarrow{\sim} G_0$ taking each tuple of elements to their product.
For every $\gamma \in \Delta^+$ we let $\xi_\gamma$ denote the coordinate function of $U_\gamma$, so that $\xi_\gamma(u_\gamma(s)) = s$ for all $s \in \KK$.
It will turn out that the $\GG_a$-action~(\ref{eqn_Ualpha-action}) can change only elements in the component~$R'_\beta$, therefore $\partial_\alpha$ vanishes on $\KK[R''_\beta]$ and $\KK[U_\beta \times T \times U_{-\beta} \times R^-_\beta]$.
In all the cases below, the explicit formula for the $\GG_a$-action (\ref{eqn_Ualpha-action}) is provided only for the component $R'_\beta$.
For short, the subscript $ij$ always denotes $i\alpha + j\beta$.
All case numbers refer to Table~\ref{table_ab}.

Case~$\alpha + \beta \notin \Delta$.
It occurs in cases \ref{no2}, \ref{no3}, \ref{no5}, \ref{no8}, \ref{no10}, \ref{no14}.
Then $R'_\beta = U_{10}$.
Since $U_{10}$ and $U_{01}$ commute, the $\GG_a$-action~(\ref{eqn_Ualpha-action}) on $R'_\beta$ is given by $(s,x_{10}) \mapsto x_{10}-s$.
Then $\partial_\alpha(\xi_{10}) = -1$.
Thus $F_{\mu-\alpha}\partial_\alpha$ preserves $\KK[G_\beta]$ if and only if $\langle \beta^\vee, \mu - \alpha \rangle \ge 0$, which is equivalent to $\langle \beta^\vee, \mu \rangle \ge \langle \beta^\vee, \alpha \rangle$.

Case~$\Theta(\alpha,\beta) = \lbrace \alpha, \beta, \alpha + \beta \rbrace$.
It occurs in cases~\ref{no4}, \ref{no6}, \ref{no11}.
We put $R'_\beta = U_{10} \times U_{11}$.
Thanks to formulas (\ref{eqn_no4}), (\ref{eqn_no6}), (\ref{eqn_no11}), the action of $\partial_\alpha$ on $\KK[R'_\beta]$ is given by
\[
\xi_{10} \mapsto -1, \quad \xi_{11} \mapsto c\xi_{01}
\]
where $c = 1,\,-2,\,-3$, respectively.
By Lemma~\ref{lemma_xi_beta}, the derivation $F_{\mu-\alpha}\partial_\alpha$ preserves $\KK[G_\beta]$ if and only if $\langle \beta^\vee, \mu - \alpha \rangle \ge 1$, which is equivalent to $\langle \beta^\vee, \mu \rangle \ge \langle \beta^\vee, \alpha \rangle + 1$.

Case~\ref{no7}.
We put $U_{10} \times U_{11} \times U_{12}$.
Thanks to formula~(\ref{eqn_no7}), for an appropriate choice of the elements $e_\gamma$ with $\gamma \in \Theta_0(\alpha,\beta)$, the action of $\partial_\alpha$ on $\KK[R'_\beta]$ is given by
\[
\xi_{10} \mapsto -1, \quad \xi_{11} \mapsto -\xi_{01}, \quad \xi_{12} \mapsto -\xi_{01}^2.
\]
By Lemma~\ref{lemma_xi_beta}, the derivation $F_{\mu-\alpha}\partial_\alpha$ preserves $\KK[G_\beta]$ if and only if $\langle \beta^\vee, \mu - \alpha \rangle \ge 2$, which is equivalent to $\langle \beta^\vee, \mu \rangle \ge \langle \beta^\vee, \alpha \rangle + 2 = 0$.

Case~\ref{no9}.
We put $U_{10} \times U_{11} \times U_{21}$.
Thanks to formula~(\ref{eqn_no9}), for an appropriate choice of the elements $e_\gamma$ with $\gamma \in \Theta_0(\alpha,\beta)$, the action of $\partial_\alpha$ on $\KK[R'_\beta]$ is given by
\[
\xi_{10} \mapsto -1, \quad \xi_{11} \mapsto \xi_{01}, \quad \xi_{21} \mapsto 2\xi_{11}.
\]
By Lemma~\ref{lemma_xi_beta}, the derivation $F_{\mu-\alpha}\partial_\alpha$ preserves $\KK[G_\beta]$ if and only if $\langle \beta^\vee, \mu - \alpha \rangle \ge 1$, which is equivalent to $\langle \beta^\vee, \mu \rangle \ge \langle \beta^\vee, \alpha \rangle + 1 = 0$.

Case~\ref{no12}.
We put $U_{10} \times U_{11} \times U_{12} \times U_{21}$.
Thanks to formula~(\ref{eqn_no11}), for an appropriate choice of the elements $e_\gamma$ with $\gamma \in \Theta_0(\alpha,\beta)$, the action of $\partial_\alpha$ on $\KK[R'_\beta]$ is given by
\[
\xi_{10} \mapsto -1, \quad \xi_{11} \mapsto 2\xi_{01}, \quad \xi_{12} \mapsto 3\xi_{01}^2, \quad \xi_{21} \mapsto -3\xi_{11}.
\]
By Lemma~\ref{lemma_xi_beta}, the derivation $F_{\mu-\alpha}\partial_\alpha$ preserves $\KK[G_\beta]$ if and only if $\langle \beta^\vee, \mu - \alpha \rangle \ge 2$, which is equivalent to $\langle \beta^\vee, \mu \rangle \ge \langle \beta^\vee, \alpha \rangle + 2 = 1$.

Case~\ref{no13}.
We put $U_{10} \times U_{11} \times U_{12} \times U_{13} \times U_{23}$.
Thanks to formula~(\ref{eqn_no13}), for an appropriate choice of the elements $e_\gamma$ with $\gamma \in \Theta_0(\alpha,\beta)$, the action of $\partial_\alpha$ on $\KK[R'_\beta]$ is given by
\[
\xi_{10} \mapsto -1, \quad \xi_{11} \mapsto -\xi_{01}, \quad \xi_{12} \mapsto \xi_{01}^2, \quad \xi_{13} \mapsto \xi_{01}^3, \quad \xi_{23} \mapsto \xi_{13} - 3\xi_{01}\xi_{12}.
\]
By Lemma~\ref{lemma_xi_beta}, the derivation $F_{\mu-\alpha}\partial_\alpha$ preserves $\KK[G_\beta]$ if and only if $\langle \beta^\vee, \mu - \alpha \rangle \ge 3$, which is equivalent to $\langle \beta^\vee, \mu \rangle \ge \langle \beta^\vee, \alpha \rangle + 3 = 0$.

Case~\ref{no15}.
We put $U_{10} \times U_{11} \times U_{21} \times U_{31} \times U_{32}$.
Thanks to formula~(\ref{eqn_no15}), for an appropriate choice of the elements $e_\gamma$ with $\gamma \in \Theta_0(\alpha,\beta)$, the action of $\partial_\alpha$ on $\KK[R'_\beta]$ is given by
\[
\xi_{10} \mapsto -1, \quad \xi_{11} \mapsto \xi_{01}, \quad \xi_{21} \mapsto -2\xi_{11}, \quad \xi_{31} \mapsto 3\xi_{21}, \quad \xi_{32} \mapsto -3\xi_{11}^2 + 3\xi_{01}\xi_{21}.
\]
By Lemma~\ref{lemma_xi_beta}, the derivation $F_{\mu-\alpha}\partial_\alpha$ preserves $\KK[G_\beta]$ if and only if $\langle \beta^\vee, \mu - \alpha \rangle \ge 1$, which is equivalent to $\langle \beta^\vee, \mu \rangle \ge \langle \beta^\vee, \alpha \rangle + 1 = 0$.
\end{proof}

\subsection{Certain LND's on \texorpdfstring{$\KK[G/U^-]$}{K[G/U-]}}

Now put $X = G/U^{-}$ and recall the open cell $X_{\mathcal D} \subset X$.
The local structure theorem gives the decomposition $X_{\mathcal D} \simeq U \times Z$ with $Z = T$.
We have $M = \mathfrak X(T)$.
Given $\lambda \in \mathfrak X(T)$, the restriction of the function $f_\lambda \in \KK[X]$ to $Z$ is just the character~$-\lambda$.
For every $\alpha \in \Delta^+$, let $\delta_\alpha$ be the LND on $\KK[U \times T]$ corresponding to the $\GG_a$-action $(s,(u,t)) \mapsto (uu_\alpha(-s),t)$.

For every $\beta \in \Pi$, let $D_\beta \subset X$ be the corresponding color and let $X_\beta \subset X$ be the open subset obtained by removing all colors except~$D_\beta$.
Clearly, $U_{-\beta}$ is a maximal unipotent subgroup of~$L_\beta$ and one has a natural $B$-equivariant isomorphism $X_\beta \simeq R_\beta \times L_\beta/U_{-\beta}$.
The open cell $X_0 \subset X_\beta$ is then naturally identified with $R_\beta \times U_{\beta} \times T$.

For every $\beta \in \Pi$, recall the $L_\beta$-module $V_\beta$.
For every $\gamma \in \Pi$, let $y_{\beta,\gamma}$ be the coordinate function for $v_{\beta,\gamma}$ and $x_\beta$ the coordinate function of~$v_{\beta,\beta}'$.
The variety $Y_\beta = L_\beta^{ss}/U_{-\beta}$ is realized as the closure in~$V_\beta$ of the orbit $L_\beta^{ss}v$ where $v = \sum \limits_{\gamma \in \Pi}v_{\beta,\gamma}$.
Moreover, in this realization one has $Y_\beta = \lbrace \sum\limits_{\gamma \in \Pi} w_\gamma \mid w_\gamma \in V_{\beta,\gamma} \setminus \lbrace 0 \rbrace \rbrace$.
In particular, $\KK[Y_\beta]$ is generated by the functions $x_\beta,y_{\beta,\beta}$, and all $y_{\beta,\gamma}^{\pm1}$ with $\gamma \ne \beta$.
Note that $y_{\beta,\gamma} = f_{-\varpi_\gamma}$ for all $\gamma \in \Pi \setminus \lbrace \beta \rbrace$ and $y_{\beta,\beta} = f_{\beta - \varpi_\beta}$.
Given $\lambda \in \mathfrak X(T)$, one has $f_\lambda \in \KK[Y_\beta]$ if and only if $\langle \beta^\vee, \lambda \rangle \ge 0$.
The subset $T \subset Y_\beta$ is given by the condition $x_\beta = 0$, and the subset $X_0 \cap Y_\beta$ is given by the condition $y_\beta \ne 0$.
Identify $X_0 \cap Y_\beta$ with $U_{\beta} \times T$ in a natural way and let $\xi_\beta$ be the coordinate function on $U_{\beta}$.
Then $\xi_\beta = x_\beta / y_{\beta,\beta}$.

For every $\alpha \in \Delta^+$ and $\beta \in \Pi$, put
\[
c'(\alpha,\beta) =
\begin{cases}
1 = c(\alpha,\alpha)-1 & \text{if} \ \alpha = \beta,\\
c(\alpha,\beta) & \text{if} \ \alpha \ne \beta.
\end{cases}
\]

\begin{proposition} \label{prop_G/U-}
Let $\alpha \in \Delta^+$ and $\mu \in \mathfrak X(T)$.
\begin{enumerate}[label=\textup{(\alph*)},ref=\textup{\alph*}]
\item \label{prop_G/U-_a}
The derivation $f_{\mu - \alpha} \delta_\alpha$ preserves $\KK[X]$ if and only if it preserves $\KK[X_\beta]$ for all $\beta \in \Pi$.
\item \label{prop_G/U-_b}
Given $\beta \in \Pi$, the derivation $f_{\mu - \alpha} \delta_\alpha$ preserves $\KK[X_\beta]$ if and only if $\langle \beta^\vee, \mu \rangle \ge c'(\alpha,\beta)$.
\end{enumerate}
\end{proposition}

\begin{proof}
(\ref{prop_G/U-_a})
This is straightforward from $\KK[X] = \bigcap \limits_{\beta \in \Pi} \KK[X_\beta]$.

(\ref{prop_G/U-_b})
Case $\alpha = \beta$.
Then $f_{\mu-\alpha} \delta_\alpha$ preserves $\KK[X_\beta]$ if and only if it preserves $\KK[Y_\beta]$.
On the open subset $X_0 \cap Y_\beta \simeq U_\beta \times T$ we have $(s,(x,t)) \mapsto (x-s, t)$ for all $s,x \in \KK$ and $t \in T$.
Hence $\delta_\alpha(\xi_\beta) = -1$ and $\delta_\alpha(\KK[T]) = 0$.
It follows that on $\KK[Y_\beta]$ we have $\delta_\alpha(x_\beta) = -y_{\beta,\beta}  = -f_{\beta-\varpi_\beta}$, $\delta(y_{\beta,\beta}) = 0$.
Thus $f_{\mu-\alpha} \delta_\alpha$ preserves $\KK[Y_\beta]$ if and only if $\langle \beta^\vee, \mu - \alpha \rangle \ge -1$, that is, $\langle \beta^\vee, \mu \rangle \ge 1$.

Case~$\alpha \ne \beta$.
Basically the argument repeats that of the proof of Proposition~\ref{prop_ext_to_G}.
\end{proof}

\subsection{Certain LND's on arbitrary horospherical homogeneous spaces}

Let $P \supset B$ be a parabolic subgroup with standard Levi subgroup $L \subset P$ and consider the corresponding opposite parabolic subgroup $P^- \supset B^-$.
Let $\Omega \subset \Delta^+$ be the set of highest weights of $\mathfrak p_u$ as an $L$-module with respect to $B_L = B \cap L$.

\begin{proposition} \label{prop_G_larger_cell}
Under the above assumptions and notation, for every $\alpha \in \Omega$ the action~\textup{(\ref{eqn_Ualpha-action})} on~$G_0$ extends to the larger open set $P_u P^- \simeq P_u \times P^-$ and is given by the formula $(s,(x,y)) \mapsto (xu_\alpha(-s),y)$ for all $s \in \KK, x \in P_u, y \in P^-$.
In particular, the derivation $\partial_\alpha$ is $B \times P^-$-normalized \textup(of biweight~$(\alpha,0)$\textup).
\end{proposition}

\begin{proof}
Put $U_L = U \cap L$.
Then $U_\alpha$ commutes with $U_L$ for all $\alpha \in \Omega$.
Since $U \simeq P_u \times U_L$, for all $s \in \KK, x_1 \in P_u, x_2 \in U_L, y \in B^-$ we have $x_1x_2u_\alpha(-s)y = x_1u_\alpha(-s)x_2y$.
It remains to notice that $U_LB^-$ is an open subset of~$P^-$.
\end{proof}

Now suppose $S \supset U^-$ is a horospherical subgroup in~$G$ with normalizer equal to~$P^-$.
Put $X = G/S$ and consider the open subset $X_{\mathcal D}$ from the local structure theorem.
Recall the weight lattice~$M$.
For every $\alpha \in \Omega$, recall the vector field $\varepsilon_\alpha$ considered in~\S\,\ref{subsec_local_descr}.

\begin{proposition} \label{prop_LND_on_G/S}
Suppose $\alpha \in \Omega$ and $\mu \in \alpha + M$.
Then the following assertions hold.
\begin{enumerate}[label=\textup{(\alph*)},ref=\textup{\alph*}]
\item \label{prop_LND_on_G/S_a}
If $\langle \beta^\vee, \mu \rangle \ge c'(\alpha,\beta)$ for all $\beta \in \Pi$ then the vector field $f_{\mu - \alpha} \varepsilon_\alpha$ extends to the whole~$X$.

\item \label{prop_LND_on_G/S_b}
If $\langle \beta^\vee, \mu \rangle \ge c(\alpha,\beta)$ for all $\beta \in \Pi$ then the extension to $G/S$ of the vector field $f_{\mu - \alpha} \varepsilon_\alpha$ is $\GG_a$-integrable.
\end{enumerate}
\end{proposition}

\begin{proof}
(\ref{prop_LND_on_G/S_a})
This follows from Proposition~\ref{prop_G/U-} by considering the natural morphism\\ $\varphi \colon G/U^- \to X$.
More precisely, we have a well-defined map of vector fields for the restricted map $\varphi^{-1}(X_{\mathcal D}) \to X_{\mathcal D}$, and it extends to the whole $G/U^-$.

(\ref{prop_LND_on_G/S_b})
Thanks to Propositions~\ref{prop_ext_to_G} and~\ref{prop_G_larger_cell}, the derivation $F_{\mu - \alpha} \partial_\alpha$ preserves $\KK[G]$ and is $B \times P^-$-normalized, respectively.
Since $\mu - \alpha \in M$, it follows that $F_{\mu - \alpha} \partial_\alpha$ is invariant with respect to the action of $S$ on the right, hence gives rise to a $\GG_a$-subgroup with the same property.
This $\GG_a$-subgroup descends to a $B$-root subgroup on~$G/S$ corresponding to the derivation $f_{\mu-\alpha}\delta_\alpha$ on~$X_{\mathcal D}$.
\end{proof}

\begin{corollary}
Suppose $X$ is an arbitrary horospherical $G$-variety containing $G/S$ as an open $G$-orbit, $\alpha \in \Omega$, $\mu \in \alpha + M$, and $\langle \beta^\vee, \mu \rangle \ge c'(\alpha,\beta)+1$ for all $\beta \in \Pi$.
Then the vector field $f_{\mu - \alpha} \varepsilon_\alpha$ on $X_{\mathcal D}$ extends to an $\GG_a$-integrable vector field on the whole~$X$.
\end{corollary}

\section{The horospherical case for a group of semisimple rank one}
\label{sec_ssrank1}

Throughout this subsection we assume that $G = \SL_2 \times S$ where $S$ is a torus.
Let $\alpha \in \Delta^+$ be the unique positive root and let $\varpi = \alpha/2$ be the corresponding fundamental weight.
Let $X$ be a horospherical $G$-variety with colored fan ${}^c \mathfrak F(X)$ and assume that the subgroup $\SL_2 \subset G$ acts nontrivially on~$X$.
Then the open $G$-orbit in $X$ is isomorphic to $G/H$ where $U^- \subset H \subset B^-$.
We also recall the open subset $X_\mathcal D$ and the canonical section~$Z \subset X_\mathcal D$.

\subsection{Vertical \texorpdfstring{$B$}{B}-root subgroups}

\begin{proposition}
Given $\mu \in \mathfrak X(T)$, there exists a nonzero $B$-semiinvariant vector field on~$X$ of weight~$\mu$ if and only if $\mu \in \alpha + \Gamma_Z$ and $\langle \alpha^\vee, \mu \rangle \ge 1$.
\end{proposition}

\begin{proof}
Recall from Theorem~\ref{thm_LNDs_on_X0} that all vertical $B$-semiinvariant vector fields on $X_\mathcal D$ have the form $f_{\mu - \alpha} \varepsilon_\alpha$ for some $\mu \in \alpha + \Gamma_Z$.
By Proposition~\ref{prop_G/U-}, $f_{\mu - \alpha} \varepsilon_\alpha$ extends to a vector field on the whole~$X$ if and only if $\langle \alpha^\vee, \mu \rangle \ge 1$.
\end{proof}

\begin{corollary} \label{crl_vert_aff_compl}
Suppose $X$ is either affine or complete and $\mu \in \mathfrak X(T)$.
Then there exists a vertical $B$-root subgroup on~$X$ of weight $\mu$ if and only if $\mu \in \alpha + \Gamma_Z$ and $\langle \alpha^\vee, \mu \rangle \ge 1$.
\end{corollary}

\begin{remark}
When $X$ is affine, Corollary~\ref{crl_vert_aff_compl} can be proved directly as follows.
Recall that the open $G$-orbit in~$X$ is isomorphic to $G / H$ with $U^- \subset H \subset B^-$, so
\[
\KK[X] \subset \KK[G/H] \subset \KK[G/U^-] \simeq \KK[S] \otimes \KK[x,y].
\]
We may assume that the $T$-weights of the functions $x$ and~$y$ are $\varpi$ and $-\varpi$, respectively.
The vector field $\varepsilon_\alpha$ from Theorem~\ref{thm_LNDs_on_X0} is induced by the action of~$U$, and so the LND corresponding to $\varepsilon_\alpha$ is $\partial_\alpha = x \frac{\partial}{\partial y}$.
Recall from Theorem~\ref{thm_LNDs_on_X0} that, for $\mu \in \mathfrak X(T)$, the derivation $f_{\mu-\alpha} \partial_\alpha$ preserves $\KK[X_{\mathcal D}]$ if and only if $\mu \in \alpha + \Gamma_Z$.
Since $\SL_2 \subset G$ acts nontrivially on~$X$, there is a simple $G$-submodule $V \subset \KK[X]$ with $\dim V \ge 2$.
Put $k = \langle \alpha^\vee, \mu \rangle$.
Then the derivation $f_{\mu-\alpha} \partial_\alpha$ can be expressed as $f_{\mu - k\varpi} x^{k-1}\frac{\partial}{\partial y}$ where $f_{\mu - k\varpi}$ is an invertible function in~$\KK[S]$.
If $k \le 0$, then $f_{\mu - \alpha} \partial_\alpha(V) \not\subset \KK[G/U^-]$, hence this derivation does not preserve~$\KK[X]$.
On the other hand, if $k \ge 1$, then $f_{\mu-\alpha} \partial_\alpha$ sends each $G$-submodule $\KK[X]_\lambda \subset \KK[X]$ to $\KK[X]_{\lambda+\mu-\alpha}$.
The conditions $\mu \in \alpha + \Gamma_Z$ and $k \ge 1$ imply $\mu - \alpha \in \Gamma$, hence $f_{\mu - \alpha} \partial_\alpha$ preserves $\KK[X]$.

\end{remark}

\begin{remark}
For affine $X$, Corollary~\ref{crl_vert_aff_compl} implies a complete description of all $B$-root subgroups on~$X$, which solves in the horospherical case the problem raised in~\cite[Remark~7.14]{AA}.
Indeed, recall from Proposition~\ref{prop_hor_aff_crit} that $X$ is simple with colored cone of the form $(\mathcal C, \mathcal D)$.
If $\mu$ is the weight of a horizontal $B$-root subgroup on~$X$, then $\mu \in \mathfrak R(\mathcal C) \cap \Lambda^+$ by Proposition~\ref{prop_pairing_with_mu}.
We already know from~\S\,\ref{ssec_stand_aff} that for every $\mu \in \mathfrak R(\mathcal C) \cap \Lambda^+$ there exists a standard horizontal $B$-root subgroup of weight~$\mu$ on~$X$.
On the other hand, it follows from Theorem~\ref{thm_LNDs_on_X0} that for every $\mu \in \mathfrak X(T)$ the space of $B$-semiinvariant vertical vector fields on~$X$ of weight~$\mu$ is at most one-dimensional, and a complete description of them is given by Corollary~\ref{crl_vert_aff_compl}.
\end{remark}

\subsection{The Lie algebra of the connected automorphism group in the complete case}

In this subsection we assume that $X$ is complete with colored fan ${}^c \mathfrak F(X)$.

\begin{proposition}
Suppose $\mu \in \alpha + \Gamma_Z$ and there exists a vertical $B$-root subgroup $H$ on~$X$ of weight~$\mu$.
Then either $\mu = \alpha$ or $\langle \alpha^\vee, \mu \rangle = 1$.
\end{proposition}

\begin{proof}
If $\langle \alpha^\vee, \mu \rangle = 0$, then $\mu \in \mathfrak X(G)$ and thus $H$ is a $G$-root subgroup by Proposition~\ref{prop_weight_is_dom}, hence $H$ is horizontal by Proposition~\ref{prop_G-root_subgroups}, a contradiction.
It follows that $\langle \alpha^\vee, \mu \rangle \ge 1$.
If $\langle \alpha^\vee, \mu \rangle \ge 2$, then $\langle \alpha^\vee, \mu - \alpha \rangle \ge 0$.
Since $\mu - \alpha \in \Gamma_Z$, we conclude that $\langle \rho, \mu - \alpha \rangle \ge 0$ for all $\rho \in \mathfrak F^1(X)$.
Since the fan $\mathfrak F(X)$ is complete, it follows that $\mu - \alpha = 0$.
\end{proof}

Recall that in~\S\,\ref{ssec_comm_rel_stand} we computed commutation relations between all simple $G$-modules of vector fields on~$X$ generated by vector fields corresponding to standard $B$-root subgroups.
In the remaining part of this subsection we compute all commutation relations between simple $G$-modules of vector fields on~$X$ involving vertical vector fields.

The vertical vector field $\varepsilon_\alpha$ comes from the action of~$U$.
For every $\mu \in \alpha + \Gamma_Z$, let $\mathcal K_\mu$ denote the simple $G$-module of vector fields on~$X$ with highest weight $\mu$ and highest weight vector $f_{\mu - \alpha} \varepsilon_\alpha$.
Since all simple $G$-modules in the space of vector fields on~$X$ are $G$-stable, they are $U$-stable and hence $\varepsilon_\alpha$-stable.
Let $\varepsilon_{-\alpha}$ denote the vector field corresponding to the action of~$U^-$ on~$X$ and put $\eta_\alpha = [\varepsilon_{\alpha}, \varepsilon_{-\alpha}]$.
Rescaling $\varepsilon_{-\alpha}$ if necessary we may assume that $\varepsilon_\alpha, \eta_\alpha, \varepsilon_{-\alpha}$ form an $\mathfrak{sl}_2$-triple.

\begin{proposition} \label{prop_comm_vert}
Suppose $\mu_1,\mu_2 \in \alpha + \Gamma_Z$ are such that $\langle \alpha^\vee, \mu_1 \rangle = \langle \alpha^\vee, \mu_2 \rangle = 1$.
Then $[\mathcal K_{\mu_1}, \mathcal K_{\mu_2}] = 0$.
\end{proposition}

\begin{proof}
Put $f_1 = f_{\mu_1-\alpha}$, $f_2 = f_{\mu_2 - \alpha}$ for short.
It suffices to check that $[[\varepsilon_{-\alpha}, f_1\varepsilon_\alpha], f_2 \varepsilon_\alpha]$ = 0.
Indeed,
\begin{multline*}
[[\varepsilon_{-\alpha}, f_1\varepsilon_\alpha], f_2 \varepsilon_\alpha] = (\varepsilon_{-\alpha}(f_1 \varepsilon_\alpha) - (f_1 \varepsilon_\alpha) \varepsilon_{-\alpha}) f_2 \varepsilon_\alpha -  f_2 \varepsilon_\alpha (\varepsilon_{-\alpha}(f_1 \varepsilon_\alpha) - (f_1 \varepsilon_\alpha)\varepsilon_{-\alpha}) = 
\\
((\varepsilon_{-\alpha} f_1) \varepsilon_\alpha - f_1 \eta_\alpha) f_2 \varepsilon_\alpha -  f_2 \varepsilon_\alpha ((\varepsilon_{-\alpha} f_1) \varepsilon_\alpha - f_1 \eta_\alpha) =\\
(\varepsilon_{-\alpha} f_1) \varepsilon_\alpha f_2 \varepsilon_\alpha - f_1 \eta_\alpha f_2 \varepsilon_\alpha -  f_2 \varepsilon_\alpha (\varepsilon_{-\alpha} f_1) \varepsilon_\alpha + f_2 \varepsilon_\alpha f_1 \eta_\alpha =\\
(\varepsilon_{-\alpha} f_1) f_2 \varepsilon_\alpha \varepsilon_\alpha - f_1 (\eta_\alpha f_2) \varepsilon_\alpha - f_1 f_2 \eta_\alpha  \varepsilon_\alpha -  f_2 (\varepsilon_\alpha \varepsilon_{-\alpha} f_1) \varepsilon_\alpha -  f_2 (\varepsilon_{-\alpha} f_1) \varepsilon_\alpha \varepsilon_\alpha + f_2 f_1 \varepsilon_\alpha \eta_\alpha =\\
- f_1 (\eta_\alpha f_2) \varepsilon_\alpha - f_1 f_2 \eta_\alpha  \varepsilon_\alpha -  f_2 (\eta_{\alpha} f_1) \varepsilon_\alpha +  f_1 f_2 \varepsilon_\alpha \eta_\alpha =\\
f_1f_2 \varepsilon_\alpha - f_1f_2 [\eta_\alpha, \varepsilon_\alpha] + f_2f_1 \varepsilon_\alpha = 0.
\end{multline*}

We have used the following relations: $[\varepsilon_\alpha, \varepsilon_{-\alpha}] = \eta_\alpha$, $\varepsilon_\alpha(f_i) = 0$, $\eta_\alpha(f_i) = \langle \alpha^\vee, \mu_i - \alpha \rangle f_i = -f_i$, $[\eta_\alpha, \varepsilon_\alpha] = 2\varepsilon_\alpha$.
\end{proof}

Now assume $\mu_1 \in \mathfrak R(\mathfrak F(X))$, $\mu_2 \in \alpha + \Gamma_Z$, and $\langle \alpha^\vee, \mu_2 \rangle = 1$.
We shall find the commutation relations between $\mathcal L_{\mu_1}$ and $\mathcal K_{\mu_2}$.
Let $\rho \in \mathfrak F^1(X)$ be the element such that $\langle \rho, \mu_1 \rangle = -1$.
Passing to an open subset of~$X$, we may assume that $\mathfrak F(X)$ consists of two cones $\lbrace 0 \rbrace$ and $\QQ_{\ge0} \rho$.
Then $X$ is simple, and we apply the construction of \S\,\ref{subsec_simple_horosph}.
Let $\widetilde {\mathcal L}_{\mu_1}, \widetilde {\mathcal K}_{\mu_2} \in \Der(\KK[\widetilde X])$ be the $\widetilde G$-modules generated by the derivations $\partial_1, \partial_2$, respectively, where $\partial_1 = \partial_{\mu_1}$ and $\partial_2 = f_{\mu_2-\alpha} \varepsilon_\alpha$.

\begin{proposition} \label{prop_comm_rel_2}
The following assertions hold.
\begin{enumerate}[label=\textup{(\alph*)},ref=\textup{\alph*}]
\item \label{prop_comm_rel_2_a}
If $\langle \rho, \mu_2 - \alpha \rangle = 0$ and $\langle \alpha^\vee, \mu_1 \rangle = 0$, then $[\mathcal L_{\mu_1}, \mathcal K_{\mu_2}] = 0$.

\item \label{prop_comm_rel_2_b}
If $\langle \rho, \mu_2 - \alpha \rangle = 0$ and $\langle \alpha^\vee, \mu_1 \rangle \ge 1$, then $\mu_1 + \mu_2 - \alpha  \in \mathfrak R_{\rho}(\mathfrak F(X))$, $[\mathcal L_{\mu_1}, \mathcal K_{\mu_2}] = \mathcal L_{\mu_1+\mu_2-\alpha}$, and $[[\varepsilon_{-\alpha}, \xi_{\mu_1}], f_{\mu_2-\alpha}\varepsilon_\alpha] = - \langle \alpha^\vee, \mu_1 \rangle \xi_{\mu_1+\mu_2-\alpha}$.

\item \label{prop_comm_rel_2_c}
If $\langle \rho, \mu_2 - \alpha \rangle \ge 1$ and $\langle \alpha^\vee, \mu_1 \rangle = 0$, then $[\mathcal L_{\mu_1}, \mathcal K_{\mu_2}] = \mathcal K_{\mu_1+\mu_2}$ and $[\xi_{\mu_1}, f_{\mu_2-\alpha}\varepsilon_\alpha] = \langle \rho, \mu_2 - \alpha \rangle f_{\mu_1+\mu_2-\alpha} \varepsilon_\alpha$.

\item \label{prop_comm_rel_2_d}
If $\langle \rho, \mu_2 - \alpha \rangle \ge 1$ and $\langle \alpha^\vee, \mu_1 \rangle \ge 1$, then $\mu_1 + \mu_2 = \alpha$, $\langle \rho, \mu_2 - \alpha \rangle = \langle \alpha^\vee, \mu_1 \rangle = 1$, $[\mathcal L_{\mu_1}, \mathcal K_{\mu_2}] = \mathcal K_{\alpha} \oplus \KK\nu_{\sigma}$ where $\sigma = \alpha^\vee - 2\rho$, and $[[\varepsilon_{-\alpha}, \xi_{\mu_1}], f_{\mu_2-\alpha}\varepsilon_\alpha] = \frac12 \eta_\alpha - \frac12 \nu_{\sigma}$.
\end{enumerate}
\end{proposition}

\begin{proof}
It suffices to prove the corresponding relations for the derivations $\widetilde {\mathcal L}_{\mu_1}$ and $\widetilde {\mathcal K}_{\mu_2}$.
Take an arbitrary derivation $\partial \in \widetilde{\mathcal L}_{\mu_1}$ and let $f_1 \in \KK[\widetilde O]_{\mu_1}$ be the corresponding function.
Put also $f_2 = f_{\mu_2 -\alpha}$ for short.
For every $\lambda \in \widehat \Gamma$ and $g \in \KK[\widehat X]_\lambda$ one has
\begin{multline} \label{eqn_comm_rel_vert}
[\partial, f_2 e_\alpha](g) = \partial (f_2 e_\alpha)(g) - (f_2 e_\alpha)\partial(g) =\\
(\partial f_2) (e_\alpha g) + f_2 \partial(e_\alpha g) - \langle \rho, \lambda \rangle f_2 e_\alpha (f_1 g) =\\
\langle \rho, \mu_2 - \alpha \rangle f_1f_2 (e_\alpha g) + \langle \rho, \lambda \rangle f_1f_2 (e_\alpha g) - \langle \rho, \lambda \rangle f_2 (e_\alpha f_1) g - \langle \rho, \lambda \rangle f_1f_2 (e_\alpha g) =\\
\langle \rho, \mu_2 - \alpha \rangle f_1f_2 (e_\alpha g) - \langle \rho, \lambda \rangle f_2 (e_\alpha f_1) g.
\end{multline}

If $\langle \rho, \mu_2 - \alpha \rangle = 0$ and $\langle \alpha^\vee, \mu_1 \rangle = 0$,  then $\dim \widetilde{\mathcal L}_{\mu_1} =1$, hence $f_1$ is proportional to~$f_{\mu_1}$, hence $e_\alpha f_1 = 0$ and $[\partial, f_2 e_\alpha](g) = 0$, whence~(\ref{prop_comm_rel_2_a}).

If $\langle \rho, \mu_2 - \alpha \rangle = 0$ and $\langle \alpha^\vee, \mu_1 \rangle \ge 1$, then $\mu_1 + \mu_2 - \alpha  \in \mathfrak R_{\rho}(\mathfrak F(X))$ and $[\partial, f_2 e_\alpha](g) = - \langle \rho, \lambda \rangle f_2 (e_\alpha f_1)g$, so that $[\partial, f_2 e_\alpha] \in \widetilde{\mathcal L}_{\mu_1+\mu_2 - \alpha}$.
If $f_1 = e_{-\alpha}f_{\mu_1}$, then $[\partial,f_2e_\alpha]$ corresponds to the function $-\langle \alpha^\vee, \mu_1 \rangle f_{\mu_1+\mu_2 - \alpha} \in \KK[\widehat X]_{\mu_1+\mu_2-\alpha}$, whence~(\ref{prop_comm_rel_2_b}).

If $\langle \rho, \mu_2 - \alpha \rangle \ge 1$ and $\langle \alpha^\vee, \mu_1 \rangle = 0$, then again $\dim \widetilde{\mathcal L}_{\mu_1} =1$, hence $f_1$ is proportional to~$f_{\mu_1}$, hence $e_\alpha f_1 = 0$ and $[\partial, f_2 e_\alpha] = \langle \rho, \mu_2 - \alpha \rangle f_1f_2 e_\alpha \in \widetilde{\mathcal K}_{\mu_1+\mu_2}$.
If $f_1 = f_{\mu_1}$, then $\langle \rho, \mu_2 - \alpha \rangle f_{\mu_1 + \mu_2 - \alpha} e_\alpha$, whence~(\ref{prop_comm_rel_2_c}).

If $\langle \rho, \mu_2 - \alpha \rangle \ge 1$ and $\langle \alpha^\vee, \mu_1 \rangle \ge 1$, then $\langle \rho', \mu_1+\mu_2 - \alpha \rangle \ge 0$ for all $\rho' \in \mathfrak F^1(X)$, whence by the completeness of the fan we get $\mu_1 + \mu_2 - \alpha = 0$, which implies $\langle \rho, \mu_2 - \alpha \rangle = \langle \alpha^\vee, \mu_1 \rangle = 1$.
If $f_1 = f_{\mu_1}$, then from~(\ref{eqn_comm_rel_vert}) we obtain $[\partial, f_2e_\alpha](g) = e_\alpha g$, which implies $[\mathcal L_{\mu_1}, \mathcal K_{\mu_2}] \supset \mathcal K_\alpha$.
For $f_1 = e_{-\alpha}f_{\mu_1}$ formula~(\ref{eqn_comm_rel_vert}) yields
\[
[\partial, f_2e_\alpha](g) = (e_{-\alpha}f_1)f_2 (e_\alpha g) - \langle \rho, \lambda \rangle g = (e_{-\alpha}f_1)f_2 (e_\alpha g) - \langle \rho, \lambda \rangle g + \frac12 h_\alpha g - \frac12 h_\alpha g.
\]
One can check that $(e_{-\alpha}f_1)f_2 (e_\alpha g) - \langle \rho, \lambda \rangle g + \frac12 h_\alpha g = \langle \frac12 \alpha^\vee - \rho, \lambda \rangle g = \frac12 \langle \sigma, \lambda \rangle g$, which implies $[\mathcal L_{\mu_1}, \mathcal K_{\mu_2}] = \mathcal K_{\alpha} \oplus \KK\nu_{\sigma}$ and $[[\varepsilon_{-\alpha}, \xi_{\mu_1}], f_{\mu_2-\alpha}\varepsilon_\alpha] = \frac12 \eta_\alpha - \frac12 \nu_{\sigma}$ and completes the proof of~(\ref{prop_comm_rel_2_d}).
\end{proof}

Propositions~\ref{prop_comm_rel_stand}, \ref{prop_comm_rel_compl}, \ref{prop_comm_vert}, and~\ref{prop_comm_rel_2} yield the commutation relations between all summands in decomposition~(\ref{eqn_aut_group}) of the Lie algebra $\mathfrak a$ of the connected automorphism group $A$ of~$X$, which uniquely determines the Lie algebra structure on~$\mathfrak a$.

\appendix

\section{Some formulas for matrix exponentials}
\label{sec_matrix_exponentials}

In this appendix we present several matrix identities, which are used in the proof of Proposition~\ref{prop_ext_to_G}(\ref{prop_ext_to_G_b}). (Similar identities without matrix realizations may be also obtained by using the Chevalley commutator formulas; see~\cite[\S\,9]{VaP}).
We realize the Lie algebras of type $\mathsf A_2$, $\mathsf B_2$ ($= \mathsf C_2$), $\mathsf G_2$ as the sets of all matrices of the form
\begin{gather*}
\begin{pmatrix}
t_1 & x_{10} & x_{11} \\
y_{10} & t_2 & x_{01} \\
y_{11} & y_{01} & -t_1-t_2
\end{pmatrix}, \quad
\begin{pmatrix}
t_1 & x_{10} & x_{11} & x_{21} \\
y_{10} & t_2 & x_{01} & x_{11} \\
y_{11} & y_{01} & -t_2 & -x_{10} \\
y_{21} & y_{11} & -y_{10} & -t_1
\end{pmatrix}, \\
\begin{pmatrix}
t_1+t_2 & x_{10} & x_{11} & \sqrt2 x_{21} & x_{31} & x_{32} & 0 \\
y_{10} & t_1 & x_{01} & -\sqrt2 x_{11} & x_{21} & 0 & -x_{32} \\
y_{11} & y_{01} & t_2 & \sqrt2 x_{10} & 0 & -x_{21} & -x_{31} \\
\sqrt2 y_{21} & -\sqrt2 y_{11} & \sqrt2 y_{10} & 0 & -\sqrt2 x_{10} & \sqrt2 x_{11} & -\sqrt2 x_{21} \\
y_{31} & y_{21} & 0 & -\sqrt2 y_{10} & -t_2 & -x_{01} & -x_{11} \\
y_{32} & 0 & -y_{21} & \sqrt2 y_{11} & -y_{01} & -t_1 & -x_{10} \\
0 & -y_{32} & -y_{31} & -\sqrt2 y_{21} & -y_{11} & -y_{10} & -t_1-t_2
\end{pmatrix},
\end{gather*}
respectively.
(The given realization of $\mathsf G_2$ is taken from~\cite[Appendix~A]{AP}.)
If $\mathfrak g$ is one of the above Lie algebras, then the set of all upper-triangular (and also the set of all lower-triangular) matrices in~$\mathfrak g$ is a Borel subalgebra and the set of all diagonal matrices in~$\mathfrak g$ is a Cartan subalgebra.
The two simple roots of~$\mathfrak g$ are denoted by $\alpha_1, \alpha_2$ (in the second and third cases $\alpha_1$ is short and $\alpha_2$ is long).
For every positive root $i\alpha_1 + j\alpha_2$, the corresponding root vector $e_{i\alpha_1+j\alpha_2} \in \mathfrak g$ is defined by $x_{ij} = 1$ and all the other coordinates being zero.
We also denote $u_{ij}(t) = \exp(te_{i\alpha + j\beta})$ where in each concrete case $\alpha$ and $\beta$ are specified in the proof of Proposition~\ref{prop_ext_to_G}(\ref{prop_ext_to_G_b}).

Type $\mathsf A_2$:

Case~\ref{no4} in Table~\ref{table_ab}.
\begin{equation} \label{eqn_no4}
u_{10}(x_{10})u_{11}(x_{11})u_{01}(x_{01})u_{10}(-s) = u_{10}(x_{10}-s)u_{11}(x_{11}+x_{01}s)u_{01}(x_{01})
\end{equation}

Type $\mathsf B_2$:

Case~\ref{no6} in Table~\ref{table_ab}.
\begin{equation} \label{eqn_no6}
u_{10}(x_{10})u_{11}(x_{11})u_{01}(x_{01})u_{10}(-s) = u_{10}(x_{10}-s)u_{11}(x_{11}-2x_{01}s)u_{01}(x_{01})
\end{equation}

Case~\ref{no7} in Table~\ref{table_ab}.
\begin{multline} \label{eqn_no7}
u_{10}(x_{10})u_{11}(x_{11})u_{12}(x_{12})u_{01}(x_{01})u_{10}(-s) =\\
u_{10}(x_{10}-s)u_{11}(x_{11}-x_{01}s)u_{12}(x_{12}-x_{01}^2s)u_{01}(x_{01})
\end{multline}

Case~\ref{no9} in Table~\ref{table_ab}.
\begin{multline} \label{eqn_no9}
u_{10}(x_{10})u_{11}(x_{11})u_{21}(x_{21})u_{01}(x_{01})u_{10}(-s) =\\
u_{10}(x_{10}-s)u_{11}(x_{11}+x_{01}s)u_{21}(x_{21}+2x_{11}s+x_{01}s^2)u_{01}(x_{01})
\end{multline}

Type $\mathsf G_2$:

Case~\ref{no11} in Table~\ref{table_ab}.
\begin{equation} \label{eqn_no11}
u_{10}(x_{10})u_{11}(x_{11})u_{01}(x_{01})u_{10}(-s) = u_{10}(x_{10}-s)u_{11}(x_{11}-3x_{01}s)u_{01}(x_{01})
\end{equation}

Case~\ref{no12} in Table~\ref{table_ab}.
\begin{multline} \label{eqn_no12}
u_{10}(x_{10})u_{11}(x_{11})u_{12}(x_{12})u_{21}(x_{21})u_{01}(x_{01})u_{10}(-s)=\\
u_{10}(x_{10}-s)u_{11}(x_{11}+2x_{01}s)u_{12}(x_{12}+3x_{01}^2s)u_{21}(x_{21}-3x_{11}s-3x_{01}s^2)u_{01}(x_{01})
\end{multline}

Case~\ref{no13} in Table~\ref{table_ab}.
\begin{multline} \label{eqn_no13}
u_{10}(x_{10})u_{11}(x_{11})u_{12}(x_{12})u_{13}(x_{13})u_{23}(x_{23})u_{01}(x_{01})u_{10}(-s) =\\
u_{10}(x_{10}-s)u_{11}(x_{11}-x_{01}s)u_{12}(x_{12}+x_{01}^2s)u_{13}(x_{13}+x_{01}^3s) \times \\
\times u_{23}(x_{23}+(x_{13}-3x_{01}x_{12})s-x_{01}^3s^2)u_{01}(x_{01})
\end{multline}

Case~\ref{no15} in Table~\ref{table_ab}.
\begin{multline} \label{eqn_no15}
u_{10}(x_{10})u_{11}(x_{11})u_{21}(x_{21})u_{31}(x_{31})u_{32}(x_{32})u_{01}(x_{01})u_{10}(-s) = \\
u_{10}(x_{10}{-}s)u_{11}(x_{11}{+}x_{01}s)u_{21}(x_{21}{-}2x_{11}s{-}x_{01}s^2)u_{31}(x_{31}{+}3x_{21}s{-}3x_{11}s^2{-}x_{01}s^3)\times \\\times u_{32}(x_{32}{+}(-3x_{11}^2{+}3x_{01}x_{21})s{-}6x_{01}x_{11}s^2{-}2x_{01}^2s^3)u_{01}(x_{01})
\end{multline}


\end{document}